\documentclass[11pt]{amsart}
\usepackage[utf8]{inputenc}
\usepackage[T1]{fontenc}
\usepackage[frenchb,english]{babel} 
\usepackage{textcomp}
\usepackage{amsmath,amssymb}
\usepackage{amsthm}
                  
\usepackage{lmodern}

\usepackage[a4paper,margin=2cm]{geometry}

\usepackage{graphicx}             
\usepackage{xcolor}               
\usepackage{microtype,stmaryrd}          

\usepackage{hyperref}
\hypersetup{pdfstartview=XYZ}

\usepackage{bbm}
\usepackage{mathrsfs}
\usepackage{subfig}

\title{Excursion theory for Brownian motion indexed by the Brownian tree}

\author{C\'eline Abraham, Jean-Fran\c cois Le Gall}
\date{}
\address{D\'epartement de math\'ematiques, Universit\'e Paris-Sud, 91405 ORSAY Cedex, FRANCE}
\email{cline.abraham@gmail.com, jean-francois.legall@math.u-psud.fr}


\def\build#1_#2^#3{\mathrel{
\mathop{\kern 0pt#1}\limits_{#2}^{#3}}}

\newtheorem{theorem}{Theorem}
\newtheorem{proposition}[theorem]{Proposition}
\newtheorem{definition}[theorem]{Definition}
\newtheorem{lemma}[theorem]{Lemma}

\newtheorem{corollary}[theorem]{Corollary}

\def\w{\mathrm{w}}
\def\t{\mathcal{T}}
\def\f{\mathcal{F}}

\def\W{\mathcal{W}}
\def\S{\mathcal{S}}
\def\T{\mathbb{T}}
\def\N{\mathbb{N}}
\def\M{\mathbb{M}}
\def\Q{\mathbb{Q}}
\def\D{\mathbb{D}}
\def\P{\mathbb{P}}
\def\E{\mathbb{E}}
\def\R{\mathbb{R}}
\def\z{\mathcal{Z}}
\def\n{\mathcal{N}}
\def\ve{{\varepsilon}}
\def\la{\longrightarrow}
\def\ov{\overline}
\begin{document}

\begin{abstract}
We develop an excursion theory for Brownian motion indexed by the Brownian tree, which 
in many respects is analogous to the classical It\^o theory for linear Brownian motion. 
Each excursion is associated with a connected component of the complement of the zero set of 
the tree-indexed Brownian motion. Each such connected
component is itself a continuous tree, and we introduce a quantity measuring the length of its boundary. 
The collection of boundary lengths coincides with the collection of jumps of a continuous-state branching process
with branching mechanism $\psi(u)=\sqrt{8/3}\,u^{3/2}$. Furthermore, conditionally on the boundary lengths, the different excursions
are independent, and we determine their conditional distribution in terms of an excursion measure $\M_0$ which is the analog
of the It\^o measure of Brownian excursions. We provide various descriptions of the excursion measure $\M_0$, and we also
determine several explicit distributions, such as the joint distribution of the boundary length and the mass of an
excursion under $\M_0$. We use the Brownian snake as a convenient tool for defining and analysing the
excursions of our tree-indexed Brownian motion. 
\end{abstract}

\maketitle

\section{Introduction}

The concept of Brownian motion indexed by a Brownian tree has appeared in various settings in the
last 25 years. The Brownian tree of interest here is the so-called CRT (Brownian Continuum Random Tree)
introduced by Aldous \cite{Al1,Al2}, or more conveniently a scaled version of the CRT with a
random ``total mass''. The CRT is a universal model for a continuous random tree, in the sense
that it appears as the scaling limit of many different classes of discrete random trees
(see in particular \cite{Al2,HM,Stu}), and of other
discrete random structures (see the recent papers \cite{CHK,PSW}). At least informally, the meaning of
Brownian motion indexed by the Brownian tree should be clear: Labels, also called spatial positions, are assigned
to the vertices of the tree, in such a way that the root has label $0$ and labels evolve like linear Brownian motion 
when moving away from the root along a geodesic segment of the tree, and of course the increments of
the labels along  disjoint segments are independent.
Combining the branching structure of the CRT with Brownian displacements led Aldous to introduce the
Integrated Super-Brownian Excursion or ISE \cite{Al3}, which is closely related with the canonical measures
of super-Brownian motion. On the other hand, the desire to get a better understanding of the historical paths of superprocesses
motivated  the definition of the so-called Brownian snake \cite{snake}, 
which is a Markov process taking values in the space of all finite paths. Roughly speaking, the value of the Brownian snake
at time $s$ is the path recording the spatial positions along the ancestral line of 
the vertex visited at the same time $s$ in the contour exploration of the Brownian tree. One may view the Brownian snake as a convenient representation of 
Brownian motion indexed by the Brownian tree, avoiding the technical difficulty of dealing with 
a random process indexed by a random set. 

The preceding concepts have found many applications. The Brownian snake has proved a powerful
tool in the study of sample path properties of super-Brownian motion and of its connections
with semilinear partial differential equations \cite{snakesolutions,livrevert}. ISE, and more generally Brownian motion indexed by the Brownian tree and its variants, also appear
in the scaling limits of various models of statistical mechanics above the critical dimension, 
including lattice trees \cite{DS}, percolation \cite{HaS} or oriented percolation \cite{HoS}. More recently, scaling 
limits of large random planar maps have been described by the so-called Brownian map \cite{Uniqueness,Mie},
which is constructed as a quotient space of the CRT for an equivalence relation defined
in terms of Brownian labels assigned to the vertices of the CRT. 

Our main goal in this work is to show that a very satisfactory excursion theory can be developed 
for Brownian motion indexed by the Brownian tree, or equivalently for the Brownian snake, which in many aspects resembles the classical excursion theory
for linear Brownian motion due to It\^o \cite{Ito}. We also expect the associated excursion measure to be an interesting
probabilistic object, which hopefully will have significant applications in related fields. 

Let us give an informal description of the main results of our study. The underlying Brownian tree that we consider is denoted 
by $\t_\zeta$, for the tree coded by a Brownian excursion $(\zeta_s)_{s\geq 0}$
under the classical It\^o excursion measure (see Section \ref{coding} for more
details about this coding, and note that the It\^o excursion measure is a $\sigma$-finite measure). The tree $\t_\zeta$ may be viewed as a scaled version of the
CRT, for which $(\zeta_s)_{s\geq 0}$ would be a Brownian excursion with duration $1$. This tree is rooted at a particular vertex $\rho$. We write $V_u$ for the 
Brownian label assigned to the vertex $u$ of $\t_\zeta$. As explained above the collection
$(V_u)_{u\in \t_\zeta}$ should be interpreted as Brownian motion indexed by $\t_\zeta$,
starting from $0$ at the root $\rho$.
Similarly as in the case of linear Brownian motion, we may then consider the
connected components of the open set
$$\{u\in \t_\zeta : V_u\not =0\},$$
which we denote by $(\mathcal{C}_i)_{i\in I}$. Of course these connected components are
not intervals as in the classical case, but they are connected subsets of the tree $\t_\zeta$,
and thus subtrees of this tree. One then considers, for each component $\mathcal{C}_i$,
the restriction $(V_u)_{u\in \mathcal{C}_i}$ of the labels to $\mathcal{C}_i$, and  this restriction 
again yields a random process indexed by a continuous random tree, which we call the excursion
${E}_i$. Our main results completely determine the ``law'' of the collection
$(E_i)_{i\in I}$ (we speak about the law of this collection though we are working under 
an infinite measure). A first important ingredient of this description is an infinite excursion measure
$\M_0$, which plays a similar role as the It\^o excursion measure in the classical setting, in the sense
that $\M_0$ describes the distribution of a typical excursion $E_i$ (this is a little informal as 
$\M_0$ is an infinite measure). We can then completely describe  the law of the collection
$(E_i)_{i\in I}$ using the measure $\M_0$ and an independence property
analogous to the classical setting. For this description, we first need to
introduce a quantity $\z_i$, called the exit measure of $E_i$, that measures the size of the boundary of  $\mathcal{C}_i$: Note that in the 
classical setting the boundary of an excursion interval just consists of two points, but here of course the
boundary of $\mathcal{C}_i$ is much more complicated. Furthermore, one can define, for every $z\geq 0$, a
conditional probability measure $\M_0(\cdot \mid \z=z)$ which corresponds to the law of an excursion
conditioned to have boundary size $z$ (this is somehow the analog of the It\^o measure conditioned to have a fixed 
duration in the classical setting). Finally, we introduce a ``local time exit process'' $(\mathcal{X}_t)_{t\geq 0}$ such that,
for every $t>0$, $\mathcal{X}_t$ measures the quantity of vertices $u$ of the tree $\t_\zeta$ with label $0$ and such that
the total accumulated local time at $0$ of the label process along the geodesic segment between $\rho$ and $u$
is equal to $t$. The distribution of $(\mathcal{X}_t)_{t> 0}$ 
is known explicitly and can be interpreted as an excursion measure for
the continuous-state branching process with stable branching mechanism $\psi(\lambda)=\sqrt{8/3}\,\lambda^{3/2}$. 
With all these ingredients at hand, we can complete our description of the distribution of the collection 
of excursions: Excursions $E_i$ are in one-to-one correspondence with jumps of the local time exit process $(\mathcal{X}_t)_{t\geq 0}$,
in such a way that, for every $i\in I$, the boundary length $\z_i$ of $E_i$ is equal to the size $z_i$ of the corresponding jump, and furthermore,
conditionally on the process $(\mathcal{X}_t)_{t\geq 0}$, the excursions $E_i$, $i\in I$ are independent, and, for every
fixed $j$, $E_j$ is distributed according to $\M_0(\cdot\mid \z=z_j)$. There is a striking
analogy with the
classical setting (see e.g. \cite[Chapter XII]{RY}), where excursions of linear Brownian excursion are in one-to-one correspondence with
jumps of the inverse local time process, and the distribution of an excursion corresponding to a jump of size $\ell$
is the It\^o measure conditioned to have duration equal to $\ell$. 

The preceding discussion is somewhat informal, in particular because we did not give a mathematically
precise definition of the excursions $E_i$. It would be possible to view these excursions as 
random elements of the space of all ``spatial trees'' in the terminology of \cite{DLG} (compact $\R$-trees $\t$
equipped with a continuous mapping $\phi:\t\to \R$) but for technical reasons we prefer to use
the Brownian snake approach. We now describe this approach in order to give a more precise formulation
of our results. 
Let $\W$ stand for the set of all finite real paths. Here a finite real path is just a continuous function $\w:[0,\zeta]\to \R$,
where $\zeta=\zeta_{(\w)}\geq 0$ depends on $\w$ and is called the lifetime of $\w$,
and, for every $\w\in \W$, we write $\hat \w=\w(\zeta_{(\w)})$ for the endpoint of $\w$. 
The topology on $\W$ is induced by a distance whose definition is recalled at the beginning of Section \ref{statespace}.
The Brownian snake is the continuous Markov process $(W_s)_{s\geq 0}$ with values in $\W$ whose distribution 
is characterized as follows:
\begin{enumerate}
\item[(i)] The lifetime process $(\zeta_{(W_s)})_{s\geq 0}$ is a reflected Brownian motion on $\R_+$.
\item[(ii)] Conditionally on $(\zeta_{(W_s)})_{s\geq 0}$, $(W_s)_{s\geq 0}$ is time-inhomogeneous Markov,
with transition kernels specified as follows: for $0\leq s<s'$,
\begin{enumerate}
\item[$\bullet$] $W_{s'}(t)=W_s(t)$ for every $0\leq t\leq m(s,s'):=\min\{\zeta_{(W_r)}:s\leq r\leq s'\}$;
\item[$\bullet$] conditionally on $W_s$, $(W_{s'}(m(s,s')+t), 0\leq t\leq \zeta_{(W_{s'})}-m(s,s'))$ is a linear Brownian motion 
started from $W_s(m(s,s'))$, on the time interval $[0,\zeta_{(W_{s'})}-m(s,s')]$.
\end{enumerate}
\end{enumerate}
We will write $\zeta_s=\zeta_{(W_s)}$
to simplify notation. Informally, the value $W_s$ of the Brownian snake at time $s$
is a random path with lifetime $\zeta_s$ evolving like reflected Brownian motion on $\R_+$. When $\zeta_s$ decreases,
the path is erased from its tip, and when $\zeta_s$ increases, the path 
is extended by adding ``little pieces'' of Brownian paths at its tip.

For the sake of simplicity in this introduction, we may and will assume that $(W_s)_{s\geq 0}$
is the canonical process on the space $C(\R_+,\W)$ of all continuous mappings from $\R_+$ into $\W$.
Later, it will be more convenient to defined this process on an adequate canonical space
of ``snake trajectories'', see Section \ref{statespace} below.

The trivial path with initial point $0$ and zero lifetime is a regular recurrent point for the process $(W_s)_{s\geq 0}$, and thus we can introduce the 
associated excursion measure $\N_0$, which is called the Brownian snake excursion measure (from $0$). This is a $\sigma$-finite measure on
the space $C(\R_+,\W)$ -- as mentioned earlier, we will later view $\N_0$ as a measure on the smaller
space of snake trajectories. The measure $\N_0$
can be described via properties analogous to (i) and (ii), with the difference that in (i) the law of
reflecting Brownian motion is replaced by the It\^o measure of positive excursions of linear Brownian motion. 
In particular, under $\N_0$, the tree $\t_\zeta$ coded by $(\zeta_s)_{s\geq 0}$ has 
the distribution prescribed in the informal discussion at the beginning of this introduction -- this distribution is
a $\sigma$-finite measure on the space of trees. Recall that the coding of $\t_\zeta$ involves a canonical projection
$p_\zeta:[0,\sigma]\to \t_\zeta$, where $\sigma=\sup\{s\geq 0:\zeta_s>0\}$
(see \cite[Section 3.2]{LGM} or Section \ref{coding} below). Notice that the definition of $\sigma$, as well as the
definition of the tree $\t_\zeta$, are
relevant under $\N_0$.  Then, the Brownian labels $(V_u)_{u\in \t_\zeta}$ are generated by taking $V_u=\hat W_s$, where
$s\in[0,\sigma]$ is any instant such that $p_\zeta(s)=u$. Furthermore, the whole path $W_s$ records the values of labels
along the geodesic segment from the root $\rho$ to $u$, and we sometimes say that $W_s$ is the historical path of $u$.

From now on, we use the Brownian snake construction and argue under the excursion measure $\N_0$. This construction allows us to give a convenient representation for the excursions $(E_i)_{i\in I}$ discussed
above. We observe that, $\N_0$ a.e., the connected components $(\mathcal{C}_i)_{i\in I}$
of $\{u\in \t_\zeta : V_u\not =0\}$ are in one-to-one correspondence with the (countable) collection $(u_i)_{i\in I}$ of all vertices of $\t_\zeta$ such that
\begin{enumerate}
\item[(a)] $V_u=0$;
\item[(b)] $u$ has a strict descendant $v$ such that labels along the geodesic segment from $u$ to $v$ do not vanish except at $u$. 
\end{enumerate}
The correspondence is made explicit by saying that $\mathcal{C}_i$ consists of all strict descendants $v$ of $u_i$ such that property (b) holds, with $u=u_i$ (it is not hard to verify that, $\N_0$ a.e., no branching point of $\t_\zeta$ can satisfy property (b), and we discard the 
event of zero $\N_0$-measure where this might happen).
Then, for every $i\in I$, there are exactly two times $0<a_i<b_i<\sigma$ such that $p_\zeta(a_i)=p_\zeta(b_i)=u_i$. 
The paths $W_s$ for $s\in[a_i,b_i]$ are the historical paths of the descendants of $u_i$. This leads us to define, for every $s\geq 0$,
a random finite path $W^{(u_i)}_s$, with lifetime $\zeta^{(u_i)}_s=\zeta_{(a_i+s)\wedge b_i}-\zeta_{a_i}$, by setting
$$W^{(u_i)}_s(t)=W_{(a_i+s)\wedge b_i}(\zeta_{a_i}+t)\;,\quad 0\leq t\leq \zeta^{(u_i)}_s.$$
If $0<s<b_i-a_i$, , the path $W^{(u_i)}_s$ starts from $0$ (note that $W^{(u_i)}_s(0)=W_{(a_i+s)\wedge b_i}(\zeta_{a_i})=W_{a_i}(\zeta_{a_i})=V_{u_i}$), and
then stays positive during some time interval $(0,\eta)$, $\eta>0$. Of course if $s=0$ or $s\geq b_i-a_i$, $W^{(u_i)}_s$
is just the trivial path with initial point $0$. 

The endpoints $\hat W^{(u_i)}_s$ of the paths $W^{(u_i)}_s$ correspond to the  labels of all descendants of $u_i$ in $\t_\zeta$. In fact, we are only interested in those descendants of $u_i$ that belong to $\mathcal{C}_i$, and for this reason we introduce the following time change
$$\tilde W^{(u_i)}_s= W^{(u_i)}_{\pi^{(u_i)}_s}$$
where, for every $s\geq 0$,
$$\pi^{(u_i)}_s:=\inf\{r\geq 0:\int_0^r \mathrm{d}t\,\mathbf{1}_{\{\tau^*_0(W^{(u_i)}_t)\geq \zeta^{(u_i)}_t\}}> s\},$$
with the notation $\tau^*_0(\w):=\inf\{t>0:\w(t)=0\}$ for $\w\in\W$. The effect of this time change is to eliminate the paths $W^{(u_i)}_s$
that return to $0$ and survive for a positive amount of time after the return time. 

Then, for every $i\in I$, the collection $(\tilde W^{(u_i)}_s)_{s\geq 0}$, which we view as a random element of
the space $C(\R_+,\W)$,  provides a mathematically precise representation of the excursion $E_i$ -- 
in fact the tree $\mathcal{C}_i$ (or rather its closure in $\t_\zeta$) is just the tree coded by the lifetime process  $(\tilde \zeta^{(u_i)}_s)_{s\geq 0}$ of
 $(\tilde W^{(u_i)}_s)_{s\geq 0}$, and the labels on $\mathcal{C}_i$ correspond in this identification to the endpoints of the paths $\tilde W^{(u_i)}_s$.
 
 In order to state our first theorem, we need one more piece of notation. For every $i\in I$, we let $\ell_i$ be the total local time at $0$
 of the historical path $W_{a_i}$ of $u_i$.
 
 \begin{theorem}
 \label{construct}
 There exists a $\sigma$-finite measure $\M_0$ on $C(\R_+,\W)$ such that, for any nonnegative measurable function 
 $\Phi$ on $\R_+\times C(\R_+,\W)$, we have
 $$\N_0\Bigg( \sum_{i\in I} \Phi(\ell_i,\tilde W^{(u_i)})\Bigg)= \int_0^\infty \mathrm{d}\ell\,\M_0\Big(\Phi(\ell,\cdot)\Big).$$
 \end{theorem} 
 
 The reason for considering a function depending on local times should be clear from the formula of the theorem: if
 $\Phi(\ell,\omega)$ does not depend on $\ell$, the right-hand side will be either $0$ or $\infty$.
 We may write $\M_0$ in the form
 $$\M_0=\frac{1}{2} ( \N^*_0 + \check \N^*_0)$$
 where $\N^*_0$ is supported on positive excursions and $\check \N^*_0$ is the image of $\N^*_0$ under $\omega\mapsto -\omega$. Then,
 for every $\delta>0$, $\N^*_0$
 gives a finite mass to ``excursions'' $\omega$ that hit $\delta$, and more precisely,
 $$\N^*_0(\{\omega: \sup\{\hat W_s(\omega):s\geq 0\} >\delta)= c_0 \delta^{-3}$$
 where $c_0$ is an explicit constant (see Lemma \ref{loi-max}). 
 
 In a way similar to the classical setting, one can give various representations of the measure $\N^*_0$. For $\ve>0$, let
 $\N_\ve$ be the Brownian snake excursion measure from $\ve$ (this is just the image of $\N_0$
 under the shift $\omega\mapsto \ve + \omega$). Consider under $\N_\ve$ the time-changed process $\tilde W$
obtained by removing those paths $W_s$ that hit $0$ and then survive for a positive amount of time (this is analogous
to the time change we used above to define $\tilde W^{(u_i)}$ from $W^{(u_i)}$). Then $\N^*_0$ may
be obtained as the limit when $\ve\to 0$ of $\ve^{-1}$ times the law of $\tilde W$ under $\N_\ve$. See 
Theorem \ref{construct-N} and Corollary \ref{approx-*} for precise statements. This result is analogous
to the classical result saying that the It\^o measure of positive excursions is the limit
(in a suitable sense) of $(2\ve)^{-1}$ times the law of linear Brownian motion started from $\ve$
and stopped upon hitting $0$. 

Similarly, one can give a description of $\N^*_0$ analogous to the well-known Bismut decomposition
for the It\^o measure \cite[Theorem XII.4.7]{RY}. Under $\N^*_0$, pick a vertex of the tree coded by $(\zeta_s)_{s\geq 0}$ according to
the volume measure on this tree, re-root the tree at that vertex and shift all labels so that the label of the new root 
is again $0$. This construction yields a new measure on $C(\R_+,\W)$, which turns out to be
the same (up to a simple density) as the measure obtained by
picking $x\leq 0$ according to
Lebesgue measure on $(-\infty,0)$ and then, under the measure 
$\N_0$ restricted to the event where one of the paths $W_s$ hits $-x$, removing all paths $W_s$ that go below level $x$. See Theorem 
\ref{reenracinement} below for a more precise statement.

We now introduce exit measures under $\M_0$. 

\begin{proposition}
\label{exit-definition}
One can choose a sequence $(\alpha_n)_{n\geq 1}$ of positive reals converging to $0$ so that, $\M_0$ a.e., the limit
$$Z^*_0:=\lim_{n\to\infty} \alpha_n^{-2}\int_0^\infty \mathbf{1}_{\{0<|\hat W_s|<\alpha_n\}}\,\mathrm{d}s$$
exists and defines a positive random variable. Furthermore, this limit does not depend 
on the choice of the sequence $(\alpha_n)_{n\geq 1}$.
\end{proposition}

\noindent{\it Remark}. At this point, a comment about our terminology is in order. Frequently in this article, we will argue on 
$\sigma$-finite measure spaces, and measurable functions defined 
on these spaces will still be called ``random variables'', as in the preceding proposition. Similarly we will speak about the ``law'' or the
``distribution'' of these random variables, though these laws will be infinite (not necessarily $\sigma$-finite)
measures. 

\smallskip
Theorem \ref{construct} and Proposition \ref{exit-definition} allow us to make sense of the quantity $Z^*_0(\tilde W^{(u_i)})$,
for every $i\in I$. 
Informally, $Z^*_0(\tilde W^{(u_i)})$ counts the number of paths $\tilde W^{(u_i)}$ that return to $0$,
and thus measures the size of the boundary of $\mathcal{C}_i$. On the other hand, the 
quantity $\sigma(\tilde W^{(u_i)})$ corresponds to the volume of $\mathcal{C}_i$. Quite remarkably, one can obtain
an explicit formula for the joint distribution of the pair $(Z^*_0,\sigma)$ under $\M_0$. This distribution has density
$$ f(z,s)=\frac{\sqrt{3}}{2 \pi} \sqrt{z} \,s^{-5/2} \exp \left( -\frac{z^2}{2s} \right)$$
with respect to Lebesgue measure on $\R_+\times \R_+$ (Proposition \ref{loisZ0sigma}).

Using scaling arguments, one can then canonically define, for every $z>0$, the conditional probability measure
$\M_0(\cdot\mid Z^*_0=z)$, which will play an important role in our description of the distribution of
the collection $(W^{(u_i)})_{i\in I}$. Before stating our theorem identifying this
distribution, we need a last ingredient. For every $s\geq  0$ and $t\in[0,\zeta_s]$, write 
$L^0_t(W_s)$ for the local time at level $0$ and at time $t$  of the path $W_s$ (this makes sense 
under the measure $\N_0$). We observe that, under the measure $\N_0$, the process
$$\mathbf{W}_s:=(W_s,L^0(W_s))=(W_s(t),L^0_t(W_s))_{0\leq t\leq \zeta_s}$$
can be viewed as the Brownian snake (under its excursion measure from $(0,0)$) associated 
with a spatial motion which is now the pair consisting of a linear Brownian motion and its
local time at $0$ (the 
Brownian snake associated with a Markov process is defined by properties analogous to (i) and (ii) above, with the
only difference that in (ii) linear Brownian motion is replaced by the Markov process in consideration). See \cite{livrevert}, and notice that the spatial motion used to define 
the Brownian snake needs to satisfy certain continuity properties which hold in the present situation.
Following \cite[Chapter V]{livrevert}, we can then define, for every $r>0$, the exit measure of
$\mathbf{W}$ from the open set $O_r=\R\times [0,r)$, and we denote this exit measure 
by $\mathcal{X}_r$ -- to be precise the exit measure is a measure on $\partial O_r$, but here it is easily seen to be
concentrated on the singleton $\{0\}\times\{r\}$, and $\mathcal{X}_r$ denotes its total mass. 
Informally, $\mathcal{X}_r$ measures the quantity of paths $W_s$ whose endpoint is $0$ and which have accumulated
a total local time at $0$ equal to $r$. 

One can explicitly determine the ``law'' of the exit measure process $(\mathcal{X}_r)_{r>0}$
under $\N_0$, 
using on one hand L\'evy's famous theorem relating the law of the local time process of
a linear Brownian motion $B$ to that of the supremum process of $B$, and on the other hand known results 
about exit measures from intervals. This process is Markovian, with the transition mechanism of the
continuous-state branching process with stable branching mechanism $\psi(\lambda)=\sqrt{8/3}\,\lambda^{3/2}$.
In particular the process $(\mathcal{X}_r)_{r>0}$ has a c\`adl\`ag modification, which we
consider from now on.

Recall that, for every $i\in I$, $\ell_i$
denotes the local time at $0$ of the historical path of $u_i$.

\begin{proposition}
\label{identi-jump}
The numbers $\ell_i$, $i\in I$ are exactly the jump times of the process
$(\mathcal{X}_r)_{r>0}$. Furthermore, for every $i\in I$,  the size $Z^*_0(\tilde W^{(u_i)})$ of the boundary of $\mathcal{C}_i$
is equal to the jump $\Delta\mathcal{X}_{\ell_i}$.
\end{proposition}

We can now state the main result of this introduction. 
\begin{theorem}
\label{main-intro}
Under $\N_0$, conditionally on the local time exit process $(\mathcal{X}_r)_{r>0}$, the 
excursions $(\tilde W^{(u_i)})_{i\in I}$ are independent and, for every $j\in I$, the conditional
distribution of $\tilde W^{(u_j)}$ is $\M_0(\cdot \mid Z^*_0=\Delta \mathcal{X}_{\ell_j})$.
\end{theorem}

In the classical theory, the collection of excursions of linear Brownian motion is described in terms
of a Poisson point process. Such a representation is also possible here and the relevant
Poisson point process is linked with the Poisson process of jumps of
the L\'evy process that corresponds to the continuous-state branching process $\mathcal{X}$
via Lamperti's transformation. We refrained 
from explaining this representation in this introduction because the formulation 
is somewhat more intricate than in the classical case (see however Proposition \ref{PoissonPP})
and requires to add extra randomness to get a complete construction of the Poisson point process. 

Let us make a few remarks. First, although we stated our main results under the infinite measure $\N_0$, one
can give equivalent statements in the more familiar setting of probability measures, for instance by 
conditioning $\N_0$ on specific events with finite mass (such as the event where at least one of the paths 
$W_s$ has accumulated a total local time at $0$ greater than $\delta$, for some fixed $\delta>0$) or by 
dealing with a Poisson measure with intensity $\N_0$ -- such Poisson measures are in fact needed 
when one studies the connections between the Brownian snake and superprocesses.
The second remark is that we could have considered excursions away from $a\not =0$ instead
of the particular case $a=0$. There is a minor difference, due to the special connected component
of $\{u\in\t_\zeta: V_u\not = a\}$ that contains the root. The study of the connected components 
other than the special one can be reduced to the case $a=0$ by an application of the so-called special Markov 
property (see Section \ref{exitSMP}). 
As a last and important remark, most of the following proofs and statements deal with excursions  ``above the minimum''
(see Section \ref{sec:construct} for the definition)
and not with the excursions away from $0$ that we considered in this introduction. 
However the results about excursions away from $0$ can then be derived using the already mentioned 
theorem of L\'evy, and we explain this derivation in detail in Section \ref{excupoint}. 
The reason for considering first excursions  above the minimum comes from the fact that
certain technical details become significantly simpler. In particular, the local time exit process
is replaced by the more familiar process of exit measures from intervals.

An important motivation for the present work comes from the construction
of the Brownian map as a quotient space of the CRT for an equivalence relation 
defined in terms of Brownian motion indexed by the CRT (see e.g. \cite[Section 2.5]{Uniqueness}).
The recent paper \cite{Hull} discusses the infinite volume version of the Brownian map
called the Brownian plane. In a way similar to the Brownian map, the Brownian plane is
obtained as a quotient space of an infinite Brownian tree equipped with nonnegative Brownian labels,
in such a way that these labels
correspond to distances from the root in the Brownian plane. The main goal
of \cite{Hull} is to study the process of hulls, where, for every $r>0$, the hull of
radius $r$ is obtained by filling in the bounded holes in the ball of radius $r$ centered at the root vertex of the
Brownian plane.  It turns out (see formula (16) of \cite{Hull}) that discontinuities of the process of hulls correspond
to excursions above the minimum for the process of labels, which is a tree-indexed Brownian motion
under a special conditioning. Such a discontinuity appears 
when the hull of radius $r$ ``swallows'' a connected component of the complement of the ball of radius $r$,
and this connected component consists of (the equivalence classes of) the vertices belonging to the
associated excursion above the minimum at level $r$. This relation explains why several
formulas and calculations below are reminiscent of those in \cite{Hull}. In particular the conditional
distribution of the mass $\sigma$ of an excursion given the boundary length $Z^*_0$ (see Proposition
\ref{loisZ0sigma}) appears in \cite[Theorem 1.3]{Hull}, as well as in the companion paper \cite{peeling},
where this distribution is interpreted as the limiting law of the number of faces of a Boltzmann triangulation
with a boundary of fixed size tending to infinity.

In the same direction, there are close relations between the present article and the recent work of Miller and Sheffield \cite{MS1,MS2,MS3}
aiming at proving the equivalence of the Brownian map and Liouville quantum gravity with parameter $\gamma=\sqrt{8/3}$. 
In particular, the paper \cite{MS1} uses what we call Brownian snake excursions above the minimum to define 
the notion of a Brownian disk, corresponding to bubbles appearing in the exploration of the Brownian map: See 
the definition of $\mu^L_{\rm DISK}$ in
Proposition 4.4, and its proof, in \cite{MS1}. A key idea of \cite{MS1} is the fact that one can use such Brownian disks
to reconstruct the Brownian map by filling in the holes of the so-called ``L\'evy net'', which itself
corresponds to the union of the boundaries of hulls centered at the root (to be precise, the definition of hulls
here requires that there is a marked vertex in addition to the root of the Brownian map). Interestingly,
Bettinelli and Miermont \cite{BM} have developed a different method, based on an approximation by
large planar maps with a boundary, to define the notion of a Brownian disk. The forthcoming paper
\cite{LG17} uses the excursion measure $\N^*_0$ introduced in the present work to unify these different approaches and 
derive new properties of Brownian disks. 

An obvious question is whether the excursion theory developed here can be extended to more general
tree-indexed processes. As a first remark, many of our arguments rely on the special Markov property
(Proposition \ref{SMP} below), which has been stated and proved rigorously only for processes indexed
by the Brownian tree. It is likely that some version of the special Markov property holds for
processes indexed by L\'evy trees \cite{DLG,We}, which are random $\R$-trees characterized by a branching property 
analogous to the one that holds for discrete Galton-Watson trees, but this has not been proven yet.
One may then ask whether Brownian motion can be replaced by another Markov process indexed by the 
Brownian tree. The recent paper \cite{LG16} shows that the special Markov property still holds provided
the underlying Markov process satisfies certain strong continuity assumptions. These assumptions are satisfied by
a ``nice'' diffusion process on the real line, and one may expect that analogs of our results will then hold in
that more general setting. Proving this would however require a different approach, since 
we can no longer use the L\'evy theorem mentioned above.

The present  paper is organized as follows.
Section \ref{sec:preli} below presents a number of preliminary observations. In contrast with the previous lines where we consider the
canonical space $C(\R_+,\W)$, we have chosen to define the measure $\N_0$ on a smaller canonical space, the
space of ``snake trajectories'' (see Section \ref{statespace}). The reason for this choice is that several transformations,
such as the re-rooting operation, or the truncation operation allowing us to eliminate paths $W_s$ hitting a certain level, are
more conveniently defined and analysed on this smaller space. Snake trajectories are in one-to-one correspondence
with tree-like paths (also defined in Section \ref{statespace}) via a homeomorphism theorem
of Marckert and Mokkadem \cite{MM}, and this bijection is useful to simplify certain convergence arguments. Section \ref{exitSMP}
gives a precise statement of the special Markov property which later plays an important role.

Section  \ref{sec:construct} provides a construction of the measure $\N^*_0$, by proving the analog of Theorem \ref{construct} for
excursions above the minimum. As a by-product, this proof also yields the above-mentioned approximation of
$\N^*_0$ in terms of the Brownian snake under $\N_\ve$, truncated at level $0$. 
Section \ref{sec:re-root} gives our analog of the Bismut decomposition theorem for the measure $\N^*_0$. The proof
is based on a re-rooting invariance property of the Brownian snake which can be found in \cite{LGWeill}. 
Then Section \ref{a.s.cons} describes 
an almost sure version of the approximation of Section \ref{sec:construct}, which is useful in further developments. 

Section \ref{sec:exit} contains the definition of the exit measure $Z^*_0$ under $\N^*_0$, and the derivation of the joint 
distribution of the pair $(Z^*_0,\sigma)$. As an important technical ingredient of the proof of our main results, we also verify that
the approximation of the measure $\N^*_0$  by a truncated Brownian snake under $\N_\ve$ can be stated jointly
with the convergence of the corresponding exit measures (Proposition \ref{conv-exit-M}). Section \ref{sec:excu-pro}
contains the proof of the results analogous to Proposition \ref{identi-jump} and Theorem \ref{main-intro} 
in the slightly different setting of excursions above the minimum. In a way very similar
to the classical theory, we introduce an auxiliary Poisson point process with intensity $\mathrm{d}t\otimes \N^*_0(\mathrm{d}\omega)$,
such that all excursions above the minimum can be recovered from the atoms of this process -- but as mentioned earlier the
construction of this Poisson point process is somewhat more delicate than in the classical case. Finally, Section \ref{excupoint} explains how
the results of the present introduction can be derived from those concerning excursions above the minimum.

\medskip
\noindent{\bf Warning.} As already mentioned, we define the Brownian snake below
on a smaller canonical space than $C(\R_+,\W)$, namely on the space $\S$ of all snake trajectories 
introduced in Definition \ref{def:snakepaths}. In particular, $(W_s)_{s\geq 0}$ will be the
canonical process on $\S$, and $\N_0$ and $\N^*_0$ will be viewed as
$\sigma$-finite measures on $\S$ rather than on $C(\R_+,\W)$. The notation used below is
therefore slightly different from the one in the Introduction, but this should create no confusion. 
 
 \medskip
 \noindent{\bf Main notation.}
 {\setlength{\leftmargini}{1em}
 \begin{itemize}
 \item $\t_h$ tree coded by a function $h$ (Section 2.1)
 \item $\prec$ genealogical order on $\t_\zeta$ (Section 2.1)
 \item $p_h$ canonical projection from $\R_+$ onto $\t_h$ (Section 2.1)
 \item $\W$ set of all finite paths, $\W_x$ set of all finite paths started at $x$ (Section 2.2)
 \item $\zeta_{(\w)}$ lifetime of $\w\in \W$ (Section 2.2)
 \item $\hat \w=\w_{\zeta_{(\w)}}$ for $\w\in \W$ (Section 2.2)
 \item $\underline{\w}=\min\{\w(t):0\leq t\leq \zeta_{(\w)}\}$ for $\w\in \W$ (Section 2.2)
  \item $\tau_y(\w)=\inf\{t\in [0,\zeta_{(\w)}]: \w(t)=y\}$, $\tau^*_y(\w)=\inf\{t\in (0,\zeta_{(\w)}]: \w(t)=y\}$ (Section 2.2)
 \item $\S$ set of all snake trajectories, $\S_x$ set of all snake trajectories with initial point $x$ (Section 2.2)
 \item $(W_s)_{s\geq 0}$ canonical process on $\S$ (Section 2.2)
 \item $\zeta_s(\omega)=\zeta_{(W_s(\omega))}$ lifetime process on $\S$ (Section 2.2)
 \item $\sigma(\omega)$ duration of the snake trajectory $\omega\in\S$ (Section 2.2)
 \item $\|\omega\|=\sup\{|\omega_s(t)|:0\leq s\leq \sigma(\omega), 0\leq t\leq \zeta_{(\omega_s)}\}$, for $\omega\in\S$ (Section 2.2)
 \item $\S^{(\delta)}=\{\omega\in \S: \|\omega\|>\delta\}$ (Section 3)
 \item $M(\omega)=\sup\{\omega_s(t):0\leq s\leq \sigma(\omega), 0\leq t\leq \zeta_{(\omega_s)}\}$, for $\omega\in\S$ (Section 2.2)
 \item $\T$ set of all tree-like paths, $\T_x$ set of all tree-like paths with initial point $x$ (Section 2.2)
 \item $\omega\mapsto \kappa_a(\omega)$ shift on snake trajectories (Section 2.2)
 \item $\omega\mapsto R_s(\omega)$ re-rooting on snake trajectories (Section 2.2)
 \item $\mathrm{tr}_y(\omega)$ truncation of $\omega\in \S$ at $y$ (Section 2.2)
 \item $\N_x$ Brownian snake excursion measure from $x$ (Section 2.3)
 \item $W_*$ minimum of the Brownian snake (Section 2.3)
 \item $\mathcal{E}^U_x$ $\sigma$-field generated by the Brownian snake paths under $\N_x$ before they exit $U$ (Section 2.4)
 \item $\z^U$ exit measure of the Brownian snake from $U$ (Section 2.4)
 \item $\z_y=\langle \z^{(y,\infty)},1\rangle$, $Z_a=\z_{-a}$ (Section 2.5)
  \item $V_u=\hat W_s$ if $u=p_\zeta(s)$, label of $u\in \t_\zeta$ (Section 3)
 \item $D$ set of all excursion debuts (Section 3)
 \item $C_u=\{w\in \t_\zeta: u\prec w \hbox{ and } V_v>V_u,\;\forall v\in \rrbracket u,w\llbracket\}$ for $u\in D$ (Section 3)
 \item $M_u=\sup\{V_v-V_u: v\in C_u\}$ height of the excursion debut $u$ (Section 3)
 \item $D_\delta$ set of all excursion debuts with height greater than $\delta$ (Section 3)
\item $W^{(u)}$ snake trajectory describing the labels of descendants of $u\in D$, shifted so that $W^{(u)}\in\S_0$ (Section 3)
\item $\tilde W^{(u)}=\mathrm{tr}_0(W^{(u)})$ truncation of $W^{(u)}$ at $0$, for $u\in D$ (Section 3)
\item $\tilde W=\mathrm{tr}_0(W)$ truncation at $0$ of the canonical process $W$ (Section 3)
\item $\tilde M=M(\tilde W)$ (Section 3)
\item $\N^*_0$ Brownian snake excursion measure ``above the minimum'' (Section 3)
\item $\displaystyle{\mathcal{N}_k^\ve(\omega) = \sum_{i\in I^\ve_k} \delta_{\omega_i^{k,\ve}}}$ point measure of excursions of 
$\omega\in \S_0$ outside $(-k\ve,\infty)$ (Section 3)
\item $\tilde \omega_i^{k,\ve} = \mathrm{tr}_0\circ \kappa_{(k+1)\ve}(\omega_i^{k,\ve})$ truncation at $0$ of the excursion $\omega_i^{k,\ve}$
shifted so that its initial point is $\ve$ (Section 3)
\item $\theta_\lambda$ scaling operator on $\S$ (Section 3)
\item $W^{[s]}(\omega)=\kappa_{-\hat W_s(\omega)}\circ R_s(\omega)$ snake trajectory $\omega$ re-rooted at $s$ and shifted 
so that the spatial position of the root is $0$ (Section 5)
\item $Z^*_0$ exit measure at $0$ under $\N^*_0$ (Section 6)
\item $\N^{*,z}_0=\N^*_0(\cdot \mid Z^*_0=z)$ (Section 6)
\item $\mathcal{Y}_b= \int_0^\sigma \mathrm{d}s\,\mathbf{1}_{\{\tau_{-b}(W_s)=\infty\}}$ for $b\geq 0$ (Section 6)
\item $\N^{(\beta)}_0=\N_0(\cdot \mid W_* <-\beta)$ (Section 7)
 \end{itemize}}

\section{Preliminaries}
\label{sec:preli}

\subsection{Coding a real tree by a function}
\label{coding}

In this subsection, we recall without proof a number of simple properties of the
coding of compact $\R$-trees by functions. We refer to \cite{DLG} and \cite{LGM}
for additional details.

Let $h:\R_+\to\R_+$ be a nonnegative continuous function on $\R_+$ such that $h(0)=0$. We assume that
$h$ has compact support, so that
$$\sigma_h:=\sup\{t\geq 0: h(t)>0\} <\infty.$$
Here and later we make the convention that $\sup\varnothing = 0$. 

For every $s,t\in\R_+$, we set
$$d_h(s,t):=h(s)+h(t)-2\min_{s\wedge t\leq r\leq s\vee t} h(r).$$
Then $d_h$ is a pseudo-distance on $\R_+$. We introduce the associated
 equivalence relation on $\R_+$, defined by setting
$s\sim_h t$ if and only if $d_h(s,t)=0$, or equivalently
$$h(s)=h(t)=\min_{s\wedge t\leq r\leq s\vee t} h(r).$$
Then, $d_h$ induces a distance on the quotient space $\R_+/\sim_h$.

\begin{lemma}
\label{cod-tree}
The quotient space $\t_h:=\R_+/\!\sim_h$ equipped with the distance $d_h$ is a compact $\R$-tree called the tree 
coded by $h$. The canonical projection from $\R_+$ onto $\t_h$
is denoted by $p_h$.
\end{lemma}

See e.g. \cite[Theorem 2.1]{DLG} for a proof of this lemma as well as for the 
definition of $\R$-trees. For every $u,v\in \t_h$, the segment $\llbracket u,v \rrbracket$
is defined as the range of the (unique) geodesic from $u$ to $v$ in $(\t_h,d_h)$. The sets
 $\rrbracket u,v \llbracket$ or $\rrbracket u,v \rrbracket$ are then defined with the
 obvious meaning. 

Write $\rho$ for the equivalent class of $0$ in the quotient $\R_+\,/\!\sim_h$, and note that,
for every $s\geq 0$, $d_h(\rho,p_h(s))=h(s)$. We call 
$\rho$ the root of $\t_h$, and the ancestral line of 
a point $u\in\t_h$ is the geodesic segment $\llbracket \rho,u \rrbracket$. We can then define 
a genealogical relation on $\t_h$ by saying that $u$ is an ancestor of $v$
(or $v$ is a descendant of $u$) if $u$ belongs to $\llbracket \rho,v \rrbracket$. 
We will use the notation $u\prec v$ to mean that $u$ is an ancestor of $v$.
If $s,t\geq 0$, the property $p_h(s)\prec p_h(t)$ holds if and only if
$$h(s)=\min_{s\wedge t\leq r\leq s\vee t}h(r).$$
If $u,v\in \t_h$, the last common ancestor of $u$ and $v$ is the unique point, denoted by $u\wedge v$, such that
$$\llbracket \rho,u \rrbracket\cap\llbracket \rho,v \rrbracket=\llbracket \rho,u\wedge v\rrbracket.$$
If $u=p_h(s)$ and $v=p_h(t)$ then $u\wedge v=p_h(r)$, where $r$ is any time in $[s\wedge t,s\vee t]$
such that $h(r)=\min\{h(r'):r'\in [s\wedge t,s\vee t]\}$.

We call leaf of $\t_h$ any point $u\in \t_h$ which has no descendant other than itself. 
We let 
${\rm Sk}(\t_h)$, the skeleton of $\t_h$, be the set of all 
points of $\t_h$ that are not leaves. The multiplicity of a 
point $u\in\t_h$ is the number of
connected components of $\t_h\backslash\{u\}$. A point
$u\not = \rho$ is a leaf if and only if its multiplicity is $1$.

Suppose in addition that $h$ satisfies the following properties:
\begin{enumerate}
\item[(i)] $h$ does not vanish on $(0,\sigma_h)$; 
\item[(ii)]$h$  is not constant on any nontrivial subinterval
of $(0,\sigma_h)$;
\item[(iii)] the local minima of $h$ on $(0,\sigma_h)$ are distinct.
\end{enumerate}
All these properties hold
in the applications developed below, where $h$ is a Brownian excursion away from $0$. Then the multiplicity of 
any point of $\t_h$ is at most $3$. Furthermore, a point $u$ has multiplicity $3$ if
and only if $u$ is the form $u=p_h(r)$ where $r$ is a time of local minimum of $h$ on $(0,\sigma_h)$. In that case there
are exactly three values of $s$ such that $p_h(s)=u$, namely 
$s=\sup\{t<r:h(t)>h(r\}$, $s=r$ and $s=\inf\{t>r:h(t)\leq h(r)\}$. 
Points of multiplicity $3$ will be called branching points of $\t_h$. 
If $u$ and $v$ are two points of $\t_h$, and if $u\wedge v\not =u$ and $u\wedge v\not =v$, then
$u\wedge v$ is a branching point. Finally, if 
$u$ is a point of ${\rm Sk}(\t_h)$ which is not a branching point, then there are exactly two times $0\leq s_1<s_2\leq \sigma_h$
such that $p_h(s_1)=p_h(s_2)=u$, and the descendants of $u$ are the points $p_h(s)$ when $s$ varies over
$[s_1,s_2]$.

\subsection{Canonical spaces for the Brownian snake}
\label{statespace}

Before we recall the basic facts that we need about the Brownian snake, 
we start by discussing the canonical space on which this random process will be defined
(for technical reasons, we choose a canonical space suitable for the definition of
the Brownian snake excursion measures, which would not be appropriate for the 
Brownian snake starting from an arbitrary initial value as considered above in the Introduction).

Recall the notion of a finite path from the Introduction. We let 
$\W$ denote the space of all finite paths in $\R$, and write
$\zeta_{(\w)}$ for the lifetime of a finite path $\w\in\W$. The set $\W$ is a Polish space when equipped with the
distance
$$d_\W(\w,\w')=|\zeta_{(\w)}-\zeta_{(\w')}|+\sup_{t\geq 0}|\w(t\wedge
\zeta_{(\w)})-\w'(t\wedge\zeta_{(\w')})|.$$
The endpoint or tip of the path $\w$ is denoted by $\hat \w=\w(\zeta_{(\w)})$.
For every $x\in\R$, we set $\W_x=\{\w\in\W:\w(0)=x\}$. The trivial element of $\W_x$ 
with zero lifetime is identified with the point $x$ -- in this way we view $\R$
as the subset of $\W$ consisting of all finite paths with zero lifetime. 
We will also use the notation
$\underline\w=\min\{\w(t):0\leq t\leq \zeta_{(\w)}\}$.

We next turn to snake trajectories.

\begin{definition}
\label{def:snakepaths}
Let $x\in \R$. A snake trajectory with initial point $x$ is a continuous mapping
\begin{align*}
\omega:\ &\R_+\to \W_x\\
&s\mapsto \omega_s
\end{align*}
which satisfies the following two properties:
\begin{enumerate}
\item[\rm(i)] We have $\omega_0=x$ and $\sup\{s\geq 0: \omega_s\not =x\}<\infty$.
\item[\rm(ii)] For every $0\leq s\leq s'$, we have
$$\omega_s(t)=\omega_{s'}(t)\;,\quad\hbox{for every } 0\leq t\leq \min_{s\leq r\leq s'} \zeta_{(\omega_r)}.$$
\end{enumerate} 
We write $\S_x$ for the set of all snake trajectories  with initial point $x$, and 
$$\S:=\bigcup_{x\in \R} \S_x$$
for the set of all snake trajectories. 
\end{definition}

If $\omega\in \S_x$, we write
$$\sigma(\omega)=\sup\{s\geq 0: \omega_s\not =x\}$$
and call $\sigma(\omega)$ the {\it duration} of the snake trajectory $\omega$. 
For $\omega\in \S$, we will also use the notation $W_s(\omega)=\omega_s$ and $\zeta_s(\omega)=\zeta_{(\omega_s)}$
for every $s\geq 0$, so that in particular $(W_s)_{s\geq 0}$ is the canonical process on $\S$.

\smallskip
\noindent{\it Remark}. Property (ii) is called the snake property. It is not hard to verify that, for any mapping $\omega:\R_+\to \W_x$
such that both the lifetime function $s\mapsto \zeta_s(\omega)$ and the tip function $s\mapsto\hat \omega_s=\hat W_s(\omega)$ are
continuous, then the snake property (ii) implies that $\omega$ is continuous. 

\smallskip
The set $\S$ is equipped with the distance
$$d_\S(\omega,\omega')= |\sigma(\omega)-\sigma(\omega')|+ \sup_{s\geq 0} \,d_\W(W_s(\omega),W_{s}(\omega')).$$
Note that $\S$ is a measurable subset of the space $C(\R_+,\W)$, which is equipped as usual with the Borel
$\sigma$-field associated with the topology of uniform convergence on every compact interval. 

We will use the notation
\begin{align*}
&\|\omega\|=\sup\{|\omega_s(t)|:s\geq 0, 0\leq t\leq \zeta_s(\omega)\}=\sup\{|\hat \omega_s|:s\geq 0\},\\
&M(\omega)= \sup\{\omega_s(t):s\geq 0, 0\leq t\leq \zeta_s(\omega)\}=\sup\{\hat \omega_s:s\geq 0\},
\end{align*}
for $\omega\in \S$. The fact that the two suprema in the definition of $\|\omega\|$ (or in the definition
of $M(\omega)$) are equal is a simple consequence of the snake property, which implies that
$$\{\omega_s(t):s\geq 0, 0\leq t\leq \zeta_s(\omega)\}=\{\hat \omega_s:s\geq 0\}.$$

One easily checks that a snake trajectory $\omega$ is completely determined by the two functions 
$s\mapsto \zeta_s(\omega)$ and  $s\mapsto\hat W_s(\omega)$. We will state this in a more precise form, but
for this we first need to introduce tree-like paths.

\begin{definition}
\label{def:treelikepaths}
A tree-like path is a pair $(h,f)$ where $h:\R_+\to \R_+$ and $f:\R_+\to \R$
are continuous functions that satisfy the following properties:
\begin{enumerate}
\item[\rm(i)] We have $h(0)=0$ and $\sigma_h:=\sup\{s\geq 0:h(s)\not =0\}<\infty$.
\item[\rm(ii)] For every $0\leq s\leq s'$, the condition
$$h(s)=h(s')=\min_{s\leq r\leq s'} h(r)$$
implies that $f(s)=f(s')$. 
\end{enumerate}
The set of all tree-like paths is denoted by $\T$, and, for every $x\in \R$,
$\T_x:=\{(h,f)\in \T:f(0)=x\}$
denotes the set of all tree-like paths with initial point $x$.
\end{definition}

\noindent{\it Remark}. Our terminology is inspired by the work of Hambly and Lyons,
who give a slightly different definition of a tree-like path in a more general setting (see \cite[Definition 1.2]{HL}). 

\smallskip

It follows from property (ii) that, if $(h,f)\in \T_x$, we have $f(s)=x$ for every $s\geq \sigma_h$.
The set $\T$ is equipped with the distance
$$d_\T((h,f),(h',f')) = |\sigma_h -\sigma_{h'}| + \sup_{s\geq 0} (|h(s)-h'(s)| + |f(s)-f'(s)|).$$

If $(h,f)$ is a tree-like path, $h$ satisfies the assumptions required in Section \ref{coding}
to define the tree $\t_h$. Then, property (ii) just says that, for every $s\geq 0$, $f(s)$ only depends 
on $p_h(s)$, and thus $f$ can as well be viewed as a function on the tree $\t_h$. Furthermore the function induced by
$f$ on $\t_h$ is also continuous. For $u\in \t_h$, we then interpret $f(u)$ as a spatial position, or a label, assigned
to the point $u$. 

\begin{proposition}
\label{homeo}
The mapping $\Delta: \S \to \T$ defined by $\Delta(\omega)= (h,f)$, where $h(s)=\zeta_s(\omega)$ and $f(s)=\hat W_s(\omega)$,
is a homeomorphism from $\S$ onto $\T$.
\end{proposition}

This is essentially the homeomorphism theorem of Marckert and Mokkadem \cite[Theorem 2.1]{MM}. Marckert and Mokkadem
impose the extra condition $\sigma=1$ for snake trajectories, and the similar condition for tree-like paths, but the proof is the same without this
condition. We mention that $\sigma(\omega)=\sigma_h$ if $(h,f)=\Delta(\omega)$. 

Let us briefly explain why Proposition \ref{homeo} is relevant to our purposes. Much of what follows is devoted to
studying the convergence of certain (random) snake trajectories. By Proposition \ref{homeo}, this convergence is
equivalent to that of the associated tree-like paths, which is often easier to establish.

\smallskip
\noindent{\it Remark}. Let $(h,f)$ be a tree-like path, and let $\omega$ be the associated snake trajectory. We already
noticed that $f$ can be viewed as a continuous function on the tree $\t_h$ coded by $\zeta$. 
The same holds for the mapping $s\mapsto \omega_s$. More precisely, for every
$s\geq 0$, and every $t\leq \zeta_s(\omega)=h(s)$, $\omega_s(t)$ is the value of $f$ at the
unique ancestor of $p_h(s)$ at distance $t$ from the root (recall that $d_h(\rho,p_h(s))=h(s)$). Thus the
finite path $\omega_s=(\omega_s(t))_{0\leq t\leq \zeta_s(\omega)}$ provides the values of $f$ along the ancestral line of
$p_h(s)$. We say that $\omega_s$ is the historical path of $p_h(s)$.

\begin{lemma}
\label{subtrajectory}
Let $\omega$ be a snake trajectory and $(h,f)=\Delta(\omega)$. Let $0< s<s' < \sigma(\omega)$ such that
$$h(s)=h(s')=\min_{s\leq r\leq s'} h(r).$$
Set, for every $r\geq 0$,
\begin{align*}
h'(r)&= h((s+r)\wedge s')-h(s)\\
f'(r)&=f((s+r)\wedge s').
\end{align*}
Then, $(h',f')$ is a tree-like path and the corresponding snake trajectory $\omega'=\Delta^{-1}(h',f')$ 
is called the subtrajectory of $\omega$ associated with the interval $[s,s']$.
\end{lemma}

We omit the easy proof. The assumption of the lemma is equivalent to saying that $p_h(s)=p_h(s')$. Suppose in addition that
$\{r\geq 0: p_h(r)=p_h(s)\}=\{s,s'\}$. Then $u:=p_h(s)$ is a point of multiplicity $2$ of ${\rm Sk}(\t_h)$, and the 
subtree of descendants of $u$ is coded by $f'$. Furthermore the snake trajectory $\omega'$
describes the spatial positions of the descendants of $u$.

\smallskip
Let us finally introduce three useful operations on snake trajectories. The first one is just the obvious translation. If $a\in \R$ and
$\omega\in \S$, $\kappa_a(\omega)$ is obtained by adding $a$ to all paths $\omega_s$: In other words 
$\zeta_s(\kappa_a(\omega))=\zeta_s(\omega)$ and $\hat W_s(\kappa_a(\omega))=\hat W_s(\omega)+a$
for every $s\geq 0$.

The second operation is the {\it re-rooting operation}. Let $\omega$ be a snake trajectory and let
$(h,f)$ be the associated tree-like path. Fix $s\in[0,\sigma(\omega)]$. We will define a new
snake trajectory $R_s(\omega)$, which is more conveniently described in terms of its associated tree-like path $(h^{[s]},f^{[s]})=\Delta(R_s(\omega))$.
Roughly speaking, $h^{[s]}$ is the coding function for the tree $\t_h$ re-rooted at $p_h(s)$ (this is informal 
since the coding function of a tree is not unique) and $f^{[s]}$ describes the ``same function'' as $f$ but viewed on
the re-rooted tree. To make this more precise, we set 
for every $r\in[0,\sigma(\omega)]$,
$$h^{[s]}(r)= h({s \oplus r}) + h(s) - 2 \underset{s \wedge (s \oplus r) \leq t \leq s \vee (s\oplus r)} {\min} h(t),$$
where  $s \oplus r = s+r  \ \text{if} \ s+r \leq \sigma(\omega)$ and $s \oplus r =s+r-\sigma(\omega)$ otherwise.
We also set $h^{[s]}(r)=0$ if $r>\sigma(\omega)$.
Furthemore we set $f^{[s]}(r)=f(s \oplus r)$ if $r\in[0,\sigma(\omega)]$ and $f^{[s]}(r)=f(s)$ if $r>\sigma(\omega)$. See \cite[Lemma 2.2]{DLG}
for the fact that the mapping $[0,\sigma(\omega)]\ni r\mapsto s\oplus r$ induces an isometry from the tree $\t_{h^{[s]}}$ onto the tree $\t_h$ (this in particular implies that $(h^{[s]},f^{[s]})$ is a tree-like path), and \cite[Section 2.3]{LGWeill}
for more details about this re-rooting operation.

The third and last operation is the {\it truncation} of snake trajectories, which will be important in this
work. Roughly speaking, if $\omega\in \S_x$ and $y\not =x$, the truncation of $\omega$ at $y$ is the new snake trajectory
$\omega'$ such that the values $\omega'_s$ are exactly all values $\omega_s$ for $s$ such that $\omega_s$ does not hit $y$, or
hits $y$ for the first time at its lifetime. Let us give a more precise definition. First, for any $\w\in\W$ and $y\in\R$, we set
$$\tau_y(\w):=\inf\{t\in[0,\zeta_{(\w)}]: \w(t)=y\}\;,\quad \tau^*_y(\w):=\inf\{t\in(0,\zeta_{(\w)}]: \w(t)=y\}\;,$$
with the usual convention $\inf\varnothing =\infty$. Note that $\tau^*_y(\w)$ may be different from $\tau_y(\w)$
only if $\w(0)=y$, but this case will be important in what follows.

\begin{proposition}
\label{truncation}
Let $x,y\in \R$. Let $\omega\in \S_x$, and
for every $s\geq 0$, set
$$A_s(\omega)=\int_0^s \mathrm{d}r\,\mathbf{1}_{\{\zeta_r(\omega)\leq\tau^*_y(\omega_r)\}},$$
and
$$\eta_s(\omega)=\inf\{r\geq 0:A_r(\omega)>s\}.$$
Then setting $\omega'_s=\omega_{\eta_s(\omega)}$ for every $s\geq 0$ defines an element of $\S_x$,
which will be denoted by  $\omega'={\rm tr}_y(\omega)$ and called the truncation of $\omega$ at $y$.
\end{proposition}

\begin{proof} First note that, by property (i) of the definition of a snake trajectory, we have 
$A_s(\omega)\to\infty $ as $s\to\infty$ (because $\zeta_r(\omega)\leq\tau^*_y(\omega_r)$ if $r\geq \sigma(\omega)$), and therefore $\eta_s(\omega)<\infty$ for every $s\geq 0$, so
that the definition of $\omega'$ makes sense. 

We need to verify that $\omega'\in\S_x$. To this end, we observe 
that the mapping $s\mapsto\eta_s(\omega)$ is right-continuous with left limits given by
$$\eta_{s-}(\omega)= \inf\{r\geq 0:A_r(\omega)=s\}\;,\quad\forall s>0.$$
To simplify notation, we write $\eta_s=\eta_s(\omega), \;\eta_{s-}=\eta_{s-}(\omega),\;A_s=A_s(\omega)$ and $\zeta_s=\zeta_s(\omega)$ in what follows.

We first verify the continuity of  the mapping $s\mapsto \omega'_s$.  
Let $s\geq 0$ such that $\zeta_{\eta_s}>0$. By the definition of $\eta_s$ there are values of $r>\eta_s$ arbitrarily close to $\eta_s$ such that 
$\zeta_r\leq \tau^*_y(\omega_r)$. Using the snake property, it then follows that the path $(\omega_{\eta_s}(t))_{0<t\leq \zeta_{\eta_s}}$ does
not hit $y$, or hits $y$ only at time $\zeta_{\eta_s}$ (notice that we excluded the value $t=0$ 
because of the particular case $y=x$, since we have trivially $\omega_{\eta_s}(0)=y$ in that case). Similarly, for every $s>0$
such that $\zeta_{\eta_{s-}}>0$, the path 
$(\omega_{\eta_{s-}}(t))_{0<t\leq \zeta_{\eta_{s-}}}$
does not hit $y$, or hits $y$ only at time $\zeta_{\eta_{s-}}$. 

Let $s>0$ be such that $\eta_{s-}<\eta_s$. The key observation is to note that 
\begin{equation}
\label{key-truncation}
\zeta_r\geq \zeta_{\eta_{s-}}=\zeta_{\eta_s}\;,\quad\forall r\in [\eta_{s-},\eta_s].
\end{equation}
In fact, suppose that \eqref{key-truncation} fails, so that certain values of $\zeta$
on the time interval $(\eta_{s-},\eta_s)$ are strictly smaller that $\zeta_{\eta_s}\vee \zeta_{\eta_{s-}}$.
Suppose for definiteness that $\zeta_{\eta_{s-}}\leq \zeta_{\eta_s}$ (the other case 
$\zeta_{\eta_{s-}}\geq \zeta_{\eta_s}$ is treated similarly).
Then
 we can find $r\in(\eta_{s-},\eta_s)$ such that 
$0<\zeta_r<\zeta_{\eta_s}$ and $\zeta_r=\min\{\zeta_u:u\in[r,\eta_s]\}$. By the snake property this means that $\omega_r$
is the restriction of $\omega_{\eta_s}$ to $[0,\zeta_r]$, and, since we know that $(\omega_{\eta_s}(t))_{0<t<\zeta_{\eta_s}}$
does not hit $y$, it follows that $\tau^*_y(\omega_r)=\infty$. Hence we have also $\tau^*_{y}(\omega_{r'})=\infty$, for all $r'$ sufficiently close to $r$, and therefore $A_{\eta_s}>A_{\eta_{s-}}$, which is a contradiction. 

The mapping $s\mapsto \omega_{\eta_s}$ is right-continuous and its left limit at $s>0$ is $\omega_{\eta_{s-}}$.
Property \eqref{key-truncation} and the snake property show that, for every $s$ such that $\eta_{s-}<\eta_s$,
we have $\omega_{\eta_{s-}}=\omega_{\eta_s}$, so that the mapping $s\mapsto \omega_{\eta_s}=\omega'_s$
is continuous. 

Furthermore, it also follows from \eqref{key-truncation} that, for every $s\leq s'$,
$$\min_{r\in[s,s']} \zeta_{\eta_r}=\min_{r\in[\eta_s,\eta_{s'}]} \zeta_r$$
and the snake property for $\omega'$ is a consequence of the same property for $\omega$.

We also need to verify that $\omega'_0=x$. This is immediate if
$y\not =x$ (because clearly $\eta_0=0$ in that case) but an argument is required in the case $y=x$,
which we consider now. It suffices to verify that $\zeta_{\eta_0}=0$. We argue by contradiction
and assume that $\zeta_{\eta_0}>0$, which implies that $\eta_0>0$. By previous observations, the path
$\omega_{\eta_0}$ does not hit $x$ during the time interval $(0,\zeta_{\eta_0})$. However, by the snake property again,
this implies that there is a set of positive Lebesgue measure of values of $r\in(0,\eta_0)$ such that
$\tau^*_x(\omega_r)=\infty$, which contradicts the definition of $\eta_0$. 

We finally notice that, for $s\geq \int_0^{\sigma(\omega)} \mathrm{d}r\,\mathbf{1}_{\{\zeta_r(\omega)\leq \tau^*_y(\omega)\}}$, we
have $\eta_s(\omega)\geq \sigma(\omega)$ and thus $\omega'_s=x$. This completes the proof of the property
$\omega'\in \S_x$.
\end{proof}

\noindent{\it Remark}. If $s>0$ is such that $\eta_{s-}<\eta_s$, and furthermore $\zeta_{\eta_s}>0$, then
we have $\tau^*_y(\omega_{\eta_s})=\zeta_s$. Indeed, 
since $A_{\eta_s}=A_{\eta_{s-}}=s$,
there exist values of $r<\eta_s$ arbitrarily close to $\eta_s$ such that 
$\tau^*_y(\omega_r)<\zeta_r$, and by the snake property it follows that we have $\hat \omega_{\eta_s}=y$.
Since we saw in the previous proof that $(\omega_{\eta_s}(t))_{0<t< \zeta_{\eta_s}}$ does not hit $y$, we get that $\tau^*_y(\omega_{\eta_s})=\zeta_s$.

\medskip

 The truncation operation
${\rm tr}_y$  is a measurable mapping from $\S_x$ into $\S_x$.  If $y\not =x$, and 
 if $\omega'={\rm tr}_y(\omega)$ is the truncation of a snake trajectory $\omega\in \S_x$, 
the paths $\omega'_s$ stay in $[y,\infty)$ (if $y<x$) or in $(-\infty,y]$ (if $y>x$) and can
only hit $y$ at their lifetime. 

The following lemma gives a simple continuity property of the truncation operations.

\begin{lemma}
\label{conti-trunc}
Let $\omega\in\S_0$ and $b<0$. Suppose that
$$\int_0^{\sigma(\omega)} \mathrm{d}s\,\mathbf{1}_{\{\tau_b(\omega_s)=\zeta_s(\omega)\}}=0.$$
Then, for any sequence $(b_n)_{n\geq 1}$ such that $b_n\downarrow b$ as $n\to\infty$, we have $\mathrm{tr}_{b_n}(\omega)\la \mathrm{tr}_{b}(\omega)$
in $\S$ as $n\to\infty$.
\end{lemma}

We omit the easy proof of this lemma.
We conclude this subsection with another lemma that 
will be useful in the proof of one of our main results. The proof is somewhat technical and may be omitted at
first reading. Recall the notation
$\underline\w=\min\{\w(t):0\leq t\leq \zeta_{(\w)}\}$ for $\w\in\W$. 

\begin{lemma}
\label{continuity-trunc}
Let $\omega\in \S$, and let $\omega'$ be a subtrajectory of $\omega$ associated with the interval
$[a,b]$. Assume that $\omega'\in\S_0$ and, for every $n\geq 1$, let $\omega^{(n)}$ be a subtrajectory of $\omega$ associated with the interval
$[a_n,b_n]$, such that $[a,b]\subset [a_n,b_n]$ for every $n\geq 1$ and $a_n\to a$, $b_n\to b$ as $n\to\infty$. 
Assume furthermore that  the following properties hold:
\begin{enumerate}
\item[(i)] $\omega_a(t)\geq 0$ for every $0\leq t\leq \zeta_{(\omega_a)}$;
\item[(ii)] for every $s\in(0,b-a)$, $\tau^*_0(\omega_s)\wedge
\zeta_{(\omega'_s)}>0$ and $\omega'_s(t)\geq 0$ for $0\leq t\leq \tau^*_0(\omega_s)\wedge
\zeta_{(\omega'_s)}$;
\item[(iii)] for every $s\in(0,b-a)$ such that $\zeta_{(\omega'_s)}>\tau^*_0(\omega'_s)$, we have $\underline\omega'_s<0$.
\end{enumerate}
Then, if $(\delta_n)_{n\geq 1}$ is any sequence of negative real numbers converging to $0$, we have 
$\mathrm{tr}_{\delta_n}(\omega^{(n)})\la \mathrm{tr}_{0}(\omega')$ in $\S$ as $n\to\infty$.
\end{lemma}

\begin{proof} The first step is to verify that $\omega^{(n)}$ converges to $\omega'$ in $\S$. To this end,
let $(h,f)$ be the tree-like path associated with $\omega$, and
notice that the tree-like path associated with $\omega^{(n)}$ is 
$(h^{(n)},f^{(n)})$, with $h^{(n)}(r)=h((a_n+r)\wedge b_n)-h(a_n)$ and 
$f^{(n)}(r)=f((a_n+r)\wedge b_n)$. From the convergences $a_n\to a$, $b_n\to b$, it immediately follows that the pair $(h^{(n)},f^{(n)})$
converges to the tree-like path $(h',f')$ associated with $\omega'$, and Proposition \ref{homeo} 
implies that $\omega^{(n)}$ converges to $\omega'$.

We also note that, for every $n\geq 1$, we have $f^{(n)}(0)=f(a_n)=\omega_{a_n}(h(a_n))= \omega_a(h(a_n))$, 
where the last equality holds because 
$p_h(a_n)$ is an ancestor of $p_h(a)$. Using (i), we get that $f^{(n)}(0)\geq 0$. By preceding remarks, we know that the
paths of $\mathrm{tr}_{\delta_n}(\omega^{(n)})$ stay in $[\delta_n,\infty)$.  

Set $\tilde \omega^{(n)}=\mathrm{tr}_{\delta_n}(\omega^{(n)})$
and $\tilde \omega'= \mathrm{tr}_{0}(\omega')$ to simplify notation.
Then set, for every $s\geq 0$,
$$A^{(n)}_s:=\int_0^s \mathrm{d}r\,\mathbf{1}_{\{ h^{(n)}(r)\leq \tau^*_{\delta_n}(\omega^{(n)}_r)\}},\quad
A'_s:=\int_0^s \mathrm{d}r\,\mathbf{1}_{\{ h'(r)\leq \tau^*_0(\omega'_r)\}},$$
and 
$$\eta^{(n)}_s:=\inf\{r\geq 0: A^{(n)}_r>s\},\quad \eta'_s:=\inf\{r\geq 0:A'_r>s\},$$
in such a way that $\tilde \omega^{(n)}_{s}=\omega^{(n)}_{\eta^{(n)}_s}$ and 
$\tilde \omega'_s=\omega'_{\eta'_s}$ by the definition of truncations. We observe that, for every
$s\geq 0$, we have
\begin{equation}
\label{trunc-tech}
A^{(n)}_s\build{\la}_{n\to\infty}^{} A'_s.
\end{equation}
To see this, note that, for $r\in[a,b]$, the paths $\omega_r$ are the same as $\omega_a$
up to time $h(a)=\zeta_a(\omega)$, and thus stay nonnegative on the time interval $[0,h(a)]$
by assumption (i). From our definitions, it follows that the paths $\omega^{(n)}_{a-a_n+r}$, for
$0\leq r\leq b-a$, stay nonnegative up to time $h(a)-h(a_n)\geq 0$. Then, for
$r\in [0,b-a]$, we have $\omega'_r(\cdot)=\omega^{(n)}_{a-a_n+r}(h(a)-h(a_n)+\cdot)$, and by (ii) we get that,
if  $h'(r)\leq \tau^*_0(\omega'_r)$, the path $\omega^{(n)}_{a-a_n+r}$ does not hit
$\delta_n<0$ between times $h(a)-h(a_n)$ and $h^{(n)}(a-a_n+r)$. Hence, 
we have, for every $r\in[0,b-a]$,
$$\mathbf{1}_{\{ h'(r)\leq \tau^*_0(\omega'_r)\}}\leq 
\mathbf{1}_{\{ h^{(n)}(a-a_n+r)\leq \tau^*_{\delta_n}(\omega^{(n)}_{a-a_n+r})\}}.$$
It follows that $A'_s\leq A^{(n)}_{a-a_n+s}\leq A^{(n)}_s + (a-a_n)$, which  implies
$$\liminf_{n\to\infty} A^{(n)}_s \geq A'_s,$$
for every $s\geq 0$. Conversely, we claim that, for every $r\in(0,b-a)$,
$$\limsup_{n\to\infty} \mathbf{1}_{\{ h^{(n)}(r)\leq \tau^*_{\delta_n}(\omega^{(n)}_r)\}}
\leq \mathbf{1}_{\{ h'(r)\leq \tau^*_0(\omega'_r)\}}.$$
Indeed, if $\tau^*_0(\omega'_r)<h'(r)$, then assumption (iii) implies that
$\omega'_r$ takes negative values before its lifetime. From the convergence 
of $\omega^{(n)}_r$ to $\omega'_r$, we get that we must have $\tau^*_{\delta_n}(\omega^{(n)}_r)<h^{(n)}(r)$
for $n$ large, proving our claim. The claim now gives
$$\limsup_{n\to\infty} A^{(n)}_s \leq A'_s,$$
completing the proof of \eqref{trunc-tech}. Notice that  \eqref{trunc-tech} also
implies that $A^{(n)}_{b_n-a_n} \la A'_{b-a}$, from which one gets that
$\sigma(\tilde\omega^{(n)})\la \sigma(\tilde \omega')$, noting that $\sigma(\tilde\omega')=A'_{b-a}$ as a consequence of (ii)
(if $0<s<A'_{b-a}$, $\tilde \omega'_s=\omega'_{\eta'_s}$ is not a trivial path by (ii)
and the fact that $0<\eta'_s<b-a$).

It follows from \eqref{trunc-tech} that we have $\eta^{(n)}_s\la \eta'_s$, and 
consequently $\tilde\omega^{(n)}_s \la \tilde\omega'_s$, as $n\to\infty$, for
every $s\geq 0$ such that $\eta'_s=\eta'_{s-}$. To see that this implies the
uniform convergence of $\tilde\omega^{(n)}$ toward $\tilde\omega'$, we argue
by contradiction. Suppose that this uniform convergence does not hold, so that
(modulo the extraction of a subsequence of $(\tilde \omega^{(n)})_{n\geq 1}$)
we can find a sequence $(s_n)_{n\geq 1}$ and a real $\xi>0$ such that,
for every $n$,
\begin{equation}
\label{contradict-trunc}
d_\S(\tilde\omega^{(n)}_{s_n},\tilde\omega'_{s_n})>\xi.
\end{equation}
Since both $\tilde\omega^{(n)}_{r}$ and $\tilde\omega'_{r}$ are constant
(and equal to a trivial path) when $r\geq \sigma(\omega)$, we can assume
that $s_n\in[0,\sigma(\omega)]$ for every $n$ and then, modulo the
extraction of  a subsequence, that $s_n\la s_\infty$ as $n\to\infty$. We must then 
have $\eta'_{s_\infty-}<\eta'_{s_\infty}$ because otherwise \eqref{trunc-tech} 
would imply that $\eta^{(n)}_{s_n}\la \eta_{s_\infty}$
and therefore $\tilde\omega^{(n)}_{s_n}\la \tilde \omega'_{s_\infty}$,
contradicting \eqref{contradict-trunc}. 
We can also assume that $0<s_\infty<\sigma(\tilde\omega')$, and therefore
$0<\eta'_{s_\infty}<b-a$, since it follows from
assumption (ii) that $\eta'$ is continuous at $\sigma(\tilde \omega')=A'_{b-a}$
(if $0<s<b-a$, property (ii) and the snake property imply that the interval $[s,b-a]$
contains a set of positive Lebesgue measure of values of $r$ such that
$\tau^*_0(\omega(r))=\infty$, and this is what we need to get the latter
continuity property).
Also notice that (ii) implies $h'(r)>0$ for $0<r<b-a$ and consequently 
$h'(\eta'_r)>0$ for $0<r<\sigma(\tilde \omega')$. 

From \eqref{trunc-tech}, we get that any 
accumulation point of the sequence $(\eta^{(n)}_{s_n})_{n\geq 1}$ must lie in the 
interval $[\eta'_{s_\infty-},\eta'_{s_\infty}]$. We claim that for any such accumulation point $r$ we have $\omega'_r=\omega'_{\eta'_{s_\infty}}$.
This implies that $\tilde\omega^{(n)}_{s_n}=\omega^{(n)}_{\eta^{(n)}_{s_n}}$
converges to $\omega'_{\eta'_{s_\infty}}=\tilde \omega'_{s_\infty}$ and contradicts \eqref{contradict-trunc}.
To verify our claim, let $r\in [\eta'_{s_\infty-},\eta'_{s_\infty}]$   be an accumulation point of 
the sequence $(\eta^{(n)}_{s_n})_{n\geq 1}$. By property \eqref{key-truncation}
in the proof of Proposition \ref{truncation}, we know that the path $\omega'_r$
coincides with $\omega'_{\eta'_{s_\infty}}$ up to $h'(\eta'_{s_\infty})=\tau^*_0(\omega'_{\eta'_{s_\infty}})$
(the last equality by the remark following Proposition \ref{truncation}).
However, $h'(r)>h'(\eta'_{s_\infty})$ is impossible since assumption (iii) would imply that
$\omega'_r$ takes negative values and cannot be an accumulation point of the sequence 
$\tilde\omega^{(n)}_{s_n}$ (because $\tilde\omega^{(n)}_{s_n}$ takes values in $[\delta_n,\infty)$
and $\delta_n$ tends to $0$ as $n\to\infty$). Therefore 
we have $h'(r)=h'(\eta'_{s_\infty})$ meaning that $\omega'_r=\omega'_{\eta'_{s_\infty}}$
as desired. This completes the proof. 
\end{proof}

\subsection{The Brownian snake}
In this section we discuss the (one-dimensional) Brownian snake excursion measures. We
avoid defining the Brownian snake starting from a general initial value (which
is briefly presented in the Introduction above) as this definition is not required in what
follows, except in the proof of one technical lemma (Lemma \ref{no-local-minimum}) which the 
reader can skip at first reading.

Let $h:\R_+\to\R_+$ satisfy the assumptions of Section \ref{coding}
(including assumptions (i)--(iii) from  the end of this subsection) and also assume that $h$
is H\"older continuous with exponent $\delta$
for some $\delta>0$. Let $(G^h_s)_{s\geq 0}$ be the centered real Gaussian process with covariance 
\begin{equation}
\label{cov-snake}
{\rm cov}(G^h_s,G^h_t)=\min_{s\wedge t\leq r\leq s\vee t} h(r),
\end{equation}
for every $s,t\geq 0$. We leave it as an exercise to verify that the right-hand side of \eqref{cov-snake}
is a covariance function (see Lemma 4.1 in \cite{LGM}).
Note that we have then 
\begin{equation}
\label{dist-snake}
E[(G^h_s-G^h_t)^2]=d_h(s,t).
\end{equation}
An application of the classical Kolmogorov lemma shows that $(G^h_s)_{s\geq 0}$ has 
a continuous modification, which we consider from now on. Then
property \eqref{dist-snake} entails that, for every fixed $0\leq s\leq t$ such that $d_h(s,t)=0$,
we have $P(G^h_s=G^h_t)=1$. A continuity argument, using the assumptions satisfied by $h$, then shows that,
a.s., for every $0\leq s\leq t$, the property $d_h(s,t)=0$
implies $G^h_s=G^h_t$. This means that apart from a set of probability $0$ which we may discard,
the pair $(h,G^h)$ is a (random) tree-like path in the sense of the preceding subsection.

The (one-dimensional) Brownian snake driven by $h$ is the random snake trajectory $W^h=(W^h_s)_{s\geq 0}$
associated with the tree-like path $(h,G^h)$. 
We write $\mathbf{P}_h(\mathrm{d}\omega)$ for the law of $W^h$ on the space $\S_0$.

We next randomize $h$: We let $\mathbf{n}(\mathrm{d}h)$ stand for It\^o's excursion measure
of positive excursions of linear Brownian motion (see e.g. \cite[Chapter XII]{RY}) normalized so
that, for every $\ve >0$,
$$\mathbf{n}\Big(\max_{s\geq 0} h(s) >\ve\Big) = \frac{1}{2\ve}.$$
Notice that $\mathbf{n}$ is supported on functions $h$ that satisfy the assumptions required above to
define $W^h$ and the probability measure $\mathbf{P}_h(\mathrm{d}\omega)$.
The Brownian snake excursion measure $\N_0$ is then the $\sigma$-finite measure on $\S_0$
defined by
$$\N_0(\mathrm{d}\omega)= \int \mathbf{n}(\mathrm{d}h)\, \mathbf{P}_h(\mathrm{d}\omega).$$
In other words, the ``lifetime process'' $(\zeta_s)_{s\geq 0}$ is distributed under $\N_0(\mathrm{d}\omega)$
according to It\^o's measure $\mathbf{n}(\mathrm{d}h)$, and, conditionally on $(\zeta_s)_{s\geq 0}$, 
$(W_s)_{s\geq 0}$ is distributed as the Brownian snake driven by $(\zeta_s)_{s\geq 0}$.
The reader will easily check that the preceding definition of $\N_0$ is consistent 
with the slightly different presentation given in the Introduction above
(see \cite{livrevert} for more details about the Brownian snake).
For every $x\in \R$, we also define $\N_x$ as the measure on $\S_x$ which is the image of $\N_0$ under the translation $\kappa_x$. 

Let us recall the first-moment formula for the Brownian snake \cite[Section IV.2]{livrevert}. For 
every nonnegative measurable function $\phi$ on $\W$,
\begin{equation}
\label{first-moment}
\N_x\Big(\int_0^\sigma \mathrm{d}s\,\phi(W_s)\Big) = \E_x\Big[\int_0^\infty \mathrm{d}t\,\phi\big((B_r)_{0\leq r\leq t}\big)\Big],
\end{equation}
where $B=(B_r)_{r\geq 0}$ stands for a linear Brownian motion starting from $x$ under the 
probability measure $\P_x$. Here we recall that $\N_x$ is a measure on $\S_x$, and so the duration $\sigma$
is well-defined under $\N_x$ as in Definition \ref{def:snakepaths}.

We define the range $\mathcal{R}$ by
$$\mathcal{R}:=\{\hat W_s:s\geq 0\}=\{W_s(t):s\geq 0, 0\leq t\leq \zeta_s\},$$
and we set
$$W_*:=\min\mathcal{R}.
$$
Then, if $x,y\in \R$ and $y<x$, we have
\begin{equation}
\label{law-mini}
\N_x(W_*\leq y)=\frac{3}{2(x-y)^2}.
\end{equation}
See e.g. \cite[Section VI.1]{livrevert}. 

\subsection{Exit measures and the special Markov property}
\label{exitSMP}
In this section, we briefly describe a key result of \cite{snakesolutions} that plays a crucial role in
the present work.
Let $U$ be a nonempty open interval of $\R$, such that $U\not =\R$. For any $\w\in \W$, set
$$\tau^U(\w):=\inf\{t\in[0,\zeta_{(\w)}]: \w(t)\notin U\}.$$
 If $x\in U$, the limit 
 \begin{equation}
 \label{limit-exit}
 \langle \z^U,\phi\rangle = \lim_{\ve\to 0} \frac{1}{\ve} \int_0^\sigma \mathrm{d}s\,\mathbf{1}_{\{\tau^U(W_s)<\zeta_s<\tau^U(W_s)+\ve\}}\,\phi(W_s(\tau^U(W_s)))
 \end{equation}
 exists $\N_x$ a.e. for any function $\phi$ on $\partial U$ and defines a
 finite random measure $\z^U$ supported on $\partial U$
(see  \cite[Chapter V]{livrevert}). Notice that here $\partial U$ has at most two points, but the 
preceding definition holds in the same form for the Brownian snake in higher dimensions with an arbitrary open set $U$.
Informally, the measure $\z^U$ ``counts'' the exit points of the paths $W_s$ from $U$, for those values
of $s$ such that $W_s$ exits $U$. In particular, $\z^U=0$ if none of the paths
$W_s$ exits $U$. 

Exit measures are needed to state the
so-called special Markov property. 
Before stating this property, we  introduce the excursions outside $U$
of a snake trajectory.
We fix $x\in U$ and we let $\omega\in \S_x$. We observe that the set
$$\{s\geq 0: \tau^U(\omega_s)<\zeta_s\}$$
is open and can therefore be written as a union of disjoint open intervals
$(a_i,b_i)$, $i\in I$, where $I$ may be empty. From the fact that $\omega$
is a snake trajectory, it is not hard to verify that we must have $p_\zeta(a_i)=p_\zeta(b_i)$ for every $i\in I$,
where $p_\zeta$ is the canonical projection from $\R_+$ onto
the tree $\t_\zeta$ coded by $(\zeta_s(\omega))_{s\geq 0}$. Furthermore the path 
$\omega_{a_i}=\omega_{b_i}$ exits $U$ exactly at its lifetime $\zeta_{a_i}=\zeta_{b_i}$.
We can then define the excursion $\omega_i$, for every $i\in I$, as the subtrajectory of $\omega$ associated with the interval $[a_i,b_i]$
(equivalently $W_s(\omega_i)$
is the finite path $(\omega_{(a_i+s)\wedge b_i}(\zeta_{a_i}+t))_{0\leq t\leq \zeta^i(s)}$
with lifetime $\zeta^i(s)=\zeta_{(a_i+s)\wedge b_i}-\zeta_{a_i}$, for every $s\geq 0$). 
The $\omega_i$'s are the ``excursions'' of the snake trajectory $\omega$ outside $U$
-- the word ``outside'' is a little misleading here, because although these excursions start from $\partial U$, they will 
typically come back inside $U$. We define the point measure of excursions of $\omega$
outside $U$ by
$$\mathcal{P}^U(\omega):=\sum_{i\in I} \delta_{\omega_i}.$$

We also need to define the $\sigma$-field on $\S_x$ containing the information given by the paths
$\omega_s$ before they exit $U$. To this end we generalize a little the definition of truncations in Section \ref{statespace}.
If $\omega\in \S_x$, we set
$$\mathrm{tr}^U(\omega)_s:=\omega_{\eta^U_s}$$
where 
$$\eta^U_s:=\inf\{r\geq 0:\int_0^r \mathrm{d}t\,\mathbf{1}_{\{\zeta_t(\omega)\leq \tau^U(\omega_t)\}} >s\}.$$
Just as in Proposition \ref{truncation}, we can verify that this defines a measurable mapping
from $\S_x$ into $\S_x$. We define the $\sigma$-field  $\mathcal{E}^{U}_x$ on $\S_x$ as the
$\sigma$-field generated by this mapping and completed by the measurable sets of $\S_x$ of $\N_x$-measure $0$.

We can now state the special Markov property.

\begin{proposition}
\label{SMP}
Let $x\in U$. 
The random measure $\z^U$ is $\mathcal{E}^U_x$-measurable. Furthermore,
under the probability measure $\N_x(\cdot\mid\mathcal{R}\cap U^c\not =\varnothing)$, conditionally on $\mathcal{E}^{U}_x$, the point measure
$\mathcal{P}^U$
is Poisson with intensity 
$$\int \z^U(\mathrm{d}y)\,\N_y(\cdot).$$
\end{proposition}

See \cite[Theorem 2.4]{snakesolutions} for a proof in a
much more general setting.
Note that, on the event $\{\mathcal{R}\cap U^c =\varnothing\}$, there are no excursions outside $U$, and this is the reason
why we restrict our attention to the event $\{\mathcal{R}\cap U^c \not =\varnothing\}$, which has finite $\N_x$-measure 
by \eqref{law-mini} (in fact, since $\z^U=0$ on $\{\mathcal{R}\cap U^c =\varnothing\}$, we could as well give a statement 
similar to Proposition \ref{SMP} without conditioning). 

\subsection{The exit measure process}
\label{sec:exitprocess}

We now specialize the discussion of the previous subsection to the case $U=(y,\infty)$ and $x>y$. 
The exit measure $\z^{(y,\infty)}$ is then a random multiple
of the Dirac mass at $y$, and is determined by its total mass, which will be denoted
by $\z_y=\langle \z^{(y,\infty)},1\rangle
$. We have
$$\{\z_y>0\}=\{W_*<y\}=\{W_*\leq y\}\;,\quad \N_x\hbox{ a.e.}$$
Note that the identity $\{W_*<y\}=\{W_*\leq y\}$, $\N_x$ a.e., follows from the fact that the 
right-hand side of \eqref{law-mini} is a continuous function of $y$. The fact that $\{\z_y>0\}=\{W_*<y\}$, $\N_x$ a.e., 
can then be deduced from the special Markov property (Proposition \ref{SMP}).

The Laplace transform of $\z_y$ under $\N_x$ can be computed from the connections between 
exit measures and semilinear partial differential equations \cite[Chapter V]{livrevert}. For every
$\lambda >0$,
\begin{equation}
\label{Laplace-exit}
\N_x(1-\exp(-\lambda \z_y))= \Big(\lambda^{-1/2} + \sqrt{\frac{2}{3}}\,(x-y)\Big)^{-2}.
\end{equation}
See formula (6) in \cite{Hull} for a brief justification. Note that letting $\lambda\to\infty$
in \eqref{Laplace-exit} is consistent with \eqref{law-mini}. A consequence of \eqref{Laplace-exit}
is the fact that
\begin{equation}
\label{first-moment-exit}
\N_x(\z_y)=1.
\end{equation}

Let us discuss Markovian properties of the process of
exit measures. If $y'<y<x$, an application of the special Markov property combined with
formula \eqref{Laplace-exit} gives on the event $\{W_*\leq y\}$, for every $\lambda >0$,
$$\N_x\Big(\exp-\lambda\z_{y'}\,\Big|\, \mathcal{E}^{(y,\infty)}_x\Big)
= \exp \Big(-\z_y\,\N_y(1-\exp(-\lambda \z_{y'}))\Big) = \exp- \z_y \Big(\lambda^{-1/2} + \sqrt{\frac{2}{3}}\,(y-y')\Big)^{-2}.$$
It follows that the process $(\z_{x-a})_{a>0}$ is Markovian under $\N_x$, with the transition kernels of the
continuous-state branching process with branching mechanism $\psi(\lambda)=\sqrt{8/3}\;\lambda^{3/2}$
(see e.g. \cite[Section 2.1]{Hull} for the definition and properties of this process). Although $\N_x$ is an infinite measure, the
previous statement makes sense by arguing on the event $\{W_*\leq x-\delta\}$, which has finite
$\N_x$-measure for any $\delta>0$, and considering $(\z_{x-\delta-a})_{a\geq0}$. 

We will use an approximation of $ \z_y$ by $\mathcal{E}^{(y,\infty)}_x$-measurable random variables
(notice that this is not the case for \eqref{limit-exit}).
Recall our notation $\tau_y(\w):=\inf\{t\in[0,\zeta_{(\w)}]: \w(t)=y\}$ for $\w\in\W$. 

\begin{lemma}
\label{approx}
Let $y<x$. 
We have
$$\ve^{-2}\int_0^\sigma \mathrm{d}s\,\mathbf{1}_{\{\zeta_s\leq \tau_y(W_s), \hat W_s<y+\ve\}}
\build{\la}_{\ve\to 0}^{} \z_y$$
where the convergence holds in probability under $\N_x(\cdot\mid W_*\leq y)$.
\end{lemma}

\begin{proof} 
This follows from arguments similar to the proof of Proposition 1.1 in \cite[Section 4.1]{Hull}, and we only sketch the proof.
For every $\ve>0$, set
$$\Lambda_\ve=\int_0^\sigma \mathrm{d}s\,\mathbf{1}_{\{\zeta_s\leq \tau_y(W_s), \hat W_s<y+\ve\}}.$$
If $\ve\in(0,x-y)$, the special Markov property applied to the domain $(y+\ve,\infty)$ shows that the conditional 
distribution of $\Lambda_\ve$, under $\N_x(\cdot\mid W_*\leq y+\ve)$ and knowing $\mathcal{E}^{(y+\ve,\infty)}$,
is the law of $S_\ve(\z_{y+\ve})$, where $(S_\ve(t))_{t\geq 0}$ is a subordinator whose L\'evy measure is
the  law of $\Lambda_\ve$ under $\N_{y+\ve}$ (recall the comments following
Proposition \ref{exit-definition} about laws of random variables under $\sigma$-finite measures), and $S_\ve$ is 
assumed to be independent of $\z_{y+\ve}$.
The first-moment formula for the Brownian snake \eqref{first-moment} gives $\N_{y+\ve}(\Lambda_\ve)=\ve^2$,
so that $S_\ve(t)$ has mean $\ve^2t$.
On the other hand, scaling arguments entail that $(S_\ve(t))_{t\geq 0}$ has the same 
distribution as $(\ve^4S_1(\ve^{-2}t))_{t\geq 0}$. Hence, under $\N_x(\cdot\mid W_*\leq y+\ve)$ and conditionally on $\mathcal{E}^{(y+\ve,\infty)}$,
$\ve^{-2}\Lambda_\ve$ has the law of $\ve^{2}S_1(\ve^{-2}\z_{y+\ve})$, and the latter random variable 
is close in probability to $\z_{y+\ve}$ by the law of large numbers ($t^{-1}S_1(t)$ converges in probability to $1$ as $t\to\infty$).
The result of the lemma follows since $\z_{y+\ve}$ converges to $\z_y$ in probability when $\ve\to 0$.
\end{proof}

We note that the quantities $\int_0^\sigma \mathrm{d}s\,\mathbf{1}_{\{\zeta_s\leq \tau_y(W_s), \hat W_s<y+\ve\}}$
are functions of 
the truncation ${\rm tr}_y(\omega)$, and therefore
$\mathcal{E}^{(y,\infty)}_x$-measurable. As a consequence of Lemma \ref{approx}, we can fix a sequence
$(\alpha_n)_{n\geq 1}$ of positive reals converging to $0$ such that
\begin{equation}
\label{approx-ps}
\z_y= \lim_{n\to\infty} \alpha_n^{-2}\int_0^\sigma \mathrm{d}s\,\mathbf{1}_{\{\zeta_s\leq\tau_y(W_s), \hat W_s<y+\alpha_n\}}\;,\qquad \N_x\hbox{ a.e.}
\end{equation}
and we can even choose the sequence $(\alpha_n)_{n\geq 1}$ independently of the pair $(x,y)$ such that $y<x$
(observe that if \eqref{approx-ps} holds for $y=x-\delta$, then an application of the special Markov property (Proposition \ref{SMP}) shows that
it holds for every $y\in(-\infty,x-\delta]$). It will be convenient to define $\z_y(\omega)$ for every $\omega\in \S_x$, by setting
$$\z_y(\omega)= \liminf_{n\to\infty} \alpha_n^{-2}\int_0^{\sigma(\omega)} \mathrm{d}s\,
\mathbf{1}_{\{\zeta_s(\omega)\leq\tau_y(W_s(\omega)), \hat W_s(\omega)<y+\alpha_n\}}.$$
By the previous considerations, this definition is consistent with (\ref{limit-exit}) up to an $\N_x$-negligible set.
Furthermore, we have $\z_y(\omega)=\z_y(\mathrm{tr}_y(\omega))$ for every $\omega\in\S_x$. 

\smallskip
In much of what follows, we will argue under the measure $\N_0$, and we simply write 
$\mathcal{E}^{(y,\infty)}$ instead of $\mathcal{E}^{(y,\infty)}_0$, for every $y<0$. For $\omega\in \S_0$, we use the
notation
$$Z_a(\omega)=\z_{-a}(\omega)$$
for every $a>0$.  Because continuous-state branching processes are Feller processes, we know that
the process $(Z_a)_{a>0}$ has a c\`adl\`ag modification under $\N_0$, and we will always consider
this modification. We call $(Z_a)_{a>0}$ the exit measure process.

We will need some bounds on the moments of $Z_a$. By \eqref{first-moment-exit}, we already
know that $\N_0(Z_a)=1$ for every $a>0$. Moreover, an application of the special Markov property
shows that the process $(Z_{\delta+a})_{a\geq 0}$ is a martingale under $\N_0(\cdot\mid W_*\leq -\delta)$,
for every $\delta>0$ (this also follows from the fact that $\psi(\lambda)=\sqrt{8/3}\,\lambda^{3/2}$ is
a critical branching mechanism).

\begin{lemma}
\label{exit-moments}
Let $p\in(1,3/2)$. For every $0<b\leq a$, we have $\N_0((Z_b)^p)\leq \N_0((Z_a)^p)<\infty$.
\end{lemma}

\begin{proof}
Write $\N^{(a)}_0:=\N_0(\cdot\mid W_*\leq -a)$ to simplify notation. As a consequence of
\eqref{Laplace-exit} and \eqref{law-mini}, we get that, for every $\lambda>0$,
$$\N^{(a)}_0\Big( e^{-\lambda Z_a}\Big) = 1- \Big(1+a^{-1}\sqrt{\frac{3}{2\lambda}}\Big)^{-2},$$
and we have also $\N^{(a)}_0(Z_a)=2a^2/3$. From a Taylor expansion, we get
$$\N^{(a)}_0\Big( e^{-\lambda Z_a}\Big) - (1-\lambda\N^{(a)}_0(Z_a)) = 
2\Big(\frac{2}{3}\Big)^{3/2}\,a^3\,\lambda^{3/2} + o(\lambda^{3/2}),
$$
as $\lambda\to 0$. 
By \cite[Theorem 8.1.6]{Bingham}, this implies the existence of 
a constant $C$ such that $\N^{(a)}_0(Z_a>x)\leq C\,x^{-3/2}$ for every
$x>0$. Consequently, $\N^{(a)}_0((Z_a)^p)<\infty$ if $1<p<3/2$. 

Finally, if $b\in(0,a)$, we get by using the martingale property of the
exit measure process,
$$\N_0((Z_b)^p)=\frac{3}{2b^2}\,\N_0^{(b)}((Z_b)^p)\leq \frac{3}{2b^2}\,\N_0^{(b)}((Z_a)^p)=\N_0((Z_a)^p)<\infty.$$
\end{proof}

\subsection{A technical lemma}
\label{sec:tech-lemma}
We finally give a technical lemma concerning local minima of the process $\hat W$.

\begin{lemma}
\label{no-local-minimum} $\N_0$ a.e., there exists no value of $s\in (0,\sigma)$ such that:
\begin{enumerate}
\item[\rm(i)] $s$ is a time of local minimum of $\hat W$, in the sense that there exists $\ve>0$
such that $\hat W_r\geq \hat W_s$ for every $r\in (s-\ve,s+\ve)$.
\item[\rm(ii)] $\hat W_s=\underline W_s$ and there exists $t\in(0,\zeta_s)$ such that $W_s(t)=\underline W_s$.
\end{enumerate}
\end{lemma}

\begin{proof} The proof uses more involved properties of the Brownian snake, which we have not recalled but for which we refer the reader
to \cite{livrevert}. We start by observing that, for every reals $y<x$, we have $\N_x$ a.e.
\begin{equation}
\label{hitting-strict}
\inf\{s\geq 0: \hat W_s<y\}=\inf\{s\geq 0:\hat W_s\leq y\}.
\end{equation}
In other words, when the Brownian snake hits $y$, it immediately hits values strictly smaller than $y$. See the proof
of Theorem VI.9 in \cite{livrevert} for an argument in a more general setting.

Then, fix $\w\in\W_0$ and let $(W'_s)_{s\geq 0}$ be a Brownian snake that starts from $\w$
under the probability measure $\P_\w$ (we write $W'_s$ and not $W_s$ because $\P_\w$ 
is not defined on the space $\S$ of snake trajectories). We let 
$(\zeta'_s)_{s\geq 0}$ be the lifetime process of $(W'_s)_{s\geq 0}$. Suppose that
there is a unique time $t_0\in(0,\zeta_{(\w)})$ such that $\w(t_0)=\underline\w$, and introduce the stopping time 
$$\tau:=\inf\{s\geq 0: \zeta'_s\leq t_0\}.$$
Notice that  the path $W'_\tau$ is equal to the restriction of $\w$ to $[0,t_0]$, and thus $\hat W'_\tau=\w(t_0)=\underline \w$.
We then claim that, $\P_\w$ a.s. on the event where $\inf\{s> 0: \hat W'_s\leq \underline\w\}<\tau$, we have
$$\inf\{s> 0: \hat W'_s\leq \underline\w\}=\inf\{s> 0: \hat W'_s< \underline\w\}.$$
This follows by using the subtree decomposition of the Brownian snake started at $\w$
(see \cite[Lemma V.5]{livrevert}) together with property \eqref{hitting-strict} above.

We can now combine the previous observations with the Markov property of the Brownian snake under $\N_0$. 
We obtain that $\N_0$ a.e. for every rational $r\in(0,\sigma)$ such that $t\mapsto W_r(t)$
attains its minimum at a (necessarily unique) time $t_0\in (0,\zeta_r)$, the property
$$\inf\{s> r: \hat W_s\leq \underline W_r\}<\inf\{s\geq r: \zeta_s\leq t_0\}$$
implies
\begin{equation}
\label{hitting-strict2}
\inf\{s>r:\hat W_{s}<\underline W_{r}\}=\inf\{s>r:\hat W_{s}\leq\underline W_r\}.
\end{equation}

Let show that this implies the statement of the lemma. We argue by contradiction, assuming that there is a value 
$s_0\in(0,\sigma)$ such that properties (i) and (ii) hold for $s=s_0$. Write $t_0$ for the
(unique) time in $(0,\zeta_{s_0})$ such that $W_{s_0}(t_0)=\underline W_{s_0}$
and choose $\delta>0$ such that $t_0<\zeta_{s_0}-\delta$. Then, using property (i) for $s=s_0$ and the properties
of the Brownian snake, we can find a rational $r<s_0$ sufficiently close to $s_0$ so that, for some $\chi>0$,
\begin{enumerate}
\item[(a)]
$\hat W_{s}\geq \hat W_{s_0}$ for every $s\in[r,s_0+\chi]$;
\item[(b)]
$\zeta_r+\delta/2>\zeta_{s}>\zeta_r-\delta/2$ for every $s\in[r,s_0]$.
\end{enumerate}
 We note that $W_r$ coincides with $W_{s_0}$ at least
up to time $\zeta_r-\delta/2>\zeta_{s_0}-\delta>t_0$.
In particular $t_0$ is also the unique time of the minimum 
of $t\mapsto W_r(t)$ on $(0,\zeta_r)$, and $\underline W_r= \underline W_{s_0}=\hat W_{s_0}$ (it already follows from property (a) that $\underline W_r\geq \hat W_{s_0}$). 
Property (b) then gives
$$\inf\{s> r: \hat W_s\leq \underline W_r\}\leq s_0<\inf\{s\geq r: \zeta_s\leq t_0\}.$$
This allows us to apply \eqref{hitting-strict2} and to get
$$\inf\{s>r:\hat W_{s}<\underline W_{r}\}=\inf\{s>r:\hat W_{s}\leq\underline W_r\}\leq s_0.$$
Since $\underline W_r=\hat W_{s_0}$, this contradicts property (a) above, and this contradiction completes the proof.
\end{proof}

\section{Construction of the excursion measure above the minimum}
\label{sec:construct}

The main goal of this section is to construct the positive excursion measure  $\mathbb{N}_0^*$. For this 
construction, we will be arguing under
the measure $\N_0$. Several properties stated below hold
only outside an $\N_0$-negligible set, but we will frequently omit the words $\N_0$ a.e.
Recall the notation $\t_\zeta$ for the random real tree coded by $(\zeta_s)_{s\geq 0}$, and 
$\text{Sk}(\mathcal{T}_{\zeta})$ for the skeleton of $\t_\zeta$. If $u\in \t_\zeta$ and $s\geq 0$ is such that
$p_\zeta(s)=u$, we already noticed that $W_s$ does not depend on the choice of $s$, and it will be
convenient to write $V_u=\hat W_s$. Then $V_u$ is interpreted as the label or spatial position of $u$.

\begin{definition}
A vertex $u \in \mathcal{T}_{\zeta}$ is an excursion debut above the minimum if
the following three properties hold:
\begin{enumerate}
\item 
 $u \in \mathrm{Sk}(\mathcal{T}_{\zeta})\;$;
 \item $V_u=\min\{V_v:v\in\llbracket \rho,u\rrbracket\}\;$;
 \item $u$ has a strict descendant  $w$ such that such that $V_v >V_u $ for all $v \in \rrbracket u,w \rrbracket$.
 \end{enumerate} We write $D$ for the set of all
excursion debuts above the minimum. If $u\in D$, $V_u$ is called the level of the excursion debut $u$.
\end{definition}

In what follows, except in Section \ref{excupoint}, we will be interested only in 
excursions above the minimum, and for this reason we will say excursion debut instead of
excursion debut above the minimum.
By definition, excursion debuts belong to the skeleton of $\t_\zeta$. Clearly, $\N_0$ a.e., the root $\rho$ is not an excursion debut (it is easy to see 
that property (3) fails for $u=\rho$) and we have $V_u<0$ for every $u\in D$. Furthermore, the quantities
$V_u,u\in D$ are pairwise distinct, $\N_0$ a.e., as a consequence of the fact that local minima of Brownian paths are a.s. distinct (this fact implies that
two local minima of labels that correspond to disjoint segments of the tree $\t_\zeta$ must be distinct).

\begin{lemma}
\label{debut-branching}
$\N_0$ a.e., no branching point is an excursion debut.
\end{lemma}

\begin{proof}
Any
branching point can be represented as $p_\zeta(r)$, where $r\in(s,t)$ and $\zeta_r=\min\{\zeta_{r'}:s\leq r'\leq t\}$, for rationals
$s$ and $t$ such that $0<s<t<\sigma$. Then, for any strict descendant $w$ of $p_\zeta(r)$, the historical path
of $w$ coincides either with $W_s$ or with $W_t$, up to a time (strictly) greater than $\zeta_r$. Since,
conditionally on the lifetime process $\zeta$, $W_s$ is just a Brownian
path over the time interval $[0,\zeta_s]$, it must take values smaller than $W_s(\zeta_r)$ immediately after time $\zeta_r$, a.s.,
and the same holds for $W_t$. We conclude that $p_\zeta(r)$ is a.s. not an excursion debut, and by varying $s$ and $t$ we get 
the desired result outside
a countable union of negligible sets.
\end{proof}

 Let $u$  be an excursion debut. We set
$$C_u=\lbrace w \in \mathcal{T}_{\zeta} : u \prec w \ \text{and} \ V_v >V_u, \ \forall v \in \rrbracket u, w \llbracket \rbrace,$$
where we recall that the notation $v \prec w$ means that $v$ is an ancestor of $w$. Note that $u\in C_u$ and that saying that $u$ is an excursion debut implies that $C_u\not = \{u\}$.
We have clearly $V_w\geq V_u$ for every $w\in C_u$.
Also, if $w\in C_u$, then $w'\in C_u$ for every $w'\in \llbracket u,w\llbracket$. 

\begin{lemma}
\label{setsCu}
$\N_0$ a.e., for every $u\in D$, the
set $C_u$ is a closed subset of $\mathcal{T}_{\zeta}$ and its interior is 
\begin{equation}
\label{interior}
\mathrm{Int}({C}_u)= \lbrace w \in C_u : V_w >V_u \rbrace.
\end{equation}
\end{lemma}

\begin{proof}
The fact that $C_u$ is closed is easy: If $(w_n)$ is a sequence in $C_u$ that converges to $w$ for the metric of $\t_\zeta$, then
we have $u\prec w$ and the ``interval''  $\rrbracket u, w \llbracket $ is contained in the union of the intervals 
$\rrbracket u, w_n \llbracket $. 

To verify \eqref{interior}, first note that the set $\lbrace w \in C_u : V_w >V_u \rbrace$ is open (if $w$ belongs to this set
and if $w'$ is sufficiently close to $w$, then $w'$ is still a descendant of $u$ and 
 $V_v >V_u$ for all $ v \in \rrbracket u, w' \rrbracket$).
 
We also need to check that, if $w \in C_u$ and $V_w=V_u$, then $w$ does not belong to the interior of $C_u$. Consider first the case $w=u$. 
Letting $s_1$ be the first time such that $p_\zeta(s)=u$, the fact that $u$ belongs to the interior of $C_u$ would imply that $\hat W_s\geq \hat W_{s_1}=V_u$
for all $s\geq s_1$ sufficiently close to $s_1$. But then $s_1$ would a point of (right) increase for both $\zeta$ and $\hat W$, and by Lemma 2.2 in \cite{IM} we
know that this cannot occur. Suppose then that $w \in C_u$, $V_w=V_u$ and $w\not =u$. Let $s\in (0,\sigma)$ such that $p_\zeta(s)=w$. Then 
property (ii) of
Lemma \ref{no-local-minimum} holds, and thus property (i) of the same lemma cannot hold. This shows that, for any neighborhood $\mathcal{N}$ of $w$ we can find $w'\in\mathcal{N}$
such that $V_{w'}<V_u$ and therefore $w'\notin C_u$. 
\end{proof}

\begin{proposition}
\label{connec-compo}
$\N_0$ a.e., the sets $\mathrm{Int}({C}_u)$, when $u$ varies in $D$, are exactly the connected components of the open set $\lbrace w \in \t_\zeta : V_w >\min\{V_v:v\in\llbracket \rho,w\rrbracket\rbrace\rbrace$.
\end{proposition}

\begin{proof}
If $w\in \t_\zeta$ is such that $V_w>\min\{V_v:v\in\llbracket \rho,w\rrbracket\rbrace$, then $w\in \mathrm{Int}({C}_u)$, where $u$
is the (unique) ancestor of $w$ such that $V_u=\min\{V_v:v\in\llbracket \rho,w\rrbracket\rbrace$. This shows that $\lbrace w \in \t_\zeta : V_w >\min\{V_v:v\in\llbracket \rho,w\rrbracket\rbrace \rbrace$
is the union of all sets $\mathrm{Int}({C}_u)$, when $u$ varies in $D$. 
Then, if $u\in D$ and $w$ and $w'$ are two vertices in $\mathrm{Int}({C}_u)$, their last common ancestor also belongs to 
$\mathrm{Int}({C}_u)$ (because $u$ is not a branching point, by Lemma \ref{debut-branching}), and the whole interval $\llbracket w,w' \rrbracket$ is contained in $\mathrm{Int}({C}_u)$. 
It follows that, for every $u\in D$, the set $\mathrm{Int}({C}_u)$ is connected. 
 Finally, if $u$ and $u'$ are  two distinct vertices  in $D$, the sets $\mathrm{Int}({C}_u)$ and $\mathrm{Int}({C}_{u'})$ are disjoint. 
 To see this, argue by contradiction and suppose that there exists $v \in  \mathrm{Int}({C}_u) \cap \mathrm{Int}({C}_{u'})$, then $u$ and $u'$ are both ancestors of $v$, hence $u$ is an ancestor of $u'$ (or $u'$ is an ancestor of $u$). 
 However, the properties $u\prec u'\prec v$ 
 and $v\in \mathrm{Int}({C}_u)$ imply that $V_{u'}>V_u$, which contradicts property (2)
 in the definition of
an excursion debut.
\end{proof}

\noindent{\it Remark}. A minor modification of the end of the proof shows in fact that the sets $C_u$, $u\in D$ are pairwise disjoint, which is
slightly stronger.

\smallskip

The last proposition implies that the set $D$ is countable, which can also be seen directly. 

\begin{definition}
\label{height-excu}
If $u$ is an excursion debut, we set
$$M_u:=\sup\{V_v-V_u: v\in C_u\} >0$$
and we call $M_u$ the height of the excursion debut $u$. For every
$\delta>0$, we set $D_\delta:=\{u\in D: M_u>\delta\}$. 
\end{definition}

\begin{lemma}
\label{large-excu}
Let $\delta>0$. The set $D_\delta$ is finite $\N_0$ a.e.
\end{lemma}

\begin{proof}
By a uniform continuity argument, there exists a (random) $\chi>0$ such that, for every $v,v'\in\t_\zeta$,
the condition $d_\zeta(v,v')\leq \chi$ implies $|V_v-V_{v'}|\leq \delta$. 
Then let $u\in D_\delta$, and let $v\in C_u$ such that $V_v-V_u>\delta$.
 We claim that the ball of radius $\chi/2$ centered at $v$ in $\t_\zeta$, which we denote by $B_{d_\zeta}(v,\chi/2)$,
 is contained in $\mathrm{Int}(C_u)$. If the claim holds, the result of the lemma follows since the sets $\mathrm{Int}(C_u)$ are disjoint
 when $u$ varies (Proposition \ref{connec-compo}), and there can be only finitely many values of $v$ such that  
the balls $B_{d_\zeta}(v,\chi/2)$ are disjoint.
  
  To verify our claim, we first note that we must have $d_\zeta(u,v)>\chi$ by our choice of $\chi$, and it follows
  that the ball $B_{d_\zeta}(v,\chi/2)$ is contained in the set of descendants of $u$. Next, if 
  $v'\in B_{d_\zeta}(v,\chi/2)$, we have, for every $w\in\llbracket v,v'\rrbracket$,
  $V_w\geq V_v-\delta>V_u$, showing that $v'\in \mathrm{Int}(C_u)$ since $\llbracket u,v'\rrbracket \subset 
  \llbracket u,v\rrbracket \cup \llbracket v,v'\rrbracket $. This gives our claim and completes the proof.
 \end{proof}

Let $u$ be an excursion debut.
Since $u\in \text{Sk}(\mathcal{T}_{\zeta})$ and $u$   is not a branching point, there are  two uniquely defined times $0< s_1 < s_2<\sigma$ such that $p_\zeta({s_1})=p_\zeta({s_2})=u$. Note that $\hat W_{s_1}=\hat W_{s_2}=V_u$
and $\zeta_{s_1}=\zeta_{s_2}=d_\zeta(\rho,u)$. 
We then define a random snake trajectory $W^{(u)} \in \S_0$ as the image under the
translation $\kappa_{-V_u}$ of the subtrajectory of $\omega$
associated with the interval $[s_1,s_2]$ (recall that the latter subtrajectory corresponds to the
spatial displacements of the descendants of $u$). Note that $W^{(u)}$ has duration $\sigma(W^{(u)})=s_2-s_1$.
Alternatively, the tree-like path corresponding to $W^{(u)}$  is 
$(\zeta_{(s_1+s)\wedge s_2}-\zeta_{s_1},\hat W_{(s_1+s)\wedge s_2}-V_u)_{s\geq 0}$.
By the definition of $D$, each of the paths $W^{(u)}_s$, for $0<s<s_2-s_1$, stays strictly above $0$
during a small interval $(0,\delta)$, for some $\delta>0$. We are in fact not interested in the behavior of these paths 
after they return to $0$ (if they do) and, for this reason, we introduce the truncation of
$W^{(u)}$ at $0$,
$$\tilde W^{(u)}:={\rm tr}_0(W^{(u)}),$$
with the notation introduced in Section \ref{statespace}. We also write $\tilde\zeta^{(u)}_s$
for the lifetime of $\tilde W^{(u)}_s$, for every $s\geq 0$. 
For every $s\in (0,\sigma(\tilde W^{(u)}))$, the path $\tilde W^{(u)}_s$ starts from $0$, stays positive during 
the interval $(0,\tilde\zeta^{(u)}_s)$ and may or may not return to $0$ at time $\tilde\zeta^{(u)}_s$.

It follows from our definitions that the paths $\tilde W^{(u)}_s$, $0\leq s\leq \sigma(\tilde W^{(u)})$, correspond 
to the historical paths after time $d_\zeta(\rho,u)$ of all vertices $v\in C_u$, provided these paths
are shifted by $-V_u$ so that they start from $0$. In particular, $M(\tilde W^{(u)})=M_u$
is the height of the excursion debut $u$. We sometimes call 
$\tilde W^{(u)}$ the excursion above the minimum starting from $u$.

Before stating the main theorem of this section, we introduce one more piece of notation. On the canonical space $\S$, we let $\tilde W=\mathrm{tr}_0(W)$ 
stand for
the truncation at $0$ of the canonical process $(W_s)_{s\geq 0}$.

\begin{theorem}
\label{construct-N}
There exists a $\sigma$-finite measure denoted by $\mathbb{N}_0^*$ on the space $\S$,
which is supported on $\S_0$, 
such that for every nonnegative measurable function $\Phi$ on $\R_+\times\S$, we have 
\begin{equation}
\label{N0*bis}
\mathbb{N}_0\left( \sum_{u \in D} \Phi(V_u,\tilde{W}^{(u)}) \right)= \int_{-\infty}^{0} d\ell\,\int \mathbb{N}_0^*(\mathrm{d}\omega)\,\Phi(\ell,\omega).
\end{equation}
The measure $\N^*_0$ gives finite mass to the set $\S^{(\delta)}:=\{\omega\in\S: \|\omega\|>\delta\}$, for every $\delta>0$. Moreover,
if $G$ is a bounded continuous real function on $\S$, and if there exists $\delta>0$ such that $G$
vanishes on $\S\backslash\S^{(\delta)}$, we have
\begin{equation}
\label{approx-N}
\lim_{\ve\to 0} \frac{1}{\ve}\,\N_\ve\Big(G(\tilde W)\Big)= \N^*_0(G).
\end{equation}
\end{theorem}

The proof of Theorem \ref{construct-N}  relies on an important technical lemma, which we state after introducing some
 notation. We consider a fixed sequence $(\ve_n)_{n\geq 1}$
of positive real numbers converging to $0$. We let $\ve$ be an element of this sequence, then for every $\omega\in\S_0$ and for every
integer $k\geq 1$, we let $\n^\ve_k(\omega)$
be the point measure of excursions of $\omega$ outside $(-k\ve,\infty)$, and we write
$$\n^\ve_k(\omega)=\sum_{i\in I^\ve_k} \delta_{\omega^{k,\ve}_i}.$$
By construction, for every $i\in I^\ve_k$, $\omega^{k,\ve}_i$
is a subtrajectory of $\omega$, and we write $[r_i^{k,\ve},s_i^{k,\ve}]$ for the
corresponding interval.
We will also use the notation $\tilde\omega^{k,\ve}_i$ for $\omega^{k,\ve}_i$  translated so that its starting point is $\ve$ and then truncated at level $0$: with the notation of Section \ref{statespace}, 
$\tilde\omega^{k,\ve}_i= {\rm tr}_{0}\circ \kappa_{(k+1)\ve} (\omega^{k,\ve}_i)\in \S_\ve$. 

Recall our notation $Z_a$ for the total mass of the
exit measure from $(-a,\infty)$. By the special Markov property (Proposition \ref{SMP}), we know that the conditional distribution of 
$\n^\ve_k$ under $\N_0(\cdot\mid Z_{k\ve}\not = 0)$ and given 
$Z_{k\ve}$ is that of a Poisson point measure with intensity
$$Z_{k\ve}\,\N_{-k\ve}(\cdot).$$
On the other hand we have $\n^\ve_k(\omega)=0$, $\N_0$ a.e. on $\{Z_{k\ve}=0\}$. 

\begin{lemma}
\label{construct-tech}
The following properties hold $\N_0$ a.e.
Let $u\in D$, and let $0<s_1<s_2<\sigma$ be determined by $p_\zeta(s_1)=p_\zeta(s_2)=u$. Then, 
for every sufficiently small $\ve$ in the sequence $(\ve_n)_{n\geq 1}$, if $k_{u,\ve}\geq 1$ is the integer 
determined by $-(k_{u,\ve}+1)\ve<V_u\leq -k_{u,\ve}\ve$, there exists a unique index $i_{u,\ve}\in I^\ve_{k_{u,\ve}}$ such that
$$(s_1,s_2)\subset (r_{i_{u,\ve}}^{k_{u,\ve},\ve},s_{i_{u,\ve}}^{k_{u,\ve},\ve}),$$
and we have
$$\tilde\omega^{k_{u,\ve},\ve}_{i_{u,\ve}} \longrightarrow \tilde W^{(u)}$$
as $\ve\to 0$ along the sequence $(\ve_n)_{n\geq 1}$.
\end{lemma}

\noindent{\it Remark.} The convergence in the last assertion of the lemma holds in $\S$, noting
that $\tilde W^{(u)}\in\S_0$ whereas $\tilde\omega^{k_{u,\ve},\ve}_{i_{u,\ve}}\in \S_\ve$. 

\begin{proof}
Note that a priori we could
have $k_{u,\ve}= 0$, but this does not occur for $\ve$ small enough since $V_u<0$. Then the index 
$i_{u,\ve}$ is determined by the fact that the excursion $\omega^{k_{u,\ve},\ve}_{i_{u,\ve}} $ corresponds to
the descendants of the first ancestor of $u$ at spatial position $-k_{u,\ve}\ve$. More specifically, the
index $i_{u,\ve}$ is determined by 
\begin{equation}
\label{construct-tech0}
r_{i_{u,\ve}}^{k_{u,\ve},\ve}=\sup\{s\leq s_1: \zeta_s\leq \tau_{-k_{u,\ve}\ve}(W_{s_1})\}
\end{equation}
where we recall the notation $\tau_a(\w)=\inf\{t\geq 0: \w(t)=a\}$. Since the image under  $p_\zeta$ of the interval $(r_{i_{u,\ve}}^{k_{u,\ve},\ve},s_{i_{u,\ve}}^{k_{u,\ve},\ve})$
corresponds to descendants of an ancestor of $u$, the inclusion
$$(s_1,s_2)\subset (r_{i_{u,\ve}}^{k_{u,\ve},\ve},s_{i_{u,\ve}}^{k_{u,\ve},\ve})$$
is immediate. For the last property of the lemma, we first verify that
\begin{equation}
\label{construct-tech1}
r_{i_{u,\ve}}^{k_{u,\ve},\ve}\la s_1\ ,\quad s_{i_{u,\ve}}^{k_{u,\ve},\ve}\la s_2
\end{equation}
as $\ve\to 0$ along the sequence $(\ve_n)_{n\geq 1}$. 

To this end, let $s$ be such that $0<s<s_1$,
and observe that we have then
$$\inf_{r\in[s,s_1]} \zeta_r < \zeta_{s_1}$$
(otherwise $u$ would be a branching point). On the other hand, for any $\gamma>0$, there exists $\chi>0$
such that $W_{s_1}(t)\geq V_u+ \chi$ if $0\leq t\leq \zeta_{s_1}-\gamma$ (by
property (2) of the definition of an excursion debut, and the fact that a Brownian path
cannot have two local minima at the same level). It follows that $\tau_{-k_{u,\ve}\ve}(W_{s_1})\la \zeta_{s_1}$
as $\ve\to 0$, and together with \eqref{construct-tech0}
the preceding observations imply that $r_{i_{u,\ve}}^{k_{u,\ve},\ve}>s$ for $\ve$ small enough, giving the
desired convergence $r_{i_{u,\ve}}^{k_{u,\ve},\ve}\la s_1$. The proof of the other convergence $s_{i_{u,\ve}}^{k_{u,\ve},\ve}\la s_2$
is analogous. 

Once we have obtained the convergences \eqref{construct-tech1}, we deduce the last assertion of the lemma
from Lemma \ref{continuity-trunc}. With the notation of this lemma, we take $\omega'=W^{(u)}$
and $\omega^{(n)}=\kappa_{-V_u}(\omega^{k_n,\ve_n}_{i_n})$, where we write $k_n=k_{u,\ve_n}$
and $i_n=i_{u,\ve_n}$ to simplify notation. We also take $\delta_n= -(k_n+1)\ve_n -V_u\in (-\ve_n,0)$. The conclusion 
of the lemma then yields the fact that $\mathrm{tr}_{\delta_n}(\omega^{(n)})$ converges to $\mathrm{tr}_0(\omega')=\tilde W^{(u)}$.
This is the result we need since one easily checks that $\mathrm{tr}_{\delta_n}(\omega^{(n)})$ coincides 
with $\tilde\omega^{k_{n},\ve_n}_{i_{n}}$ translated by $\delta_n$. We still need to verify that
assumptions (i)--(iii) of Lemma \ref{continuity-trunc} hold with our choice of $\omega'$.
Assumptions (i) and (ii) hold by the definition of an excursion debut. Assumption (iii)
holds because otherwise this would mean that there are two distinct  
local minimum times corresponding to the same local minimum of a path $W_s$, which
is impossible. This completes the proof of the last assertion of the lemma.
\end{proof}

\begin{proof}[Proof of Theorem \ref{construct-N}] In order to prove the first part of the theorem, it is enough to 
construct the $\sigma$-finite measure $\mathbb{N}^*_0$ such that the identity \eqref{N0*bis} holds
whenever $\Phi(\ell,\omega)=g(\ell)\,G(\omega)$, where $g$ and $G$ are nonnegative measurable
functions defined on $\R$ and on $\S$ respectively. 
We fix two such functions $g$ and $G$, 
and, in a first step, we assume that both $g$ and $G$ are bounded and continuous and 
take nonnegative values. Moreover, we assume that $g$ is nontrivial and is supported on a compact subinterval
of $(-\infty,0)$, and that there exists $\delta >0$ such that $G(\omega)=0$ if $\omega\notin \S^{(\delta)}$. The functions 
$G$ and $g$ will be fixed until the last lines of the proof, where we explain how to get rid of the extra assumptions on $G$ and $g$. 

By our assumptions on $G$,
the quantity
$ G(\tilde{W}^{(u)})$ is zero if $u\notin D_\delta$, and a fortiori if $u\notin D_{\delta/2}$. 
Since $D_{\delta/2}$ is a.e. finite (Lemma \ref{large-excu})
we get, using the notation and the conclusion of Lemma \ref{construct-tech},
$$ \sum_{u \in D} g(V_u) G(\tilde{W}^{(u)}) = \sum_{u \in D_{\delta/2}} g(V_u) G(\tilde{W}^{(u)})=\lim_{\ve \to 0}
 \sum_{u \in D_{\delta/2}} g(-\ve k_{u,\ve} ) G(\tilde\omega^{k_{u,\ve},\ve}_{i_{u,\ve}} ),$$
 $\N_0$ a.e. (here and in the remaining part of the proof, we consider only values of $\ve$
 in the sequence $(\ve_n)_{n\geq 1}$, even if this is not mentioned explicitly). We next observe that we have
\begin{equation}
\label{sum-ident}
\sum_{u \in D_{\delta/2}} g(-\ve k_{u,\ve} ) G(\tilde\omega^{k_{u,\ve},\ve}_{i_{u,\ve}} )
 = \sum_{k=1}^\infty \sum_{i\in I^\ve_k} g(-\ve k)\,G(\tilde \omega^{k,\ve}_i)
 \end{equation}
for $\ve$ small enough, $\N_0$ a.e. 
 To see this, suppose that 
 $\ve<\delta/2$, and fix $k\geq 1$ and $i\in I^\ve_k$. If $G(\tilde \omega^{k,\ve}_i)\not =0$, there exists 
a real $s\geq 0$ such that the path $W_s(\tilde \omega^{k,\ve}_i)$ hits level $\delta$. This also means that there exists a real $s'\geq 0$
 such that the path $W_{s'}(\omega^{k,\ve}_i)$ hits $-(k+1)\ve+\delta$ before hitting $-(k+1)\ve$, and we can take the smallest such real  $s'$. Let $s''$
 such that  $W_{s''}(\omega^{k,\ve}_i)$ coincides with $W_{s'}(\omega^{k,\ve}_i)$ truncated at the (unique) time where it reaches 
 its minimum before hitting $-(k+1)\ve+\delta$ (in the tree coded by $\zeta(\omega^{k,\ve}_i)$, $s''$ corresponds to the unique ancestor
 with minimal spatial position
 of the vertex $s'$). Then it follows from our definitions that $u:=p_\zeta(r^{k,\ve}_i+s'')$ is an excursion
 debut, with $k_{u,\ve}=k$ and $i_{u,\ve}=i$ by construction, and the height of $u$ is at least 
 $\delta-\ve>\delta/2$, so that $u\in D_{\delta/2}$. Thus any (nonzero) term appearing in the right-hand side
 of \eqref{sum-ident} also appears, at least once, in the left-hand side. To complete the proof of \eqref{sum-ident}, we must still verify that,
 for $\ve$ small enough,
 no (nonzero) term in the right-hand side appears twice in the left-hand side. But this follows from the fact
 that the values of $V_u$ for $u\in D$ are all distinct: since
 $D_{\delta/2}$ is finite, for $\ve$ small enough, there cannot be two distinct elements $u$, $u'$ of $D_{\delta/2}$ 
 such that $V_u$ and $V_{u'}$ lie in the same interval $(-(k+1)\ve,-k\ve)$. 
 
 From the preceding considerations, we get that
 $$ \sum_{u \in D} g(V_u) G(\tilde{W}^{(u)}) = \lim_{\ve\to 0} 
 \sum_{k=1}^\infty \sum_{i\in I^\ve_k} g(-\ve k)\,G(\tilde \omega^{k,\ve}_i),$$
$\N_0$ a.e. We then notice that we can fix $\chi >0$
such that $g(x)=0$ if $x\geq -\chi$, and restrict our attention to the set
$\{W_*\leq -\chi\}$, which has finite $\N_0$-measure. 
The next step is to deduce from the preceding convergence that we have also
\begin{equation}
\label{key-conv-N*}
\N_0\Bigg(\sum_{u \in D} g(V_u) G(\tilde{W}^{(u)})\Bigg) = \lim_{\ve\to 0} 
 \N_0\Bigg(\sum_{k=1}^\infty \sum_{i\in I^\ve_k} g(-\ve k)\,G(\tilde \omega^{k,\ve}_i)\Bigg).
 \end{equation}
 For this, some uniform integrability is needed. For every integer $k\geq 1$, set
 $$n^\ve_k:= \sum_{i\in I^\ve_k} \mathbf{1}_{\{ \|\tilde \omega^{k,\ve}_i\| >\delta\}}.$$
 Recalling our assumptions on $g$ and $G$, we see that in order to deduce 
 \eqref{key-conv-N*} from the preceding convergence, it suffices to verify that,
 for $p\in (1,3/2)$, and for every $A>\chi$,
\begin{equation}
\label{key-conv-N*-tech}
\N_0\Bigg[\Bigg( \sum_{k= \lfloor \chi/\ve\rfloor+1}^{\lfloor A/\ve\rfloor} n^\ve_k\Bigg)^p\Bigg]
\end{equation}
 is bounded independently of $\ve$. By the special Markov property (Proposition \ref{SMP}), conditionally 
 on the $\sigma$-field $\mathcal{E}^{(-k\ve,\infty)}$, $n^\ve_k$ is Poisson with intensity
$c_{\ve,\delta}\,Z_{k\ve}$, 
where
 $c_{\ve,\delta}=\N_\ve(\|\tilde W\|>\delta)$. In particular,
 $$\N_0\Big((n^\ve_k-c_{\ve,\delta}\,Z_{k\ve})^2\,\Big|\,\mathcal{E}^{(-k\ve,\infty)}\Big)= c_{\ve,\delta}\,Z_{k\ve}$$
 and 
 $$\mathcal{M}^\ve_k:=\sum_{j= \lfloor \chi/\ve\rfloor+1}^k (n^\ve_j-c_{\ve,\delta}\,Z_{j\ve})\;,\quad k\geq \lfloor \chi/\ve\rfloor$$
 is a martingale with respect to the filtration $(\mathcal{E}^{(-(k+1)\ve,\infty)})_{ k\geq \lfloor \chi/\ve\rfloor}$ -- note that, by the
 construction of the truncated excursions $\tilde \omega^{k,\ve}_i$, $n_k^\ve$ is $\mathcal{E}^{(-(k+1)\ve,\infty)}$-measurable. 
 The discrete Burkholder-Davis-Gundy inequalities (see e.g. \cite[Th\'eor\`eme 5]{Mey}) now give, for $p\in (1,3/2)$ and for some constant $K_{(p)}$
 depending only on $p$,
 \begin{align}
 \label{loi-max-tech07}
 \N_0\Big(|\mathcal{M}^\ve_{\lfloor A/\ve\rfloor}|^p\Big)
& \leq K_{(p)}\,\N_0\Bigg[\Bigg(\sum_{j= \lfloor \chi/\ve\rfloor+1}^{\lfloor A/\ve\rfloor} (\mathcal{M}^\ve_j-\mathcal{M}^\ve_{j-1})^2
 \Bigg)^{p/2}\Bigg]\\
& \leq K_{(p)}\,\N_0(M_*<-\chi)^{1-p/2}\,\Bigg(\N_0\Bigg[\sum_{j= \lfloor \chi/\ve\rfloor+1}^{\lfloor A/\ve\rfloor} (\mathcal{M}^\ve_j-\mathcal{M}^\ve_{j-1})^2
\Bigg]\Bigg)^{p/2}\nonumber\\
&=K_{(p)}\,\N_0(M_*<-\chi)^{1-p/2}\,c_{\ve,\delta}^{p/2}\, 
\Bigg(\N_0\Bigg[\sum_{j= \lfloor \chi/\ve\rfloor+1}^{\lfloor A/\ve\rfloor} Z_{j\ve}\Bigg]\Bigg)^{p/2}\nonumber\\
&= K_{(p)}\,\N_0(M_*<-\chi)^{1-p/2}\,c_{\ve,\delta}^{p/2}\,(\lfloor A/\ve\rfloor-\lfloor \chi/\ve\rfloor)^{p/2},\nonumber
\end{align} 
using Jensen's inequality (with respect to the probability measure $\N_0(\cdot\mid W_*<-\chi)$) in the second line, 
and in the last line the fact that $\N_0(Z_r)=1$ for every $r>0$ (see \eqref{first-moment-exit}). 

Then observe that
$$c_{\ve,\delta}=\N_\ve(\|\tilde W\|>\delta)=\N_\ve(M(\tilde W)>\delta)=\N_\ve(\langle \z^{(0,\delta)},\mathbf{1}_{\{\delta\}}\rangle >0)
=\N_0(\langle \z^{(-\ve,\delta-\ve)},\mathbf{1}_{\{\delta-\ve\}}\rangle >0),$$
where the third equality follows from the special Markov property (Proposition \ref{SMP}). 
It follows from formula (9) in \cite[Section 4]{Cactus}, together with a monotonicity argument, that
\begin{equation}
\label{loi-max-tech}
\lim_{\ve\to 0} \ve^{-1}\,c_{\ve,\delta}= c_0\,\delta^{-3},
\end{equation}
where $c_0$ is a positive constant (made explicit in Lemma \ref{loi-max} below).
In particular, there 
exists a constant $c_\delta<\infty$ such that
$c_{\ve,\delta}\leq c_\delta\,\ve$ for every $\ve<\delta/2$. From
\eqref{loi-max-tech07}, we then get that the quantities $\N_0(|\mathcal{M}^\ve_{\lfloor A/\ve\rfloor}|^p)$ are uniformly bounded when $\ve<\delta/2$.  Finally, we write
$$ \sum_{k= \lfloor \chi/\ve\rfloor+1}^{\lfloor A/\ve\rfloor} n^\ve_k= \mathcal{M}^\ve_{\lfloor A/\ve\rfloor}
+ c_{\ve,\delta}\sum_{k= \lfloor \chi/\ve\rfloor+1}^{\lfloor A/\ve\rfloor} Z_{k\ve}$$
and we use again the bound $c_{\ve,\delta}\leq c_\delta\,\ve$ together with the fact that the random variables
$Z_a$, $0< a\leq A$ are bounded in $L^p(\N_0)$ when $1<p<3/2$ (Lemma \ref{exit-moments}). This gives us the desired bound for the
quantities in \eqref{key-conv-N*-tech}, and justifies the passage to the limit under the integral in \eqref{key-conv-N*} -- incidentally
this also shows that the left-hand side of  \eqref{key-conv-N*} is a finite quantity.

We then use the special Markov property once again to obtain
$$ \N_0\Bigg(\sum_{k=1}^\infty \sum_{i\in I^\ve_k} g(-\ve k)\,G(\tilde \omega^{k,\ve}_i)\Bigg)
= \sum_{k=1}^\infty g(-\ve k) \N_0\big(Z_{-k \ve}\,\N_\ve(G(\tilde W))\big)
=\Bigg( \sum_{k=1}^\infty g(-\ve k)\Bigg)\,\N_\ve(G(\tilde W)),$$
where the last equality holds because $\N_0(Z_{-k \ve})=1$, by (\ref{first-moment-exit}). 
Now note that 
$$\ve \sum_{k=1}^\infty g(-\ve k) \build{\la}_{\ve\to 0}^{} \int_{-\infty}^0 g(x)\,\mathrm{d}x,$$
and so we deduce from   \eqref{key-conv-N*} and the preceding two displays that
$$\ve^{-1}\N_\ve(G(\tilde W)) \build{\la}_{\ve\to 0}^{} K_{G}$$
where the limit $K_G<\infty$ is such that
$$\N_0\Bigg(\sum_{u \in D} g(V_u) G(\tilde{W}^{(u)})\Bigg) = K_G\,\int_{-\infty}^0 g(x)\,\mathrm{d}x.$$
We now set, for every measurable subset $F$ of $\S$,
\begin{equation}
\label{definiN*}
\N^*_0(F):= \frac{\N_0\Big(\sum_{u \in D} g(V_u) \mathbf{1}_F(\tilde{W}^{(u)})\Big)}{\int_{-\infty}^0 g(x)\,\mathrm{d}x}.
\end{equation}
This defines a positive measure on $\S$, which is
supported on $\S_0$ since $W^{(u)}\in \S_0$ for every $u\in D$. 
Furthermore, we have
$$\N^*_0(G)=K_G<\infty.$$
Noting that the definition (\ref{definiN*}) of $\N^*_0$ does not involve the choice of $G$, this
implies that the sets $\S^{(\delta)}$ have a finite
$\N^*_0$-measure. Since it is clear that $\N^*_0(\|\omega\|=0)=0$, we get that $\N^*_0$ is $\sigma$-finite. Furthermore, we have also
$$\N^*_0(G)=\lim_{\ve \to 0} \ve^{-1}\N_\ve(G(\tilde W)),$$
which gives  \eqref{approx-N} for the function $G$ we had fixed, and then also for any function $G$ satisfying the same assumptions, since
(\ref{definiN*}) does not depend on the choice of $G$
(note that we considered a fixed sequence of values of $\ve$, but the same would hold
for any such sequence). 

Finally, the last display shows that $\N^*_0$ does not depend on the choice of $g$, since a measure on $\S$ supported on $\S_0$ and which is finite 
on the sets $\S^{(\delta)}$ and puts no mass on $\{\omega:\|\omega\|=0\}$ is determined by its values against
functions $G$ satisfying the assumptions of the beginning of the proof. 
By 
(\ref{definiN*}), formula (\ref{N0*bis}) holds if $\Phi(\ell,\omega)=g(\ell)G(\omega)$,  when $G$ is an indicator function and the function $g$ satisfies the previous assumptions. By standard
monotone class arguments, it holds when $\Phi(\ell,\omega)=g(\ell)G(\omega)$, for any nonnegative measurable functions $g$ and $G$. 
This completes the proof.
\end{proof}

Recall the notation $M(\omega)=\sup\{\omega_s(t):s\geq 0, 0\leq s\leq \zeta_s\}$. We also set 
$$\tilde M=M(\tilde W)=\sup\{W_s(t):s\geq 0, t\leq \zeta_s\wedge \tau^*_0(W_s)\}.$$
We can derive the
distribution of $M$ under $\N^*_0$.

\begin{lemma}
\label{loi-max}
For every $\delta>0$, we have
$$\N^*_0(M>\delta)= c_0\,\delta^{-3},$$
where the constant $c_0$ is given by
$$c_0= 3\,\pi^{-3/2}\,\Gamma(\frac{1}{3})^3\Gamma(\frac{7}{6})^3.$$
\end{lemma}

\begin{proof}
By (\ref{loi-max-tech}), we have
\begin{equation}
\label{loi-max-tech2}
\lim_{\ve \to 0} \ve^{-1} \N_\ve(\tilde M>\delta) = c_0\,\delta^{-3}
\end{equation}
and the value of the constant $c_0$ is determined in \cite[Section 4]{Cactus}. On the other hand,
we know that
$$\N^*_0(G)=\lim_{\ve \to 0} \ve^{-1}\N_\ve(G(\tilde W)),$$
for any bounded continuous function $G$ vanishing on the complement 
of $\S^{(\chi)}$ for some $\chi>0$. 
Noting that the limit in \eqref{loi-max-tech2} depends continuously
on $\delta$, we can approximate the indicator function of the
 set $\{M>\delta\}$ by such functions $G$, and obtain 
 $$\N^*_0(M>\delta)=\lim_{\ve \to 0} \ve^{-1} \N_\ve(\tilde M>\delta) = c_0\,\delta^{-3}.$$
 This completes the proof.
\end{proof}

We may now restate the last assertion of Theorem \ref{construct-N} in a way more
suitable for our applications.

\begin{corollary}
\label{approx-*}
Let $\delta>0$. As $\ve\to 0$, the law of $\tilde W$ under $\N_\ve(\cdot\mid \tilde M>\delta)$
converges weakly to $\N^*_0(\cdot \mid M>\delta)$.
\end{corollary}

\begin{proof}
Let $G$ be bounded and continuous on $\S$ and such that $G(\omega)=0$
if $\omega\notin \S^{(\delta)}$. Then, for $\ve\in(0,\delta)$,
$$\N_\ve\Big(G(\tilde W)\,\Big|\, \tilde M>\delta\Big)
=\frac{\N_\ve(G(\tilde W))}{\N_\ve(\tilde M>\delta)}\ 
\build{\la}_{\ve\to0}^{} \ \frac{\N^*_0(G)}{\N^*_0(M>\delta)}= \N^*_0\Big(G\,\Big|\, M>\delta\Big),$$
using \eqref{approx-N}, \eqref{loi-max-tech2}, and Lemma \ref{loi-max}. The desired result follows.
\end{proof}

We conclude this section by deriving a useful scaling property of $\N^*_0$. For $\lambda>0$, for every 
$\omega\in \S$, we define $\theta_\lambda(\omega)\in \S$
by $\theta_\lambda(\omega)=\omega'$, with
$$\omega'_s(t)= \sqrt{\lambda}\,\omega_{s/\lambda^2}(t/\lambda)\;,\qquad
\hbox{for } s\geq 0,\;0\leq t\leq \zeta'_s=\lambda\zeta_{s/\lambda^2}.$$
Note that, for every $x\geq 0$, $\theta_\lambda(\N_x)= \lambda\, \N_{x\sqrt{\lambda}}$. The measure $\N^*_0$ enjoys a similar scaling property.

\begin{lemma}
\label{scaling-N*}
For every $\lambda>0$, $\theta_\lambda(\N^*_0)= \lambda^{3/2}\, \N^*_0$.
\end{lemma}

\begin{proof} Let $G$ be a function on $\S$ satisfying the conditions required for \eqref{approx-N}. Then,
\begin{align*}
\N^*_0(G)&=\lim_{\ve \to 0} \ve^{-1}\N_\ve(G(\tilde W))\\
&=\lim_{\ve \to 0} \ve^{-1}\lambda^{-1}\N_{\ve/\sqrt{\lambda}}(G(\theta_\lambda(\tilde W)))\\
&=\lim_{\ve \to 0} \lambda^{-3/2}\times (\ve/\sqrt{\lambda})^{-1}\N_{\ve/\sqrt{\lambda}}(G(\theta_\lambda(\tilde W)))\\
&=\lambda^{-3/2}\,\N^*_0(G\circ \theta_\lambda)
\end{align*}
giving the desired result.
\end{proof}

\section{The re-rooting representation}
\label{sec:re-root}

In this section, we provide a formula connecting the measures $\N_0$ and $\N^*_0$ via a re-rooting technique.
We first need to introduce some notation. 

Recall the re-rooting operator $R_s$ from Section \ref{statespace}. For every $\omega\in\S_0$, for every $s\in[0,\sigma(\omega)]$, we set
$$W^{[s]}(\omega)=\kappa_{-\hat W_s(\omega)}\circ R_s(\omega).$$
In other words,
$W^{[s]}(\omega)$ is just $\omega$ re-rooted at $s$ and then shifted 
so that the spatial position of the root is again $0$. Note that we slightly abuse notation here because it would have been more consistent with the notation of 
Section \ref{statespace} to take $W^{[s]}(\omega)=R_s(\omega)$.

\begin{theorem}
\label{reenracinement}
For every nonnegative measurable function $G$ on $\S$, the following equality holds.
$$ \mathbb{N}_0^*\left(\int_0^{\sigma} \mathrm{d}r \,G(W^{[r]})\right)= 2\,\mathbb{N}_0\left( \int_0^{\infty} \mathrm{d}b \,G({\rm tr}_{-b}(W)) Z_b\right)$$
where we recall that $Z_b$ stands for the total mass of the exit measure outside $(-b, \infty)$. 
\end{theorem}

\begin{proof}
We start from the re-rooting theorem in \cite[Theorem 2.3]{LGWeill}.
For every nonnegative measurable function $F$ on $\mathbb{R}_+ \times \S$,
\begin{equation}
\label{reenracinementLGWeill}
\mathbb{N}_0\left( \int_0^{\sigma} \mathrm{d}s\, F(s, W^{[s]})\right)= \mathbb{N}_0\left(\int_0^{\sigma} \mathrm{d}s\, F(s,W)\right)
\end{equation}

We apply this result to a function $F$ of the form
$$ F(s,\omega)=G\big({\rm tr}_{\underline{\omega}_{\sigma-s}}(\omega)\big)\,g(\underline{\omega}_{\sigma-s}-\hat\omega_{\sigma-s}),$$
where we recall the notation $\underline{\w}=\min\{\w(t):0\leq t\leq \zeta_{(\w)}\}$, and we suppose that $G$ and $g$ satisfy the assumptions stated at the beginning of the proof of
Theorem \ref{construct-N}, and the additional assumption that there exists a constant $K>0$
such that $G(\omega)=0$ if $\|\omega\|\geq K$. We note that
our definitions give under $\N_0$,
\begin{align*}
\hat W_{\sigma-s}^{[s]} &= - \hat W_s\\
\underline{W}_{\sigma-s}^{[s]}&= \underline{W}_s-\hat W_s.
\end{align*}
Consequently, we have
$$F(s, W^{[s]})= G\big({\rm tr}_{\underline{W}_s-\hat W_s}(W^{[s]})\big)\,g(\underline{W}_s).$$
We can then decompose the integral
$$\int_0^{\sigma} \mathrm{d}s\, F(s, W^{[s]})$$
as a sum over the sets $\{s\in[0,\sigma]: p_\zeta(s)\in C_u\}$ where $u$ varies over $D$. These sets cover $[0,\sigma]$ 
(except for a Lebesgue negligible subset) and they are pairwise disjoint. Furthermore, if $u\in D$, it follows from our
definitions that we have $\underline{W}_s=V_u$ for every  $s\in[0,\sigma]$ such that $p_\zeta(s)\in C_u$,
and 
$$\int_{\{s\in[0,\sigma]: p_\zeta(s)\in C_u\}} \mathrm{d}s\, G\big({\rm tr}_{\underline{W}_s-\hat W_s}(W^{[s]})\big)
= H(\tilde W^{(u)}),$$
where
$$H(\omega)=\int_0^{\sigma(\omega)} \mathrm{d}r\,G(W^{[r]}(\omega)).$$
Summarizing, the left-hand side of \eqref{reenracinementLGWeill} is equal to
\begin{equation}
\label{re-root-tech1}
\N_0\Bigg(\sum_{u\in D} g(V_u)\,H(\tilde W^{(u)})\Bigg) = \Big(\int_{-\infty}^0 g(x)\,\mathrm{d}x\Big)\,\N^*_0(H)
\end{equation}
by Theorem \ref{construct-N}. 

On the other hand, the right-hand side of \eqref{reenracinementLGWeill} is equal
to 
$$\N_0\Big(\int_0^\sigma \mathrm{d}s\,G({\rm tr}_{\underline{W}_s}(W))\,g(\underline{W}_s-\hat W_s)\Big).$$
We can evaluate this quantity via a discrete approximation. Using Lemma \ref{conti-trunc}, we have $\N_0$ a.e.
$$\int_0^\sigma \mathrm{d}s\,G({\rm tr}_{\underline{W}_s}(W))\,g(\underline{W}_s-\hat W_s)
=\lim_{n\to \infty} \int_0^\sigma \mathrm{d}s\,g(\underline{W}_s-\hat W_s)
\sum_{k=1}^\infty \mathbf{1}_{\{\underline{W}_s\in (-(k+1)/n,-k/n]\}}\,G({\rm tr}_{-k/n}(W)),$$
and we note that, if $g$ is supported on $[-A,0]$, the quantities in the 
right-hand side are bounded independently of $n\geq 1$ by a constant times
$$\int_0^\sigma \mathrm{d}s\,\mathbf{1}_{\{\underline{W}_s\geq -K-1\}}\mathbf{1}_{\{\underline{W}_s-\hat W_s\geq -A\}}.$$
The point is that if $s\in [0,\sigma]$ is such that $\underline{W}_s<-K-1$, then the unique integer $k$ such that 
$\underline{W}_s\in (-(k+1)/n,-k/n]$ also satisfies $-k/n<-K$ and we have $G({\rm tr}_{-k/n}(W))=0$ by our assumption on $G$. 
The quantity in the last display is integrable under $\N_0$ as a simple application of the first-moment formula for the Brownian snake
\eqref{first-moment}. 
This makes it possible to use dominated convergence and to get that
\begin{align}
\label{re-root-limit}
&\N_0\Big(\int_0^\sigma \mathrm{d}s\,G(\mathrm{tr}_{\underline{W}_s}(W))\,g(\underline{W}_s-\hat W_s)\Big)\\
&=\lim_{n\to \infty} \sum_{k=1}^\infty
\N_0\Big(\int_0^\sigma \mathrm{d}s\,g(\underline{W}_s-\hat W_s)
\, \mathbf{1}_{\{\underline{W}_s\in (-(k+1)/n,-k/n]\}}\,G({\rm tr}_{-k/n}(W))\Big).\nonumber
\end{align}
Then, for every integer $k\geq 1$, an application of the special Markov property (note that 
$G({\rm tr}_{-k/n}(W))$ is $\mathcal{E}^{(-k/n,\infty)}$-measurable by the very definition of this $\sigma$-field) gives 
\begin{align*}
&\N_0\Big(\int_0^\sigma \mathrm{d}s\,g(\underline{W}_s-\hat W_s)
\, \mathbf{1}_{\{\underline{W}_s\in (-(k+1)/n,-k/n]\}}\,G({\rm tr}_{-k/n}(W))\Big)\\
&=\N_0\Bigg(Z_{k/n}\,G({\rm tr}_{-k/n}(W))\,\N_{-k/n}\Big(\int_0^\sigma\mathrm{d}s\,\mathbf{1}_{\{\underline{W}_s>-(k+1)/n\}}
\,g(\underline{W}_s-\hat W_s)\Big)\Bigg)\\
&=\N_0\big(Z_{k/n}\,G({\rm tr}_{-k/n}(W))\big)\times\N_{-k/n}\Big(\int_0^\sigma\mathrm{d}s\,\mathbf{1}_{\{\underline{W}_s>-(k+1)/n\}}
\,g(\underline{W}_s-\hat W_s)\Big)\\
&=\N_0\big(Z_{k/n}\,G({\rm tr}_{-k/n}(W))\big)\,\E_{-k/n}\Big[\int_0^\infty \mathrm{d}t\,
\mathbf{1}_{\{\min\{B_r:0\leq r\leq t\}>-(k+1)/n\}}\,g(\min\{B_r:0\leq r\leq t\}-B_t)\Big]\\
&=\frac{2}{n}\,\Big(\int_{-\infty}^0 \mathrm{d}x\,g(x)\Big)\,\N_0\big(Z_{k/n}\,G({\rm tr}_{-k/n}(W))\big),
\end{align*}
using again the first-moment formula for the Brownian snake \eqref{first-moment} in the third equality, and in the
last one the property
$$\E_0\Big[\int_0^\infty \mathrm{d}t\, \mathbf{1}_{\{\min\{B_r:0\leq r\leq t\}>-\ve\}}\,g(\min\{B_r:0\leq r\leq t\}-B_t)\Big] = 2\ve\int_{-\infty}^0  \mathrm{d}x\,g(x),$$
which holds for every $\ve>0$, by direct
calculations since the law of $(B_t,\min\{B_r:0\leq r\leq t\})$ is known explicitly (or via a simple application of standard excursion theory).
From \eqref{re-root-limit}, we then deduce that
\begin{align*}
\N_0\Big(\int_0^\sigma \mathrm{d}s\,G(\mathrm{tr}_{\underline{W}_s}(W))\,g(\underline{W}_s-\hat W_s)\Big)
&=\lim_{n\to\infty}\frac{2}{n}\,\Big(\int_{-\infty}^0 \mathrm{d}x\,g(x)\Big)\,
\sum_{k=1}^\infty \N_0\big(Z_{k/n}\,G({\rm tr}_{-k/n}(W))\big)\\
&=2\,\Big(\int_{-\infty}^0 \mathrm{d}x\,g(x)\Big)\,\N_0\Big(\int_0^\infty \mathrm{d}b\,Z_b\,G({\rm tr}_{-b}(W))\Big),
\end{align*}
where the last equality is justified by Lemma \ref{conti-trunc} together with our assumptions on $G$ 
and the integrability properties of the
exit measure process $Z$ that were already used in the proof of Theorem \ref{construct-N}. 

Finally, the equality
between the right-hand side of the last display and the right-hand side of \eqref{re-root-tech1} gives the
identity of the theorem under our special assumptions on $G$. However, since both sides 
of this identity define $\sigma$-finite measures (which are finite on sets of the form $\{\delta < \|\omega\| < K\}$),
the fact that these measures take the same values on the particular functions $G$ considered in the proof implies
that they are equal. 
\end{proof}

\section{An almost sure construction}
\label{a.s.cons}

In this section, we fix $\delta>0$ and we give an almost sure construction of 
a snake trajectory distributed according to $\N^*_0(\cdot \mid M>\delta)$. This 
construction will be useful later when we discuss exit measures.

Let $0<\ve<\ve'<\delta$, and let $W^{\delta,\ve}$ be a random snake trajectory 
 distributed according to $\N_\ve(\cdot\mid \tilde M >\delta)$. 
Consider  the excursions of $W^{\delta,\ve}$ outside the interval $(0,\ve')$.
The conditioning on $\{\tilde M >\delta\}$ implies that there is at least one
such excursion $\omega'$ starting from $\ve'$ and such that $\tilde M(\omega')>\delta$. Furthermore, 
if we pick uniformly at random one of the excursions $\omega'$ starting from $\ve'$ that satisfy
$\tilde M(\omega')>\delta$, the special Markov property (Proposition \ref{SMP}) ensures that this excursion will be
distributed according to $\N_{\ve'}(\cdot\mid \tilde M >\delta)$. For $\omega\in\S_\ve$ such that $\tilde M(\omega)>\delta$,
let $\Theta_{\ve,\ve'}(\omega,\mathrm{d}\omega')$ be the probability measure on $\S_{\ve'}$ 
defined as the law of an excursion of $\omega$ outside $(0,\ve')$ chosen uniformly at random among those
excursions that satisfy $\tilde M>\delta$. Then, the preceding considerations show that the second marginal
of the probability measure $\Pi_{\ve,\ve'}$ defined on $\S_\ve\times \S_{\ve'}$ by
$$\Pi_{\ve,\ve'}(\mathrm{d}\omega\,\mathrm{d}\omega')= \N_\ve(\mathrm{d}\omega\mid \tilde M >\delta)\,\Theta_{\ve,\ve'}(\omega,\mathrm{d}\omega')$$
is 
$\N_{\ve'}(\cdot\mid \tilde M >\delta)$.

Now let $(\ve_n)_{n\geq 1}$ be a sequence of positive reals in $(0,\delta)$ decreasing to $0$. We claim
that we can construct, on  a suitable
probability space, a sequence $(W^{\delta,\ve_n})_{n\geq 1}$ 
of random variables with values in $\S$
 such that the following holds:
\begin{enumerate}
\item[(i)] For every $n\geq 1$, $W^{\delta,\ve_n}$ is distributed according to $\N_{\ve_n}(\cdot\mid \tilde M >\delta)$.
\item[(ii)] For every $1\leq n<m$, $W^{\delta,\ve_n}$ is an excursion of $W^{\delta,\ve_m}$ outside $(0,\ve_n)$.
\end{enumerate}
Indeed, we use the Kolmogorov extension theorem to construct the  
sequence $(W^{\delta,\ve_n})_{n\geq 1}$ so that, for every $n\geq 1$, the law of $(W^{\delta,\ve_n},W^{\delta,\ve_{n-1}},\ldots,W^{\delta,\ve_1})$
is
$$\N_{\ve_n}(\mathrm{d}\omega_n\mid \tilde M >\delta)\,\Theta_{\ve_n,\ve_{n-1}}(\omega_n,\mathrm{d}\omega_{n-1})
\Theta_{\ve_{n-1},\ve_{n-2}}(\omega_{n-1},\mathrm{d}\omega_{n-2})\ldots
\Theta_{\ve_2,\ve_{1}}(\omega_2,\mathrm{d}\omega_{1})$$
and properties (i) and (ii) hold by construction. 

For every $n\geq 1$,  set $\tilde W^{\delta,\ve_n}=\mathrm{tr}_0(W^{\delta,\ve_n})$, and let
$\sigma_n=\sigma(\tilde W^{\delta,\ve_n})$ be the duration of $\tilde W^{\delta,\ve_n}$. Clearly, it is still true that,
for $1\leq n<m$, $\tilde W^{\delta,\ve_n}$ is an excursion of $\tilde W^{\delta,\ve_m}$ outside $(0,\ve_n)$.
Therefore, for every $1\leq n<m$,  $\tilde W^{\delta,\ve_n}$ is a subtrajectory of $\tilde W^{\delta,\ve_m}$ and we write
$[a_{n,m},b_{n,m}]\subset[0,\sigma_m]$ for the associated interval. Note that $b_{n,m}-a_{n,m}=\sigma_n$.
 Furthermore, 
if $1\leq n<m<\ell$, we have $[a_{n,\ell},b_{n,\ell}]\subset [a_{m,\ell},b_{m,\ell}]$, and more precisely
\begin{align}
a_{n,\ell}&= a_{n,m}+a_{m,\ell}\,,\label{as-conv-tech1}\\
\sigma_\ell-b_{n,\ell}&= (\sigma_m-b_{n,m})+ (\sigma_\ell-b_{m,\ell})\,. \label{as-conv-tech2}
\end{align}
In particular, for $n$ fixed, the sequence $(a_{n,m})_{m>n}$ is increasing, and we denote its limit by $a_{n,\infty}$ (the fact 
that this limit is finite will be obtained at the beginning of the proof of the next proposition).

\begin{proposition}
\label{as-conv}
We have a.s.
$$\tilde W^{\delta,\ve_n}\build{\la}_{n\to\infty}^{} W^{\delta,0},\quad\hbox{in }\S\,,$$
where the a.s. limit $W^{\delta,0}$ is distributed according to $\N^*_0(\cdot\mid M>\delta)$.
Furthermore,  $\tilde W^{\delta,\ve_n}$ is a subtrajectory of $W^{\delta,0}$, for every $n\geq 1$,
and $\sigma(\tilde W^{\delta,\ve_n})\uparrow \sigma(W^{\delta,0})$ as $n\to\infty$. 
\end{proposition}

\begin{proof}
By Corollary \ref{approx-*}, we already know that the sequence $(\tilde W^{\delta,\ve_n})_{n\geq 1}$
converges in distribution to $\N^*_0(\cdot\mid M>\delta)$, and in particular $\sigma_n=\sigma(\tilde W^{\delta,\ve_n})$
converges in distribution to the law of $\sigma$ under $\N^*_0(\cdot\mid M>\delta)$. On the other
hand, the sequence $(\sigma_n)_{n\geq 1}$ is increasing and thus has an a.s.
limit $\sigma_\infty$. We conclude that $\sigma_\infty$ is distributed as $\sigma$ under $\N^*_0(\cdot\mid M>\delta)$,
and in particular, $\sigma_\infty<\infty$ a.s.  

 Since $a_{n,m}\leq \sigma_m- \sigma_n$ if $n<m$, we obtain
that, for every $n$,
$$a_{n,\infty}\leq \sigma_\infty- \sigma_n.$$
It follows that
\begin{equation}
\label{as-conv-t00}
\lim_{n\to\infty} a_{n,\infty} =0,\ \hbox{a.s.}
\end{equation}
Then, for every fixed $n$, $b_{n,m}=a_{n,m}+\sigma_n$
converges as $m\uparrow\infty$ to $b_{n,\infty}=a_{n,\infty}+\sigma_n$, and, by letting $\ell$ tend to $\infty$ in
\eqref{as-conv-tech1} and \eqref{as-conv-tech2}, we get, for $n<m$, 
\begin{equation}
\label{as-conv-tech3}
a_{n,\infty}=a_{n,m}+a_{m,\infty}\;,\qquad \sigma_\infty-b_{n,\infty}
= (\sigma_\infty-b_{m,\infty})+(\sigma_m-b_{n,m})\,.
\end{equation}

Set $\tilde\zeta^{\delta,\ve_n}_s= \zeta_s(\tilde W^{\delta,\ve_n})$ to simplify notation. By the definition of 
subtrajectories we know that  $\tilde\zeta^{\delta,\ve_n}_s=
\tilde\zeta^{\delta,\ve_m}_{(a_{n,m}+s)\wedge b_{n,m}}-\tilde\zeta^{\delta,\ve_m}_{a_{n,m}}$ 
if $n<m$. We claim that we have a.s.
\begin{equation}
\label{as-conv-t1}
\lim_{n\to\infty} \Bigg(\sup_{m>n}\Big(\sup_{0\leq s\leq a_{n,m}} \tilde\zeta^{\delta,\ve_m}_s\Big)\Bigg)=0
\end{equation}
To verify this claim, first observe that, if $n<n'<m$, we have
$$\sup_{0\leq s\leq a_{n',m}} \tilde\zeta^{\delta,\ve_m}_s\leq \sup_{0\leq s\leq a_{n,m}} \tilde\zeta^{\delta,\ve_m}_s$$
because $a_{n',m}\leq a_{n,m}$. It then follows that the supremum over $m>n$ in
\eqref{as-conv-t1} is a decreasing function of $n$, and so the limit in the left-hand side  of
\eqref{as-conv-t1} exists a.s. as a decreasing limit. Call $L$ this limit. We argue by contradiction assuming
that $P(L>0)>0$. Then we choose $\xi>0$ such that $P(L>\xi)>0$, and we note that, 
on the event $\{L>\xi\}$, we
can find a sequence $n_1<m_1<n_2<m_2<\cdots$, such that, for every
$i=1,2,\ldots$, we have
$$\sup_{0\leq s\leq a_{n_i,m_i}} \tilde\zeta^{\delta,\ve_{m_i}}_s>\xi.$$
It then follows that, on the same event $\{L>\xi\}$ of positive probability, for any integer $k\geq 1$,
and for every large enough $n$, there exist $k$ disjoint intervals $[r_1,s_1],\ldots,[r_k,s_k]$
such that 
$\tilde\zeta^{\delta,\ve_n}_{s_i}-\tilde\zeta^{\delta,\ve_n}_{r_i} >\xi$ for every $1\leq i\leq k$. The latter property contradicts the 
tightness of the sequence of the laws of $\tilde W^{\delta,\ve_n}$ in $\S$, and this contradiction proves our claim 
\eqref{as-conv-t1}. 

By the same argument,
we have also
\begin{equation}
\label{as-conv-t2}
\lim_{n\to\infty} \Bigg(\sup_{m>n}\Big(\sup_{b_{n,m}\leq s\leq \sigma_m} \tilde\zeta^{\delta,\ve_m}_s\Big)\Bigg)=0.
\end{equation}

We can now use  \eqref{as-conv-t1} and  \eqref{as-conv-t2}
to verify that $(\tilde\zeta^{\delta,\ve_n}_s)_{s\geq 0}$ converges uniformly as $n\to\infty$, a.s.
To this end, we define
$$\zeta^{(n)}_s=\left\{
\begin{array}{ll}
0&\hbox{if }s\leq a_{n,\infty},\\
\tilde\zeta^{\delta,\ve_n}_{s-a_{n,\infty}}\quad&\hbox{if }a_{n,\infty}\leq s\leq b_{n,\infty},\\
0&\hbox{if } s\geq b_{n,\infty}.
\end{array}
\right.
$$
Recalling the formula $\tilde\zeta^{\delta,\ve_n}_s=
\tilde\zeta^{\delta,\ve_m}_{(a_{n,m}+s)\wedge b_{n,m}}-\tilde\zeta^{\delta,\ve_m}_{a_{n,m}}$, and using \eqref{as-conv-tech3},
we get for $n<m$,
$$\sup_{s\geq 0} |\zeta^{(n)}_s- \zeta^{(m)}_s| \leq \sup_{0\leq s\leq a_{n,m}} \tilde\zeta^{\delta,\ve_m}_s
+ \sup_{b_{n,m}\leq s\leq \sigma_m} \tilde\zeta^{\delta,\ve_m}_s,$$
and the right-hand side tends to $0$ a.s. as $n$ and $m$ tend to $\infty$ with $n<m$, by \eqref{as-conv-t1} and  \eqref{as-conv-t2}. 
This gives the a.s. uniform convergence of $(\zeta^{(n)}_s)_{s\geq 0}$ as $n\to\infty$. Write $(\zeta^{\delta,0}_s)_{s\geq 0}$ for the limit. The 
a.s. uniform convergence of $(\tilde\zeta^{\delta,\ve_n}_s)_{s\geq 0}$ 
toward the same limit $(\zeta^{\delta,0}_s)_{s\geq 0}$ then follows using now \eqref{as-conv-t00},
and we have also
$\sup\{s\geq 0: \tilde\zeta^{\delta,\ve_n}_s>0\}=\sigma_n\la \sigma_\infty=\sup\{s\geq 0: \zeta^{\delta,0}_s>0\}$ as $n\to\infty$.

Let $\Gamma^{\delta,\ve_n}_s$ stand for the endpoint of the path ${ \tilde W}^{\delta,\ve_n}_s$. Very similar arguments show that 
the analogs of \eqref{as-conv-t1} and \eqref{as-conv-t2} where $\tilde\zeta^{\delta,\ve_m}_s$ is replaced by
$\Gamma^{\delta,\ve_m}_s$ hold, and it follows that
$(\Gamma^{\delta,\ve_n}_s)_{s\geq 0}$ also converges uniformly
to a limit denoted by  $(\Gamma^{\delta,0}_s)_{s\geq 0}$, a.s. The pair $(\zeta^{\delta,0},\Gamma^{\delta,0})$
is then a random tree-like path, and letting $W^{\delta,0}$ be the associated snake trajectory, we have
obtained that $\tilde W^{\delta,\ve_n}$ converges a.s. to $W^{\delta,0}$. Since we know
that $\tilde W^{\delta,\ve_n}$
converges in distribution to $\N^*_0(\cdot\mid M>\delta)$, $W^{\delta,0}$ is distributed according to
$\N^*_0(\cdot\mid M>\delta)$. 

Finally, it follows from our construction that, for every $n\geq 1$,  $\tilde W^{\delta,\ve_n}$ is the subtrajectory of $W^{\delta,0}$
associated with the interval $[a_{n,\infty},b_{n,\infty}]$, and the property 
$\sigma(\tilde W^{\delta,\ve_n})\uparrow \sigma(W^{\delta,0})$ is just the fact that $\sigma_n\uparrow \sigma_\infty$. This completes the proof. 
\end{proof}

\section{The exit measure}
\label{sec:exit}

We now define the exit measure from $(0,\infty)$ under $\N^*_0$. Informally, this exit 
measure corresponds to the quantity of snake trajectories that return to $0$. 

\begin{proposition}
\label{construct-exit}
The limit
$$\lim_{\ve\to 0} \ve^{-2}\int_0^\sigma \mathrm{d}s\,\mathbf{1}_{\{\hat W_s<\ve\}}$$
exists in probability under $\N^*_0(\cdot\mid \sigma>\chi)$, for every $\chi >0$, and defines a finite random variable 
denoted by $Z^*_0$.
\end{proposition}

\begin{proof}
We rely on the re-rooting property of Section \ref{sec:re-root}.
Let $(\ve_n)_{n\geq 1}$ be a sequence of positive reals converging to $0$. Recalling Lemma \ref{approx} and 
the subsequent remarks, we can extract from the sequence  $(\ve_n)_{n\geq 1}$
a subsequence $(\beta_n)_{n\geq 1}$ such that, for every $b<0$,
\begin{equation}
\label{approx-ps2}
Z_b= \lim_{n\to\infty} \beta_n^{-2}\int_0^\sigma \mathrm{d}s\,\mathbf{1}_{\{\zeta_s\leq\tau_{-b}(W_s), \hat W_s<-b+\beta_n\}}\;,\qquad \N_0\hbox{ a.e.}
\end{equation}
Then, for $\omega\in \S$, we set $G(\omega)=0$
if the limit
$$\lim_{n\to\infty} \beta_n^{-2}\int_0^{\sigma(\omega)} \mathrm{d}s\,\mathbf{1}_{\{\hat W_s(\omega)<W_*(\omega)+\beta_n\}}$$
exists (and is finite), and $G(\omega)=1$ otherwise. By \eqref{approx-ps2}, we have $G({\rm tr}_{-b}(W))=0$,
$\N_0$ a.e. on the event $\{W_*\leq -b\}=\{Z_b>0\}$, for every $b>0$. By Theorem \ref{reenracinement},
we have then
$$\N^*_0\Bigg(\int_0^\sigma \mathrm{d}s\,G(W^{[s]})\Bigg)=0.$$
We have thus obtained that $\N^*_0$ a.e., for Lebesgue a.e. $r\in [0,\sigma]$, $G(W^{[r]})=0$. By considering
just one value of $r$ for which $G(W^{[r]})=0$, this says that the convergence of the proposition holds $\N^*_0$ a.e. along the
sequence $(\beta_n)_{n\geq 1}$. We have thus shown that from any sequence of positive real
numbers converging to $0$ we can extract a subsequence along which the convergence of the proposition
holds $\N^*_0$ a.e. The statement of the proposition follows.
\end{proof}

Recall from Section \ref{sec:exitprocess} that we have fixed a  sequence $(\alpha_n)_{n\geq 1}$ such that \eqref{approx-ps} holds. We then
define $Z^*_0(\omega)$ for {\it every} $\omega\in \S$, by setting
\begin{equation}
\label{def-exit-gene}
Z^*_0(\omega)=\liminf_{n\to\infty} \alpha_n^{-2}\int_0^\sigma \mathrm{d}s\,\mathbf{1}_{\{\hat W_s(\omega)<W_*(\omega)+\alpha_n\}}.
\end{equation}
By the argument we have just given in the proof of Proposition \ref{construct-exit},
the liminf is a limit $\N^*_0$ a.e. 
In what follows, we will be concerned by the values of $Z^*_0(\omega)$ under $\N^*_0$, and we note that
the quantity $W_*(\omega)$ in (\ref{def-exit-gene}) can be replaced by $0$, $\N^*_0$ a.e., so that 
\eqref{def-exit-gene} is consistent with Proposition \ref{construct-exit}. 

\smallskip
Our next goal is to compute the joint distribution of the pair $( Z^*_0, \sigma)$ under 
$\N^*_0$.

\begin{proposition}
\label{loisZ0sigma}
 The distribution of the pair $(Z^*_0, \sigma)$ under $\mathbb{N}_0^*$ has a density $f$ given for $z>0$ and $s>0$ by
$$ f(z,s)=\frac{\sqrt{3}}{2 \pi} \sqrt{z} \,s^{-5/2} \exp \left( -\frac{z^2}{2s} \right).$$
In particular, the respective densities $g$ of $Z^*_0$ and $h$ of $\sigma$ under $\mathbb{N}_0^*$
are given by
$$ g(z)= \sqrt{\frac{3}{2 \pi}} z^{-5/2}, \quad z>0,$$
and
$$h(s)= \frac{\sqrt{3}}{2 \pi} 2^{-1/4} \Gamma(3/4) s^{-7/4},\quad s>0.$$
\end{proposition}

\begin{proof}
We fix $\lambda>0$ and $\mu>0$, and compute
$$\N^*_0\Big(\sigma\,\exp(-\lambda Z^*_0-\mu \sigma)\Big).$$
Recalling \eqref{def-exit-gene}, and using  Lemma \ref{approx}, we get that
$Z^*_0(\mathrm{tr}_{-b}(W))=Z_b$, $\N_0$ a.e. on $\{Z_b>0\}$, for every $b>0$.
Hence, by applying Theorem \ref{reenracinement} to the function $G(\omega)=\exp(-\lambda Z^*_0(\omega)-\mu \sigma(\omega))$,
we obtain
$$\N^*_0\Big(\sigma\,\exp(-\lambda Z^*_0-\mu \sigma)\Big)
=2\int_0^\infty \mathrm{d}b\,\N_0\Big(Z_b\exp(-\lambda Z_b-\mu \mathcal{Y}_b)\Big)$$
with the notation 
$$\mathcal{Y}_b= \int_0^\sigma \mathrm{d}s\,\mathbf{1}_{\{\tau_{-b}(W_s)=\infty\}}$$
(note that $\mathcal{Y}_b=\sigma(\mathrm{tr}_{-b}(W))$, $\N_0$ a.e.). Set
$$u_{\lambda, \mu}(b)=\mathbb{N}_{0}\Big(1-\exp(-\lambda Z_b-\mu \mathcal{Y}_b)\Big),$$
and note that 
$$\frac{\mathrm{d}}{\mathrm{d}\lambda}u_{\lambda, \mu}(b)
= \N_0\Big(Z_b\exp(-\lambda Z_b-\mu \mathcal{Y}_b)\Big).$$
The quantity $u_{\lambda,\mu}(b)$ is computed explicitly in \cite[Lemma 4.5]{Hull}: If $ \lambda < \sqrt{\frac{\mu}{2}}$,
$$u_{\lambda, \mu} (b)= \sqrt{\frac{\mu}{2}} \left(3 \left( \tanh^2 \left( (2 \mu)^{1/4} b + \tanh^{-1} \sqrt{\frac{2}{3}+\frac{1}{3} \sqrt{\frac{2}{\mu}}\lambda} \right) \right)- 2\right),$$
and a similar formula holds if $ \lambda > \sqrt{\frac{\mu}{2}}$. From this explicit formula,  in the case 
$ \lambda < \sqrt{\frac{\mu}{2}}$ one gets
\begin{align*}
\frac{\mathrm{d}}{\mathrm{d}\lambda} u_{\lambda, \mu} (b)= &K_{\lambda, \mu}^{-1} \tanh \left( (2 \mu)^{1/4} b + \tanh^{-1} \sqrt{\frac{2}{3}+\frac{1}{3} \sqrt{\frac{2}{\mu}}\lambda} \right)\\
&\times \left(\cosh^{2} \left( (2 \mu)^{1/4} b + \tanh^{-1} \sqrt{\frac{2}{3}+\frac{1}{3} \sqrt{\frac{2}{\mu}}\lambda} \right)\right)^{-1}
\end{align*}
where 
$$K_{\lambda, \mu}= \frac{1}{3}\left( 1-\sqrt{\frac{2}{\mu}}\lambda\right) \sqrt{\frac{2}{3}+\frac{1}{3} \sqrt{\frac{2}{\mu}}\lambda}.$$
By integrating the last formula between $b=0$ and $b=\infty$, we arrive at
$$\int_0^\infty \mathrm{d}b\,\N_0\Big(Z_b\exp(-\lambda Z_b-\mu \mathcal{Y}_b)\Big)
= \int_0^\infty \mathrm{d}b\,\frac{d}{d\lambda} u_{\lambda, \mu} (b)
= \frac{1}{2}\,\sqrt{\frac{3}{2}} \Big(\lambda + \sqrt{2\mu}\Big)^{-1/2}.$$
Similar calculations give the same result when $ \lambda > \sqrt{\frac{\mu}{2}}$ (and also
in  the case $ \lambda = \sqrt{\frac{\mu}{2}}$ by a suitable passage to the limit). Summarizing, we have proved that,
for every $\lambda >0$ and $\mu>0$,
$$\N^*_0\Big(\sigma\,\exp(-\lambda Z^*_0-\mu \sigma)\Big)=\sqrt{\frac{3}{2}} \Big(\lambda + \sqrt{2\mu}\Big)^{-1/2}.$$
At this stage, we only need to verify that, with the function $f$ defined in the proposition, we have also
$$\int_0^{\infty} \int_0^{\infty} s \exp(-\lambda z -\mu s) \,f(z,s)\,\mathrm{d}z\mathrm{d}s
=  \sqrt{\frac{3}{2}} \Big(\lambda + \sqrt{2\mu}\Big)^{-1/2}.$$
To see this, first note that, for every $z>0$,
$$z\int_0^\infty s^{-3/2} \exp \Big( -\frac{z^2}{2s}-\mu z \Big)\,\mathrm{d}s= \sqrt{2\pi}\,e^{-z\sqrt{2\mu}},$$
by the classical formula for the Laplace transform of hitting times of a standard linear Brownian motion. 
The desired result easily follows.
\end{proof}

We now state a technical result that will be important for our purposes. Let us fix $\delta>0$, and, for every
$\ve\in(0,\delta)$, write
$W^{\delta,\ve}$ for a random snake trajectory distributed according to $\N_\ve(\cdot\mid \tilde M >\delta)$, where we
recall the notation $\tilde M=\sup\{W_s(t):s\geq 0, t\leq \zeta_s\wedge \tau_0(W_s)\}$. As usual, write
$\tilde W^{\delta,\ve}$ for $W^{\delta,\ve}$ truncated at level $0$. By Corollary \ref{approx-*}, 
the distribution of $\tilde W^{\delta,\ve}$ converges to $\N^*_0(\cdot\mid M>\delta)$ as $\ve\to 0$. The next proposition
shows that this convergence holds jointly with that of the exit measures from $(0,\infty)$. Recall the notation
$\z_0(W^{\delta,\ve})$ for the (total mass of the) exit measure of $W^{\delta,\ve}$ from $(0,\infty)$.

\begin{proposition}
\label{conv-exit-M}
As $\ve\to 0$, the distribution of the pair $(\tilde W^{\delta,\ve},\z_0(W^{\delta,\ve}))$ converges weakly to that
of the pair $(W^{\delta,0},Z^*_0(W^{\delta,0}))$, where $W^{\delta,0}$ is distributed according to $\N^*_0(\cdot\mid M>\delta)$.
\end{proposition}

\begin{proof}
We may argue along a sequence $(\ve_n)_{n\geq 1}$ strictly decreasing to $0$. To simplify notation, we
set $W^n= W^{\delta,\ve_n}$ and $\tilde W^n=\tilde W^{\delta,\ve_n}$. From Proposition \ref{as-conv},
we may
construct (on a suitable probability space) the whole sequence $(W^n)_{n\geq 1}$ and the snake trajectory $W^{\delta,0}$ in such a way that
$W^n$ is an excursion of $W^m$ outside $(0,\ve_n)$ for every $n<m$, 
$\tilde W^n$ is a subtrajectory of $W^{\delta,0}$ for every $n\geq 1$, $\tilde W^n\la W^{\delta,0}$ in $\S$ as $n\to\infty$, a.s., and moreover
$\sigma(\tilde W^n)\uparrow \sigma(W^{\delta,0})$ as $n\to\infty$. These properties imply that, for every $\gamma>0$
and every $1\leq n\leq m$, we have
$$\int_0^{\sigma(W^n)} \mathrm{d}s\,\mathbf{1}_{\{\zeta^n_s\leq \tau_0(W^n_s),\,\hat W^n_s<\gamma\}}
\leq \int_0^{\sigma(W^m)} \mathrm{d}s\,\mathbf{1}_{\{\zeta^m_s\leq \tau_0(W^m_s),\,\hat W^m_s<\gamma\}}
\leq \int_0^{\sigma(W^{\delta,0})}  \mathrm{d}s\,\mathbf{1}_{\{\hat W^{\delta,0}_s<\gamma\}}.$$
If we multiply this inequality by $\gamma^{-2}$ and let $\gamma$ tend to $0$, we obtain that, for every $1\leq n\leq m$, 
$$\z_0(W^n)\leq \z_0(W^m)\leq Z^*_0(W^{\delta,0}).$$
In particular the almost sure increasing limit
$$Z'_0:=\lim_{n\to\infty}\uparrow \z_0(W^n)$$
exists and we have $Z'_0\leq Z^*_0(W^{\delta,0})$. The result of the proposition will follow
if we can verify that we have indeed $Z'_0= Z^*_0(W^{\delta,0})$ a.s. To this end, fix $\lambda>0$ and $\mu>0$. 
Write $E[\cdot]$ for the expectation on the probability space where 
the sequence $(W^n)_{n\geq 1}$ and $W^{\delta,0}$ are defined. 
We note that
\begin{equation}
\label{conv-exit-tech0}
E[\exp(-\lambda Z'_0)(1-\exp(-\mu \sigma(W^{\delta,0})))]
\leq \liminf_{n\to\infty} E[\exp(-\lambda \z_0(W^n))(1-\exp(-\mu \sigma(\tilde W^n)))]
\end{equation}
by Fatou's lemma. We will verify that
\begin{equation}
\label{conv-exit-tech}
\liminf_{n\to\infty} E[\exp(-\lambda \z_0(W^n))(1-\exp(-\mu \sigma(\tilde W^n)))]\leq E[\exp(-\lambda Z^*_0(W^{\delta,0}))
(1-\exp(-\mu \sigma(W^{\delta,0})))].
\end{equation}
If \eqref{conv-exit-tech} holds, then by combining this with the previous display, we get
$$E[\exp(-\lambda Z'_0)(1-\exp(-\mu \sigma(W^{\delta,0})))]\leq E[\exp(-\lambda Z^*_0(W^{\delta,0}))
(1-\exp(-\mu \sigma(W^{\delta,0})))],$$
and since we already know that $Z'_0\leq Z^*_0(W^{\delta,0})$, this is only possible if $Z'_0= Z^*_0(W^{\delta,0})$ a.s. 

\smallskip
Let us prove  \eqref{conv-exit-tech}. Since $W^n$ is distributed according to $\N_{\ve_n}(\cdot\mid \tilde M>\delta)$, 
$W^{\delta,0}$ is distributed according to $\N^*_0(\cdot\mid \tilde M>\delta)$, and 
we know that $\N_\ve(\tilde M>\delta) \sim \ve\,\N^*_0(M>\delta)$
as $\ve\to 0$ (see the proof of Lemma \ref{loi-max}), we see that \eqref{conv-exit-tech} is equivalent to
\begin{equation}
\label{conv-exit-tech2}
\liminf_{n\to\infty} \frac{1}{\ve_n}\,\N_{\ve_n}\Big(\exp(-\lambda \z_0)(1-\exp(- \mu \mathcal{Y}_0))\,\mathbf{1}_{\{\tilde M>\delta\}}\Big)
\leq \N^*_0 (\exp(-\lambda Z^*_0)(1-\exp(-\mu \sigma))\,\mathbf{1}_{\{M>\delta\}}),
\end{equation} 
where we recall that
$$\mathcal{Y}_0=\int_0^\sigma \mathrm{d}s\,\mathbf{1}_{\{\tau_{0}(W_s)=\infty\}}.$$
Observe that, for any choice of $\gamma\in(0,\delta)$, the argument leading to \eqref{conv-exit-tech0}
(using also the fact that $M(\tilde W^n)$ converges a.s. to $M(W^{\delta,0})$)  gives
\begin{equation}
\label{conv-exit-tech3}
\liminf_{n\to\infty} \frac{1}{\ve_n}\,\N_{\ve_n}\Big(\exp(-\lambda \z_0)(1-\exp(- \mu \mathcal{Y}_0))\,\mathbf{1}_{\{\gamma <\tilde M\leq\delta\}}\Big)
\geq \N^*_0 (\exp(-\lambda Z^*_0)(1-\exp(-\mu \sigma))\,\mathbf{1}_{\{\gamma<M\leq \delta\}}),
\end{equation}
and by letting $\gamma$ tend to $0$,
$$\liminf_{n\to\infty} \frac{1}{\ve_n}\,\N_{\ve_n}\Big(\exp(-\lambda \z_0)(1-\exp(- \mu \mathcal{Y}_0))\,\mathbf{1}_{\{\tilde M\leq\delta\}}\Big)
\geq \N^*_0 (\exp(-\lambda Z^*_0)(1-\exp(-\mu \sigma))\,\mathbf{1}_{\{M\leq \delta\}}).
$$
So if \eqref{conv-exit-tech2} fails, we get
$$\liminf_{n\to\infty} \frac{1}{\ve_n}\,\N_{\ve_n}\Big(\exp(-\lambda \z_0)(1-\exp(- \mu \mathcal{Y}_0))\Big)
>  \N^*_0 (\exp(-\lambda Z^*_0)(1-\exp(-\mu \sigma))).$$
We will prove that we have
\begin{equation}
\label{conv-exit-tech4}
\lim_{n\to\infty} \frac{1}{\ve_n}\,\N_{\ve_n}\Big(\exp(-\lambda \z_0)(1-\exp(- \mu \mathcal{Y}_0))\Big)
=  \N^*_0 (\exp(-\lambda Z^*_0)(1-\exp(-\mu \sigma))),
\end{equation}
showing by contradiction that \eqref{conv-exit-tech2} and thus also \eqref{conv-exit-tech} hold. 

The right-hand side of \eqref{conv-exit-tech4} can be  computed from the formula
$$ \N^*_0 (\sigma \exp(-\lambda Z^*_0-\mu \sigma))
=\sqrt{\frac{3}{2}} \Big(\lambda + \sqrt{2\mu}\Big)^{-1/2},$$
which was obtained in the proof of Proposition \ref{loisZ0sigma}. We get
\begin{align}
\label{conv-exit-tech5}
 \N^*_0 (\exp(-\lambda Z^*_0)(1-\exp(-\mu \sigma)))
 &= \N^*_0 \Big(\exp(-\lambda Z^*_0)\int_0^\mu \mathrm{d}\mu'
 \,\sigma \exp(-\mu'\sigma)\Big)\\
 &=\int_0^\mu  \mathrm{d}\mu'\, \sqrt{\frac{3}{2}} \Big(\lambda + \sqrt{2\mu'}\Big)^{-1/2}\nonumber\\
 &=\sqrt{\frac{3}{2}}  \int_0^{\sqrt{2\mu}} \mathrm{d}x\,x\,(\lambda+x)^{-1/2}\nonumber\\
 &=\sqrt{\frac{2}{3}} \Big( (\lambda+\sqrt{2\mu})^{3/2} - 3\lambda\,(\lambda+\sqrt{2\mu})^{1/2}+ 2\lambda^{3/2}\Big). \nonumber
\end{align}

On the other hand, we have, for every $\ve>0$,
\begin{align*}
\N_{\ve}\Big(\exp(-\lambda \z_0)(1-\exp(- \mu \mathcal{Y}_0))\Big)
&=\N_{\ve}\Big(1-\exp(-\lambda \z_0- \mu \mathcal{Y}_0)\Big)- \N_\ve\Big(1-\exp(-\lambda \z_0)\Big)\\
&=u_{\lambda,\mu}(\ve)- \Big(\frac{1}{\sqrt{\lambda}} +\ve\sqrt{\frac{2}{3}}\,\Big)^{-2},
\end{align*}
recalling \eqref{Laplace-exit} and using the notation introduced  in the proof of Proposition \ref{loisZ0sigma}. Formula (26) in \cite{Hull} gives
$$\lim_{\ve\to 0} \frac{1}{\ve}\,(u_{\lambda,\mu}(\ve)-\lambda) = \frac{\mathrm{d}}{\mathrm{d}\ve}u_{\lambda,\mu}(\ve)_{|\ve=0}
= \sqrt{\frac{2}{3}}\,(\lambda + \sqrt{2\mu})^{1/2}\,(\sqrt{2\mu}-2\lambda).$$
It follows that
$$\lim_{n\to\infty} \frac{1}{\ve_n}\,\N_{\ve_n}\Big(\exp(-\lambda \z_0)(1-\exp(- \mu \mathcal{Y}_0))\Big)
=  \sqrt{\frac{2}{3}}\,(\lambda + \sqrt{2\mu})^{1/2}\,(\sqrt{2\mu}-2\lambda)+ 2  \sqrt{\frac{2}{3}} \,\lambda^{3/2},$$
and one immediately verifies that the right-hand side of the last display coincides with the right-hand side of \eqref{conv-exit-tech5}.
This completes the proof of \eqref{conv-exit-tech4} and of the proposition.
\end{proof}

In view of our applications, it will be important to define the measure $\N^*_0$ conditioned on 
a given value of the exit measure. This is the goal of the next proposition. Before that, we mention a 
useful scaling property. Recall the definition of the scaling operator $\theta_\lambda$ at the
end of Section \ref{sec:construct}. Then for every $\lambda>0$, we have for every $\omega\in \S$,
\begin{equation}
\label{exit-scaling}
Z^*_0 \circ \theta_\lambda(\omega) = \lambda\,Z^*_0(\omega).
\end{equation}
The proof is easy, recalling from \eqref{def-exit-gene} the definition of $Z^*_0(\omega)$ for an arbitrary $\omega\in \S$ and writing
\begin{align*}
Z^*_0 \circ \theta_\lambda(\omega) &=
\liminf_{n\to\infty} \alpha_n^{-2}\int_0^{\lambda^2\sigma(\omega)} \mathrm{d}s\,\mathbf{1}_{\{\sqrt{\lambda}\hat W_{s/\lambda^2}(\omega)<\sqrt{\lambda}W_*+\alpha_n\}}\\
&= \lambda \liminf_{n\to\infty} (\alpha_n/\sqrt{\lambda})^{-2} \int_0^{\sigma(\omega)}  \mathrm{d}s\,\mathbf{1}_{\{\hat W_s<
W_*+\alpha_n/\sqrt{\lambda}\}}\\
&=\lambda Z^*_0(\omega).
\end{align*}

\begin{proposition}
\label{conditioned-exit}
There exists a unique collection $(\N^{*,z}_0)_{z>0}$ of probability measures on $\S$
such that:
\begin{enumerate}
\item[\rm(i)] We have
$$\N^*_0=
 \sqrt{\frac{3}{2 \pi}} 
\int_0^\infty \mathrm{d}z\,z^{-5/2}\, \N^{*,z}_0.$$
\item[\rm(ii)] For every $z>0$, $\N^{*,z}_0$ is supported on $\{Z^*_0=z\}$.
\item[\rm(iii)] For every $z,z'>0$, $\N^{*,z'}_0=\theta_{z'/z}(\N^{*,z}_0)$.
\end{enumerate}
We will write $\N^{*,z}_0=\N^*_0(\cdot\mid Z^*_0=z)$.
\end{proposition}

\begin{proof}
Recall from Proposition \ref{loisZ0sigma} that the ``law'' of $Z^*_0$
under $\N^*_0$ is the measure $\mathbf{1}_{\{z>0\}}\,\sqrt{3/2\pi} \,z^{-5/2}\,\mathrm{d}z$, which
we denote here by $\nu(\mathrm{d}z)$ to simplify notation.
The existence of a collection of probability measures on $\S$ that satisfy both (i) and (ii) in
the proposition is a consequence of standard disintegration theorems (see e.g. 
\cite[Chapter III, Paragraphs (70--74)]{DellacherieMeyer}). Two such collections coincide up to
a negligible set of values of $z$. We need to verify that we can choose this collection so
that the additional scaling property (iii) also holds (which will imply the stronger uniqueness 
in the proposition). 

We start with any measurable collection $(\Q_z)_{z>0}$
of probability measures on $\S$ such that the properties stated in (i) and (ii)
hold when $(\N^{*,z}_0)_{z>0}$ is replaced by $(\Q_z)_{z>0}$. From Lemma \ref{scaling-N*}, we get
that, for every $\lambda>0$,
$$\int \theta_\lambda(\Q_z)\,\nu(\mathrm{d}z)= \theta_\lambda(\N^*_0)=\lambda^{3/2}\N^*_0= 
\lambda^{3/2}\int \Q_z\,\nu(\mathrm{d}z).$$
From the change of variables $z=z'/\lambda$ in the first integral, we thus get
$$\int \theta_\lambda(\Q_{z/\lambda})\,\nu(\mathrm{d}z)= \int \Q_z\,\nu(\mathrm{d}z).$$
Using the scaling property \eqref{exit-scaling}, we see that the collection $(\theta_\lambda(\Q_{z/\lambda}))_{z>0}$
also satisfy the conditions (i) and (ii), and so we get for every fixed $\lambda>0$,
$$\theta_\lambda(\Q_{z/\lambda})=\Q_z\;,\quad \mathrm{d}z\ \hbox{a.e.}$$
From Fubini's theorem, we have then $\theta_\lambda(\Q_{z/\lambda})=\Q_z$, $\mathrm{d}\lambda$ a.e.,
 $\mathrm{d}z$ a.e. At this stage, we can pick $z_0>0$ such that the equality $\theta_\lambda(\Q_{z_0/\lambda})=\Q_{z_0}$
 holds $\mathrm{d}\lambda$ a.e., and define $\N^{*,z}_0:=\theta_{z/z_0}(\Q_{z_0})$ for every $z>0$. 
We have then $\N^{*,z}_0=\Q_z$, $\mathrm{d}z$ a.e., so that (i) holds for the collection $(\N^{*,z}_0)_{z>0}$. Similarly
(ii) holds because $\Q_{z_0}$ is supported on $\{Z^*_0=z_0\}$, and we use the scaling property
\eqref{exit-scaling}. Property (iii) holds by construction.

To get uniqueness, observe that (iii) implies that the mapping $z\mapsto \N^{*,z}_0$ 
is continuous for the weak convergence of probability measures. The uniqueness is then a simple consequence of this continuous dependence
and the fact that two collections that satisfy both (i) and (ii) must coincide up to
a negligible set of values of $z$.
\end{proof}

\section{The excursion process}
\label{sec:excu-pro}

For technical reasons in this section, it is preferable to argue under a probability measure rather than under $\N_0$. 
So we fix $\beta>0$, and we argue under the conditional measure $\N^{(\beta)}_0:=\N_0(\cdot \mid W_*<-\beta)$. 
We will then consider,
under $\N^{(\beta)}_0$, the excursion debuts whose level is smaller than $-\beta$. For every
$\delta>0$, we write $u^\delta_1,\ldots,u^\delta_{N_\delta}$ for the excursion debuts with height greater than $\delta$
whose level is smaller than $-\beta$, listed in decreasing order of the levels, so that
$$V_{u^\delta_{N_\delta}}<V_{u^\delta_{N_\delta-1}}<\cdots<V_{u^\delta_1}<-\beta.$$
Notice that $N_\delta$ and $u^\delta_1,\ldots,u^\delta_{N_\delta}$ depend on the choice of $\beta$,
which will remain fixed in the first three subsections below (although on a couple of occasions we mention the consequences
that one derives by letting
$\beta$ tend to $0$, but this should create no confusion). For every integer $i\geq 1$, we also set
$$T^\delta_i:=\left\{\begin{array}{ll}
-V_{u^\delta_i}\qquad&\hbox{if}\ i\leq N_\delta,\\
\infty&\hbox{if}\ i> N_\delta.
\end{array}
\right.
$$
It is easy to verify that, for every $a>0$, the event $\{T^\delta_i<a\}$ belongs to the $\sigma$-field $\mathcal{E}^{(-a,\infty)}$
(the knowledge of $\mathcal{E}^{(-a,\infty)}$ gives enough information to recover the excursion debuts -- and the
corresponding heights -- such that $V_u>-a$). Since $\{T^\delta_i=a\}$ is $\N_0$-negligible, it follows that
$T^\delta_i$ is a stopping time of the filtration $(\mathcal{E}^{(-a,\infty)})_{a\geq 0}$, where, by convention,
$\mathcal{E}^{(0,\infty)}$ is the $\sigma$-field generated by the $\N_0$-negligible sets. Finally, it will also be useful to
write $N^\circ_\delta=\# D_\delta$ for the total number of excursion debuts with height greater than $\delta$.

\subsection{The excursions with height greater than $\delta$}

Recall the notation $\tilde W^{(u)}$ for the excursion starting at the excursion debut $u\in D$. 

\begin{proposition}
\label{sequence-excu}
Let $j\geq 1$. Then, under the conditional probability measure $\N^{(\beta)}_0(\cdot\mid N_\delta \geq j)$,
$\tilde W^{(u^\delta_j)}$ is independent of the $\sigma$-field generated by $(\tilde W^{(u^\delta_1)},\ldots,\tilde W^{(u^\delta_{j-1})})$ 
and $\mathcal{E}^{(-\beta,\infty)}$, and is distributed
according to $\N^*_0(\cdot\mid M>\delta)$.
\end{proposition}

\noindent{\it Important remark.} In view of the analogous statement for linear
Brownian motion, one might naively expect that $\tilde W^{(u^\delta_1)},\ldots,\tilde W^{(u^\delta_{j})}$
are (independent and) identically distributed under $\N_0^{(\beta)}(\cdot\mid N_\delta \geq j)$. This is not true as soon as
$j\geq 2$: The point is that the knowledge of the event $\{N_\delta \geq j\}$ influences the 
distribution of $(\tilde W^{(u^\delta_1)},\ldots,\tilde W^{(u^\delta_{j-1})})$. 

\begin{proof}
The first step of the proof is to determine the law of $\tilde W^{(u^\delta_1)}$ under $\N^{(\beta)}_0(\cdot\mid N_\delta \geq 1)$.
We fix two bounded nonnegative functions $G$ and $g$ defined respectively on $\S$ and on $\R$. We assume that $G$ is bounded and
continuous on the set $\{\omega:M(\omega)>\delta\}$, and vanishes outside this set. The function
$g$ is assumed to be continuous with compact support contained in $(-\infty,-\beta]$. 

We retain much of the notation of the proof of Theorem \ref{construct-N}. In particular, for every integers $n\geq1$ and $k\geq 1$,
we let
$\n^{2^{-n}}_k$
be the point measure of excursions of the Brownian snake outside $(-k2^{-n},\infty)$, and we write
$$\n^{2^{-n}}_k=\sum_{i\in I^{2^{-n}}_k} \delta_{\omega^{k,2^{-n}}_i}.$$
Recall that, for every atom $\omega^{k,2^{-n}}_i$, $\tilde\omega^{k,2^{-n}}_i$ stands  for $\omega^{k,2^{-n}}_i$
translated so that its starting point is $2^{-n}$ and then  truncated at level $0$. Furthermore, we let $A_{n,k}$ stand for the event
$\{T^\delta_1\geq k2^{-n}\}=\{V_{u^\delta_1}\leq -k2^{-n}\}$. Finally, we let $B\in \mathcal{E}^{(-\beta,\infty)}$. 

We then claim that
\begin{equation}
\label{seq-excu-key}
\N_0^{(\beta)}\Big(\mathbf{1}_B\,g(V_{u^\delta_1})\,G(\tilde W^{(u^\delta_1)})\,\mathbf{1}_{\{N_\delta\geq 1\}}\Big)
=\lim_{n\to\infty}
\N_0^{(\beta)}\Bigg(\mathbf{1}_B\,\sum_{k=1}^\infty \mathbf{1}_{A_{n,k}} g(-k2^{-n})\sum_{i\in I^{2^{-n}}_k} G(\tilde\omega^{k,2^{-n}}_i)\Bigg),
\end{equation}
In order to verify our claim,  we first observe that 
\begin{equation}
\label{seq-excu-key2}
\sum_{k=1}^\infty \mathbf{1}_{A_{n,k}} g(-k2^{-n})\sum_{i\in I^{2^{-n}}_k} G(\tilde\omega^{k,2^{-n}}_i)\build{\la}_{n\to\infty}^{}
g(V_{u^\delta_1})\,G(\tilde W^{(u^\delta_1)})\,\mathbf{1}_{\{N_\delta\geq 1\}},\qquad \N_0\ \hbox{a.e.}
\end{equation}
To see this, note that if $N_\delta=0$ then, for $n$ large enough, all quantities $G(\tilde\omega^{k,2^{-n}}_i)$ vanish (the point
is that, if $G(\tilde\omega^{k,2^{-n}}_i)>0$, then the excursion $\omega^{k,2^{-n}}_i$ must ``contain'' an excursion debut with
height greater than $\delta-2^{-n}$, and no such excursion debut exists when $n$ is large enough, under the condition 
$N_\delta=0$). 
Then, if $N_\delta\geq 1$, similar arguments show that, for $n$ large enough, the only nonzero term in the sum over $k$ in the 
left-hand side of \eqref{seq-excu-key2} corresponds to the integer $k_0=k_0(n)$ such that $-(k_0+1)2^{-n}<V_{u^\delta_1}\leq -k_02^{-n}$.
Indeed, we have $\mathbf{1}_{A_{n,k}} =0$ if $k>k_0$, since $A_{n,k}=\{V_{u^\delta_1}\leq -k2^{-n}\}$. On the other hand, if $n$ is large enough, then, for $k<k_0$, the quantities 
$G(\tilde\omega^{k,2^{-n}}_i)$, $i\in I^{2^{-n}}_k$, vanish by the same argument as used above in the case $N_\delta=0$, recalling
that $G$ is zero outside the set 
$\{\omega:M(\omega)>\delta\}$, 

Next, for $k=k_0$, the sum over $i\in I^{2^{-n}}_k$ reduces (for
$n$ large enough) to a single term, namely $i=i_0=i_{u^1_\delta,2^{-n}}$ with the notation of Lemma \ref{construct-tech}.
The last assertion of Lemma \ref{construct-tech} yields that $G(\tilde\omega^{k_0,2^{-n}}_{i_0})$ converges to $G(\tilde W^{(u^\delta_1)})$
as $n\to\infty$, and \eqref{seq-excu-key2} follows. 

To derive \eqref{seq-excu-key} from \eqref{seq-excu-key2}, we use exactly the same uniform integrability argument as
in the proof of Theorem \ref{construct-N} to justify the convergence \eqref{key-conv-N*}. 

Next recall that $A_{n,k}$ is measurable with respect to the $\sigma$-field $\mathcal{E}^{(-k2^{-n},\infty)}$, and
note that $g(-k2^{-n})=0$ if $k\leq 2^n\beta$. By 
applying the special Markov property, we then get 
\begin{align*}
&\N_0^{(\beta)}
\Bigg(\mathbf{1}_B\,\sum_{k\geq 2^n\beta} \mathbf{1}_{A_{n,k}} g(-k2^{-n})\sum_{i\in I^{2^{-n}}_k} G(\tilde\omega^{k,2^{-n}}_i)\Bigg)\\
&\qquad=\sum_{k\geq 2^n\beta} g(-k2^{-n})\, \N_0^{(\beta)}\Bigg(\mathbf{1}_B\mathbf{1}_{A_{n,k}} \N_0^{(\beta)}\Bigg( \sum_{i\in I^{2^{-n}}_k} G(\tilde\omega^{k,2^{-n}}_i)
\,\Bigg|\,\mathcal{E}^{(-k2^{-n},\infty)}\Bigg)\Bigg)\\
&\qquad=\sum_{k\geq 2^n\beta} g(-k2^{-n})\, \N_0^{(\beta)}\Big(\mathbf{1}_B\mathbf{1}_{A_{n,k}}\, Z_{k2^{-n}}\,\N_{2^{-n}}\big(G(\tilde W)\big)\Big)\\
&\qquad=\Bigg(\sum_{k\geq 2^n\beta} g(-k2^{-n})\, \N_0^{(\beta)}\Big(\mathbf{1}_B\mathbf{1}_{A_{n,k}}\, Z_{k2^{-n}}\Big)\Bigg)\times \N_{2^{-n}}\big(G(\tilde W)\big).
\end{align*}
Recalling \eqref{seq-excu-key}, we have thus obtained 
\begin{equation}
\label{seq-excu-key3}
\lim_{n\to \infty} \Bigg(\sum_{k\geq 2^n\beta} g(-k2^{-n}) \N_0^{(\beta)}\Big(\mathbf{1}_B\mathbf{1}_{A_{n,k}}\, Z_{k2^{-n}}\Big)\Bigg)\times \N_{2^{-n}}\big(G(\tilde W)\big) = \N_0^{(\beta)}\Big(\mathbf{1}_B\,g(V_{u^\delta_1})\,G(\tilde W^{(u^\delta_1)})\,\mathbf{1}_{\{N_\delta\geq 1\}}\Big).
\end{equation}
In the particular case $G=\mathbf{1}_{\{M>\delta\}}$ this gives
\begin{equation}
\label{seq-excu-key4}
\lim_{n\to \infty} \Bigg(\sum_{k\geq 2^n\beta} g(-k2^{-n}) \N_0^{(\beta)}\Big(\mathbf{1}_B\mathbf{1}_{A_{n,k}}\, Z_{k2^{-n}}\Big)\Bigg)\times \N_{2^{-n}}\big(\tilde M>\delta\big) = \N_0^{(\beta)}\Big(\mathbf{1}_B\,g(V_{u^\delta_1})\,\mathbf{1}_{\{N_\delta\geq 1\}}\Big),
\end{equation}
since $M(\tilde W^{(u^\delta_1)})>\delta$ by construction. It follows from 
\eqref{seq-excu-key3} and \eqref{seq-excu-key4} that
\begin{align*}
\N_0^{(\beta)}\Big(\mathbf{1}_B\,g(V_{u^\delta_1})\,G(\tilde W^{(u^\delta_1)})\,\mathbf{1}_{\{N_\delta\geq 1\}}\Big)
&= \N_0^{(\beta)}\Big(\mathbf{1}_B\,g(V_{u^\delta_1})\,\mathbf{1}_{\{N_\delta\geq 1\}}\Big) \times \lim_{n\to\infty}
\frac{\N_{2^{-n}}(G(\tilde W))}{\N_{2^{-n}}(\tilde M>\delta)}\\
& = \N_0^{(\beta)}\Big(\mathbf{1}_B\,g(V_{u^\delta_1})\,\mathbf{1}_{\{N_\delta\geq 1\}}\Big)
\times \N^*_0(G\mid M>\delta),
\end{align*}
by Corollary \ref{approx-*}. The last display shows both that $\tilde W^{(u^\delta_1)}$ is distributed according to $\N^*_0(\cdot\mid M>\delta)$ under $\N_0^{(\beta)}(\cdot\mid N_\delta\geq 1)$ (take a sequence of functions $g$ that increase to the indicator function of $(-\infty,-\beta)$)
and that $\tilde W^{(u^\delta_1)}$ is independent of the $\sigma$-field generated by $V_{u^\delta_1}$
and $\mathcal{E}^{(\beta,\infty)}$, still under $\N_0^{(\beta)}(\cdot\mid N_\delta\geq 1)$. 

We have obtained that the law of the first excursion above the minimum with height greater than $\delta$
and level smaller than $-\beta$,
under $\N^{(\beta)}_0(\cdot\mid N_\delta \geq 1)$, is $\N^*_0(\cdot\mid M>\delta)$. By letting $\beta$
tend to $0$, we deduce that the law of the  the first excursion above the minimum with height greater than $\delta$,
under $\N_0(\cdot\mid N^\circ_\delta\geq 1)$, is also $\N^*_0(\cdot\mid M>\delta)$ -- we recall our notation $N^\circ_\delta$
for the total number of excursion debuts with height greater than $\delta$.
Moreover, the same passage to the limit shows that this first excursion is independent of the level at which it occurs.
These remarks will be useful in the second part of the proof.

The general statement of the proposition can be deduced 
from the special case $j=1$, via an induction argument
using the special Markov property. Let us explain this argument in detail when $j=2$ (the reader will be able
to fill in the details needed for a general value of $j$). Let $G_1$ and $G_2$ be two nonnegative
measurable functions on $\S$, and consider again $B\in \mathcal{E}^{(-\beta,\infty)}$. 
Recall that $T^\delta_1>\beta$ by definition. By monotone convergence, we have
\begin{align}
\label{seq-excu-1}
&\N_0^{(\beta)}\Big(\mathbf{1}_B\, G_1(\tilde W^{(u^\delta_1)})\,G_2(\tilde W^{(u^\delta_2)})\,\mathbf{1}_{\{N_\delta\geq 2\}}\Big)\\
&\qquad=\lim_{n\to\infty} \sum_{k\geq 2^n\beta}
\N_0^{(\beta)}\Big(\mathbf{1}_B\,G_1(\tilde W^{(u^\delta_1)})\,G_2(\tilde W^{(u^\delta_2)})\,\mathbf{1}_{\{k2^{-n}\leq T^\delta_1<(k+1)2^{-n}\leq T^\delta_2<\infty\}}
\Big).\nonumber
\end{align}
Then, for every $k\geq 2^n\beta$, noting that $\mathbf{1}_B\,G_1(\tilde W^{(u^\delta_1)})\,\mathbf{1}_{\{T^\delta_1<(k+1)2^{-n}\}}$ is 
$\mathcal{E}^{(-(k+1)2^{-n},\infty)}$-measurable, we get
\begin{align}
\label{seq-excu-11}
&\N_0^{(\beta)}\Big(\mathbf{1}_B\,G_1(\tilde W^{(u^\delta_1)})\,G_2(\tilde W^{(u^\delta_2)})\,\mathbf{1}_{\{k2^{-n}\leq T^\delta_1<(k+1)2^{-n}\leq T^\delta_2<\infty\}}
\Big)\\
&\ = \N_0^{(\beta)}\Big(\mathbf{1}_B\,G_1(\tilde W^{(u^\delta_1)})\,\mathbf{1}_{\{k2^{-n}\leq T^\delta_1<(k+1)2^{-n}\leq T^\delta_2\}}
\,\N_0^{(\beta)}\Big(G_2(\tilde W^{(u^\delta_2)})\,\mathbf{1}_{\{T^\delta_2<\infty\}}\,\Big|\,\mathcal{E}^{(-(k+1)2^{-n},\infty)}\Big)\Big).
\nonumber
\end{align}
Applying the special Markov property (Proposition \ref{SMP}) to the interval $(-(k+1)2^{-n},\infty)$ now gives on the event
$\{T^\delta_1<(k+1)2^{-n}\leq T^\delta_2\}$,
\begin{equation}
\label{seq-excu-2}
\N_0^{(\beta)}\Big(G_2(\tilde W^{(u^\delta_2)})\,\mathbf{1}_{\{T^\delta_2<\infty\}}\,\Big|\,\mathcal{E}^{(-(k+1)2^{-n},\infty)}\Big)
= \Big(1-\exp(-Z_{(k+1)2^{-n}}\,\N_0(N_\delta\geq 1))\Big)\,\N_0^*(G_2\mid M>\delta).
\end{equation}
Let us explain this. From the special Markov property, there is a Poisson number $\nu$
with parameter $Z_{(k+1)2^{-n}}\N_0(N^\circ_\delta\geq 1)$ of Brownian snake excursions
outside $(-(k+1)2^{-n},\infty)$ that contain at least one excursion debut with height greater than $\delta$,
and these excursions are independent and distributed according to $\N_{0}(\cdot \mid N^\circ_\delta\geq 1)$, modulo the
obvious translation by $(k+1)2^{-n}$. For each of these $\nu$ excursions, the first excursion above the minimum with height
greater than $\delta$ is distributed according to $\N^*_0(\cdot\mid M>\delta)$, and is independent of the level
at which it occurs (by the first part of the proof). On the event $\{T^\delta_1<(k+1)2^{-n}\leq T^\delta_2\}$, $\tilde W^{(u^\delta_2)}$
is well defined if $T^\delta_2<\infty$, which is equivalent to $\nu\geq 1$, and is
obtained by taking, among these first excursions above the mimimum with height
greater than $\delta$, the one that occurs at the highest level. Clearly it is also distributed according to $\N^*_0(\cdot\mid M>\delta)$.

Since  we have $1-\exp(-Z_{(k+1)2^{-n}}\,\N_0(N^\circ_\delta\geq 1))=\N_0^{(\beta)}(T^\delta_2<\infty\mid\mathcal{E}^{(-(k+1)2^{-n},\infty)})$
on the event
$\{T^\delta_1<(k+1)2^{-n}\leq T^\delta_2\}$, we deduce from \eqref{seq-excu-11} and \eqref{seq-excu-2} that, for every $k\geq 2^n\beta$,
\begin{align*}
&\N_0^{(\beta)}\Big(\mathbf{1}_B\,G_1(\tilde W^{(u^\delta_1)})\,G_2(\tilde W^{(u^\delta_2)})\,\mathbf{1}_{\{k2^{-n}\leq T^\delta_1<(k+1)2^{-n}\leq T^\delta_2<\infty\}}\Big)\\
&\qquad=\N_0^{(\beta)}\Big(\mathbf{1}_B\,G_1(\tilde W^{(u^\delta_1)})\,\mathbf{1}_{\{k2^{-n}\leq T^\delta_1<(k+1)2^{-n}\leq T^\delta_2\}}\,
\N_0^{(\beta)}(T^\delta_2<\infty\mid\mathcal{E}^{(-(k+1)2^{-n},\infty)})\Big)\times \N_0^*(G_2\mid M>\delta)\\
&\qquad= \N_0^{(\beta)}\Big(\mathbf{1}_B\,G_1(\tilde W^{(u^\delta_1)})\,\mathbf{1}_{\{k2^{-n}\leq T^\delta_1<(k+1)2^{-n}\leq T^\delta_2<\infty\}}\Big)\times \N_0^*(G_2\mid M>\delta).
\end{align*}
Finally, returning to \eqref{seq-excu-1}, we obtain by monotone convergence
$$\N_0^{(\beta)}\Big( \mathbf{1}_B\,G_1(\tilde W^{(u^\delta_1)})\,G_2(\tilde W^{(u^\delta_2)})\,\mathbf{1}_{\{N_\delta\geq 2\}}\Big)
= \N_0^{(\beta)}\Big(\mathbf{1}_B\,G_1(\tilde W^{(u^\delta_1)})\,\mathbf{1}_{\{N_\delta\geq 2\}}\Big)\,\N_0^*(G_2\mid M>\delta).$$
This gives the case $j=2$ of the proposition.
\end{proof}

\noindent{\it Remark}. We could have shortened the proof a little by using a strong version 
of the special Markov property (applying to a random interval $(-T,\infty)$) of the type discussed in \cite{Hull}. 

\smallskip

The next lemma shows that the sequence $(\tilde W^{(u^\delta_1)},\ldots,\tilde W^{(u^\delta_{N_\delta})})$
can be viewed as the beginning of an i.i.d. sequence. 

\begin{lemma}
\label{complete-seq}
On an auxiliary probability space $(\Omega,\mathcal{F},\P)$, consider a sequence $(\overline W^{\delta,1},
\overline W^{\delta,2},\ldots)$
of independent random variables distributed according to $\N^*_0(\cdot\mid M>\delta)$. Under the product probability 
measure $\P\otimes \N_0^{(\beta)}$, consider the sequence $(W^{\delta,1}, W^{\delta,2},\ldots)$
defined by
$$W^{\delta,j}=\left\{\begin{array}{ll}
 \tilde W^{(u^\delta_j)}\qquad&\hbox{if }1\leq j\leq N_\delta\\
 \overline W^{\delta,j-N_\delta}\qquad&\hbox{if }j> N_\delta
 \end{array}
 \right.
 $$
 Then $(W^{\delta,1}, W^{\delta,2},\ldots)$ is a sequence of i.i.d. random variables distributed according to
 $\N^*_0(\cdot\mid M>\delta)$, and this sequence is independent of the $\sigma$-field $\mathcal{E}^{(-\beta,\infty)}$.
 \end{lemma}
 
 \begin{proof} This follows from Proposition \ref{sequence-excu} by an argument which is valid in a 
 much more general setting. Let us give a few details. Let $k\geq 2$, and let $\phi_1,\ldots,\phi_k$  be bounded nonnegative 
 measurable functions defined  on $\S$. Also let $B\in \mathcal{E}^{(-\beta,\infty)}$.
 We need to verify that
 \begin{equation}
\label{comple-seq}
E\Big[ \mathbf{1}_B\,\phi_1(W^{\delta,1})\,\phi_2(W^{\delta,2})\cdots \phi_k(W^{\delta,k})\Big]
 = \N_0^{(\beta)}(B)\times \prod_{i=1}^k \N^*_0(\phi_i\mid M>\delta),
 \end{equation}
 where $E[\cdot]$ stands for the expectation under $\P\otimes \N_0^{(\beta)}$.  By dealing separately with the possible values of $N_\delta$ and using the independence of the
 $\overline W^{\delta,j}$'s, we immediately get that
\begin{align*}
&E\Big[\mathbf{1}_{\{N_\delta < k\}}\mathbf{1}_B\,\phi_1(W^{\delta,1})\,\phi_2(W^{\delta,2})\cdots \phi_k(W^{\delta,k})\Big]\\
&\qquad= 
 E\Big[\mathbf{1}_{\{N_\delta < k\}}\mathbf{1}_B\,\phi_1(W^{\delta,1})\cdots \phi_{k-1}(W^{\delta,k-1})\Big] \times  \N^*_0(\phi_k\mid M>\delta).
 \end{align*}
 On the other hand,  Proposition \ref{sequence-excu} exactly says that
\begin{align*}
 &E\Big[\mathbf{1}_{\{N_\delta\geq k\}}\mathbf{1}_B\,\phi_1(W^{\delta,1})\,\phi_2(W^{\delta,2})\cdots \phi_k(W^{\delta,k})\Big]\\
 &\qquad=E\Big[\mathbf{1}_{\{N_\delta\geq k\}}\mathbf{1}_B\,\phi_1( \tilde W^{(u^\delta_1)})\,\phi_2( \tilde W^{(u^\delta_2)})\cdots \phi_k( \tilde W^{(u^\delta_k)})\Big]\\
&\qquad=E\Big[\mathbf{1}_{\{N_\delta\geq k\}}\mathbf{1}_B\,\phi_1(W^{\delta,1})\cdots \phi_{k-1}(W^{\delta,k-1})\Big]\times 
\N^*_0(\phi_k\mid M>\delta).
\end{align*}
By summing the last two displays, we get
$$E\Big[\mathbf{1}_B\,\phi_1(W^{\delta,1})\,\phi_2(W^{\delta,2})\cdots \phi_k(W^{\delta,k})\Big]
= E\Big[\mathbf{1}_B\,\phi_1(W^{\delta,1})\cdots \phi_{k-1}(W^{\delta,k-1})\Big] \times 
\N^*_0(\phi_k\mid M>\delta),$$
and the proof of \eqref{comple-seq} is completed by an induction argument.
 \end{proof}

\subsection{Excursion debuts and discontinuities of the exit measure process}

We start with a first proposition that relates levels of excursion debuts to discontinuity times for the process
$(Z_x)_{x>0}$.

\begin{proposition}
\label{discont}
$\N_0$ a.e., discontinuity times for the process $(Z_x)_{x>0}$ are exactly all reals of the
form $-V_u$ for $u\in D$.
\end{proposition}

\begin{proof}
Recall that, for every $x\geq 0$ we have set
$$\mathcal{Y}_x=\int_0^\sigma \mathrm{d}s\,\mathbf{1}_{\{\tau_{-x}(W_s)=\infty\}}.$$
If $(x_n)$ is a monotone increasing sequence that converges to $x>0$, 
then the indicator functions $\mathbf{1}_{\{\tau_{-x_n}(W_s)=\infty\}}$
converge to $\mathbf{1}_{\{\tau_{-x}(W_s)=\infty\}}$, and by dominated convergence
it follows that $(\mathcal{Y}_x)_{x>0}$ has left-continuous sample paths. On the 
other hand, if  $(x_n)$ is a monotone decreasing sequence that converges to $x>0$,
with $x_n>x$ for every $n$, one immediately gets that
$$\int_0^\sigma\mathrm{d}s\, \mathbf{1}_{\{\tau_{-x_n}(W_s)=\infty\}}\build{\la}_{n\to\infty}^{}
\int_0^\sigma\mathrm{d}s\, \mathbf{1}_{\{\underline W_s\geq -x\}}.$$
It follows that $(\mathcal{Y}_x)_{x>0}$ also has right limits, and that $x$
is a discontinuity point of $\mathcal{Y}$ if and only if
$$\int_0^\sigma \mathrm{d}s\, \mathbf{1}_{\{\underline W_s= -x\}}>0.$$
The latter condition holds if and only if there exists $s\in[0,\sigma]$
such that $\hat W_s>-x$ and $\underline W_s= -x$ (we use the fact
that $\N_0$ a.e.  for every $y\in \R$, $\int_0^\sigma\mathrm{d}s\,\mathbf{1}_{\{\hat W_s=y\}}=0$,
which follows from the existence of local times for the tip process of the Brownian snake, see
e.g. \cite{BMJ}). However, the existence of $s\in[0,\sigma]$
such that $\hat W_s>-x$ and $\underline W_s= -x$ implies that there is an excursion debut $u$
with $V_u=-x$, and the converse is also true. Summarizing, we have obtained that
discontinuity times for the process $(\mathcal{Y}_x)_{x>0}$ are exactly all reals of the
form $-V_u$ for $u\in D$.

To complete the proof of the proposition, we use the fact that 
discontinuity times for $(\mathcal{Y}_x)_{x>0}$ are the same as discontinuity times
for $(Z_x)_{x>0}$, as a consequence of Corollary 4.9 in \cite{Hull} which essentially identifies the
joint distribution of this pair of processes. To be precise the latter
result is not concerned with 
the processes $Z$ and $\mathcal{Y}$ under $\N_0$ but with superpositions of these processes corresponding
to a Poisson measure with intensity $\N_0$. A simple argument however shows that this implies the result we need.
\end{proof}

We now identify the value of the jump of the process $Z$ at the time $-V_u$ when $u\in D$. 
For every $u\in D$, the exit measure $Z^*_0(\tilde W^{(u)})$ makes sense by
\eqref{def-exit-gene}, and can also be defined by the approximation in
Proposition \ref{construct-exit}, using
Proposition \ref{sequence-excu} to relate properties of $\tilde W^{(u)}$ to those valid a.e. under $\N^*_0$. 

\begin{proposition}
\label{identi-jump2}
$\N_0$ a.e. for every $u\in D$, the jump of the process $Z$ at time $-V_u$ is equal to
$Z^*_0(\tilde W^{(u)})$.
\end{proposition}

\begin{proof} We fix $\delta>0$, and we will prove that the assertion of the proposition 
holds $\N^{(\beta)}_0$ a.e.  when $u=u^{\delta}_1$, the first excursion debut with level smaller than $-\beta$ and height greater than $\delta$, on the
event $\{N_\delta\geq 1\}$. We then observe that, for any excursion debut $u$, there are choices of rationals $\beta$ and $\delta$
that make $u$ the first excursion debut with level smaller than $-\beta$ and height greater than $\delta$.
This gives the 
desired result for every $u\in D$.

So from now on we focus on the case $u=u^\delta_1$, and in what follows
we restrict our attention to the event $\{N_\delta\geq 1\}$, so that
$u^\delta_1$ is well defined. Recall that for integers $n\geq 1$
and $k\geq 1$, $(\omega^{k,2^{-n}}_i)_{i\in I^{2^{-n}}_k}$ is the collection of excursions of
the Brownian snake outside $(-k2^{-n},\infty)$, and we 
keep using the notation $\tilde\omega^{k,2^{-n}}_i$ for $\omega^{k,2^{-n}}_i$ translated so that its starting point is $2^{-n}$
and then truncated at level $0$. Let $n_0$ be the first integer such that $2^{n_0}\beta\geq 1$. From now
on we consider values of $n$ such that $n\geq n_0$. We define
 $H_n=\lfloor - 2^nV_{u^\delta_1}\rfloor\geq 1$, in such a way that
\begin{equation}
\label{bound-jump}
H_n\,2^{-n}\leq -V_{u^\delta_1} < (H_n+1)2^{-n}.
\end{equation}
If we set for $\omega\in \S$,
$$O(\omega)=\sup\{ \hat W_s(\omega)-\underline W_s(\omega) : 0\leq s\leq \sigma\},$$
then $H_n$ is the first integer $k\geq 1$ such that 
$O(\tilde\omega^{k,2^{-n}}_i)>\delta$ for some $i\in I^{2^{-n}}_k$. This index $i$ may be not unique, and for this
reason we introduce the event $A_n\subset \{N_\delta\geq 1\}$ where the property $O(\tilde\omega^{k,2^{-n}}_i)>\delta$ holds for $k=H_n$ for exactly one
index $i=i_n\in I^{2^{-n}}_{H_n}$. On the event $A_n$, we let $\omega_{(n)}=\tilde\omega^{H_n,2^{-n}}_{i_n}$
be the corresponding excursion and on the complement of $A_n$
we let $\omega_{(n)}$ be the trivial snake path with duration $0$ in $\S_0$. Notice that, on the event $A_n$, the excursion debut $u^\delta_1$ must then belong to (the subtree coded by
the interval corresponding to) the excursion $\omega^{H_n,2^{-n}}_{i_n}$.
We also note that the sequence $(A_n)_{n\geq n_0}$ 
is monotone increasing, and that $\N^{(\beta)}_0(A_n\mid N_\delta\geq 1)$ converges to $1$ as $n\to\infty$
because there cannot be two excursion debuts at the same level (and therefore, recalling Lemma \ref{large-excu}, two excursion debuts with height 
greater than $\delta$ must be ``macroscopically separated''). 

Furthermore, we claim that the distribution of $\omega_{(n)}$
under $\N^{(\beta)}_0(\cdot\mid A_n)$ is the law of $\tilde W$ under $\N_{2^{-n}}(\cdot\mid O(\tilde W)>\delta)$.  This is basically a consequence
of the special Markov property, but we will provide a few details. Let $\Phi$ be a nonnegative measurable function on $\S$, such that
$\Phi(\omega)=0$ if $O(\omega)\leq \delta$. For every $k\geq 1$, let $B_{n,k}$ be the event where there is a unique index $i\in I^{2^{-n}}_k$ such that 
$O(\tilde\omega^{k,2^{-n}}_i)>\delta$. Then, 
\begin{align*}
\N^{(\beta)}_0\Big( \mathbf{1}_{A_n}\,\mathbf{1}_{\{H_n=k\}}\Phi(\omega_{(n)})\Big)
&=\N^{(\beta)}_0\Bigg( \mathbf{1}_{\{H_n\geq k\}}\,\mathbf{1}_{B_{n,k}}\,\sum_{i\in I^{2^{-n}}_k}\Phi(\tilde\omega^{k,2^{-n}}_i)\Bigg).
\end{align*}
We observe that the event $\{H_n\geq k\}$ is measurable with respect to the $\sigma$-field $\mathcal{E}^{(-k2^{-n},\infty)}$, 
because, if $j<k$, the property $O(\tilde \omega^{j,2^{-n}}_i)>\delta$ for some $ i\in I^{2^{-n}}_j$ can be checked 
from the snake $W$ truncated at level $-k2^{-n}$. Therefore
we can apply the special Markov property, using the fact that, if a Poisson measure with intensity $\mu$ is 
conditioned to have a single atom in a measurable set $C$ of positive and finite $\mu$-measure, the law of this atom
is  $\mu(\cdot\mid C)$. It follows that the quantities in the last display are equal to
$$\N^{(\beta)}_0\Big( \mathbf{1}_{\{H_n\geq k\}}\,\mathbf{1}_{B_{n,k}}\,\N_{2^{-n}}(\Phi(\tilde W)\mid O(\tilde W)>\delta)\Big)
=\N^{(\beta)}_0\Big( \mathbf{1}_{A_n}\,\mathbf{1}_{\{H_n=k\}}\Big)\times  \N_{2^{-n}}(\Phi(\tilde W)\mid O(\tilde W)>\delta).$$
We then sum over $k\geq 1$ to get the desired claim. 

We then note that, for every $n\geq n_0$, we have on the event $A_n$,
\begin{equation}
\label{jump0}
Z_{(H_n+1)2^{-n}}= \sum_{i\in I^{2^{-n}}_{H_n}} \z_0(\tilde \omega^{H_n,2^{-n}}_i)=\z_0(\omega_{(n)})+ \sum_{i\in I^{2^{-n}}_{H_n},i\not=i_n} \z_0(\tilde \omega^{H_n,2^{-n}}_i).
\end{equation}
To simplify notation, we write $b= -V_{u^\delta_1}$. We claim that 
\begin{equation}
\label{jump2}
\sum_{i\in I^{2^{-n}}_{H_n},i\not=i_n} \z_0(\tilde\omega^{H_n,2^{-n}}_i)
\build{\la}_{n\to\infty}^{} Z_{b-},
\end{equation}
where the convergence holds in probability under $\N_0^{(\beta)}(\cdot\mid N_\delta\geq1)$ --- the fact
that  $i_n$ is only defined on $A_n$ creates no problem here since  $\N^{(\beta)}_0(A_n\mid N_\delta\geq 1)$ converges to $1$.

\medskip
\noindent{\it Proof of \eqref{jump2}}. It will be convenient to introduce the point measure
$$\tilde\n^{2^{-n}}_k=\sum_{i\in I^{2^{-n}}_k} \delta_{\tilde \omega^{k,2^{-n}}_i},$$
for every $n\geq 1$ and $k\geq 1$.
We first observe that, on the event $A_n$, we have the equality
$$\sum_{i\in I^{2^{-n}}_{H_n},i\not=i_n} \z_0(\tilde\omega^{H_n,2^{-n}}_i)
=\int_{\{O\leq \delta\}} \tilde\n^{2^{-n}}_{H_n}(\mathrm{d}\omega) \,\z_0(\omega).$$

Since $\N^{(\beta)}_0(A_n\mid N_\delta\geq1)$ converges to $1$, the proof of \eqref{jump2} reduces to checking that
$$\int_{\{O\leq \delta\}} \tilde\n^{2^{-n}}_{H_n}(\mathrm{d}\omega) \,\z_0(\omega) 
\build{\la}_{n\to\infty}^{} Z_{b-}.$$
Since $2^{-n}H_n\uparrow -V_{u^\delta_1}= b$, we have 
$Z_{2^{-n}H_n} \la Z_{b-}$, a.e. under $\N_0^{(\beta)}(\cdot\mid N_\delta\geq 1)$, and so it is enough to prove 
that
$$\int_{\{O\leq \delta\}} \tilde\n^{2^{-n}}_{H_n}(\mathrm{d}\omega) \,\z_0(\omega) - Z_{2^{-n}H_n}
\build{\la}_{n\to\infty}^{} 0.$$

Note that we may have $H_n=\lfloor - 2^nV_{u^\delta_1}\rfloor<2^n\beta$ although $-V_{u^\delta_1}\geq \beta$, but this occurs
with $\N^{(\beta)}_0$-probability tending to $0$. Thanks to this observation, the preceding convergence will
hold provided that, for every $\ve >0$, the quantities in the next display tend to $0$
as $n\to\infty$:
\begin{align*}
&\N_0^{(\beta)}\Bigg(\Big\{\Big|\int_{\{O\leq \delta\}} \tilde\n^{2^{-n}}_{H_n}(\mathrm{d}\omega) \,\z_0(\omega) - Z_{2^{-n}H_n}\Big|>\ve\Big\}
\cap \{N_\delta\geq 1\}\cap\{H_n\geq 2^n\beta\}\Bigg)\\
&\qquad=\sum_{k\geq 2^n\beta} \N_0^{(\beta)}\Bigg(\Big\{\Big|\int_{\{O\leq \delta\}} 
\tilde\n^{2^{-n}}_{k}(\mathrm{d}\omega) \,\z_0(\omega) - Z_{k2^{-n}}\Big|>\ve\Big\}
\cap \{N_\delta\geq 1\} \cap \{H_n= k\}\Bigg)\\
&\qquad=\sum_{k\geq 2^n\beta} \N_0^{(\beta)}\Bigg(\Big\{\Big|\int_{\{O\leq \delta\}} 
\tilde\n^{2^{-n}}_{k}(\mathrm{d}\omega) \,\z_0(\omega) - Z_{k2^{-n}}\Big|>\ve\Big\}
\cap \{\tilde\n^{2^{-n}}_k(O>\delta)\geq 1\} \cap \{H_n\geq k\}\Bigg).
\end{align*}
The last equality holds because the event $\{N_\delta\geq 1\} \cap \{H_n= k\}$ coincides with 
$\{\tilde\n^{2^{-n}}_k(O>\delta)\geq 1\} \cap \{H_n\geq k\}$. Next we recall that the event $\{H_n\geq k\}$ is $\mathcal{E}^{(-k2^{-n},\infty)}$-measurable
and we notice that, under $\N^{(\beta)}_0$, conditionally on $\mathcal{E}^{(-k2^{-n},\infty)}$, $\tilde\n^{2^{-n}}_{k}$ is a Poisson measure whose 
intensity is $Z_{k2^{-n}}$ times the ``law'' of $\tilde W$ under $\N_{2^{-n}}$. It follows that the quantities in the last display are
also equal to 
\begin{equation}
\label{jump2tech}
\sum_{k\geq 2^n\beta} \N_0^{(\beta)}\Big(\psi^n_\ve(Z_{k2^{-n}})\,\mathbf{1}_{ \{\tilde\n^{2^{-n}}_k(O>\delta)\geq 1\} \cap \{H_n\geq k\}}\Bigg),
\end{equation}
where, for every $a\geq 0$,
$$\psi^n_\ve(a)= P\Bigg( \Big|\int_{\{O\leq \delta\}} \n_{n,a}(\mathrm{d}\omega)\,\z_0(\omega) - a\Big| >\ve\Bigg),$$
if $\n_{n,a}$ denotes a Poisson measure whose intensity is $a$ times the ``law'' of $\tilde W$ under $\N_{2^{-n}}$. 
It is easy to verify that $\psi^n_\ve(a)$ tends to $0$ as $n\to\infty$, for every fixed $a$. First note that we can 
remove the restriction to $\{O\leq \delta\}$ since $P(\n_{n,a}(O>\delta)>0)$ tends to $0$. Then we just have
to observe that
$\int \n_{n,a}(\mathrm{d}\omega)\,\z_0(\omega)$ converges in probability to $a$ as $n\to\infty$,
as a straightforward consequence of \eqref{Laplace-exit}. Furthermore, a simple monotonicity argument shows that
the convergence of $\psi^n_\ve(a)$ to $0$ holds uniformly when $a$ varies over a compact subset of $\R_+$.

Finally, using again the fact that $\{\tilde\n^{2^{-n}}_k(O>\delta)\geq 1\} \cap \{H_n\geq k\}=\{N_\delta\geq 1\} \cap \{H_n= k\}$,
the quantity in \eqref{jump2tech} is bounded by
$$\N_0^{(\beta)}\Big(\psi^n_\ve(Z_{2^{-n}H_n})\,\mathbf{1}_{\{N_\delta\geq 1\}}\Big),$$
and this tends to $0$ as $n\to\infty$ by the previous observations and the fact that $\sup\{Z_a:a\geq 0\}<\infty$, $\N_0$ a.e. This completes the proof of our claim \eqref{jump2}. \hfill$\square$

\medskip
Let us complete the proof of the proposition.
We already noticed that the distribution of $\omega_{(n)}$
under $\N^{(\beta)}_0(\cdot\mid A_n)$ is the law of $\tilde W$ under $\N_{2^{-n}}(\cdot\mid O(\tilde W)>\delta)$. We observe that,
for every $\ve>0$, 
the following inclusions hold $\N_{\ve}$ a.e.
$$\{\tilde M>\delta+\ve\}\subset \{O(\tilde W)>\delta\} \subset \{\tilde M>\delta\}$$
and moreover the ratio $\N_{\ve}( \tilde M\geq \delta+\ve)/ \N_{\ve}( \tilde M>\delta)$
tends to $1$ as $\ve\to 0$. It follows that the result of Proposition \ref{conv-exit-M} remains valid if, in the definition
of $W^{\delta,\ve}$, the
conditioning by $\{ \tilde M>\delta\}$ is replaced by $\{O(\tilde W)>\delta\}$. Thanks to this simple
observation, we can deduce  from Proposition \ref{conv-exit-M} that 
\begin{equation}
\label{jump4}
(\omega_{(n)},\z_0(\omega_{(n)}))\build{\la}_{n\to\infty}^{\rm(d)}(W^{\delta,0},Z^*_0(W^{\delta,0})),
\end{equation}
where $W^{\delta,0}$ is distributed according to $\N^*_0(\cdot\mid M>\delta)$
and the convergence holds in distribution under $\N_0^{(\beta)}(\cdot\mid N_\delta\geq 1)$. 
Furthermore, from
the last assertion of Lemma \ref{construct-tech}, and the fact that $\omega^{H_n,2^{-n}}_{i_n}$ is, on the event $A_n$, the excursion
outside $(-H_n2^{-n},\infty)$ that ``contains'' $u^\delta_1$, we get that
$\omega_{(n)}$ converges to $\tilde W^{(u^\delta_1)}$, $\N_0^{(\beta)}$  a.e. on $\{N_\delta\geq 1\}$. 

On the other hand, \eqref{bound-jump}
and the right-continuity of sample paths of $Z$
imply that 
\begin{equation}
\label{jump8}
Z_{(H_n+1)2^{-n}}\build{\la}_{n\to\infty}^{}Z_b,
\end{equation}
$\N_0^{(\beta)}$ a.s. on $\{N_\delta\geq1\}$.
Then, using \eqref{jump0}, \eqref{jump2} and \eqref{jump8}, we immediately get that $\z_0(\omega_{(n)})$ converges
to the random variable $Z_b-Z_{b-}$, in probability under $\N_0^{(\beta)}(\cdot\mid N_\delta\geq 1)$.
So we know that the pair $(\omega_{(n)},\z_0(\omega_{(n)}))$ converges in probability to 
$(\tilde W^{(u^\delta_1)},Z_b-Z_{b-})$ under $\N_0^{(\beta)}(\cdot\mid N_\delta\geq 1)$,
and it follows from \eqref{jump4} that the law of $(\tilde W^{(u^\delta_1)},Z_b-Z_{b-})$ 
under $\N_0^{(\beta)}(\cdot\mid N_\delta\geq 1)$ is the law of $(W^{\delta,0},Z^*_0(W^{\delta,0}))$. This forces
$Z_b-Z_{b-}= Z^*_0(\tilde W^{(u^\delta_1)}))$, which completes the proof.
\end{proof}

\subsection{The Poisson process of excursions}
\label{sec:Poisson}

The following proposition is reminiscent of It\^o's famous Poisson point process 
of excursions of linear Brownian motion. We recall that $\beta>0$ is fixed and that 
$u^\delta_1,\ldots,u^\delta_{N_\delta}$ are the successive excursion debuts 
with height greater than $\delta$ and level smaller than $-\beta$.

\begin{proposition}
\label{PoissonPP}
We can find an auxiliary probability space $(\Omega,\f,\P)$ such that, on the 
product space $\Omega\times \S$ equipped with the 
probability measure $\P\otimes \N^{(\beta)}_0$, we can
 construct a Poisson measure $\mathcal{P}$ on 
$\R_+\times \S$ with intensity $\mathrm{d}t\otimes \N^*_0(\mathrm{d}\omega)$
so that the following holds.  For every $\delta>0$, if $(t^\delta_1,\omega^\delta_1),(t^\delta_2,\omega^\delta_2),\ldots$ is the sequence of atoms
of the measure $\mathcal{P}(\cdot\cap(\R_+\times\{M>\delta\}))$, ranked so that $t^\delta_1<t^\delta_2<\cdots$, we have
$\tilde W^{(u^\delta_i)}=\omega^\delta_i$ for every $1\leq i\leq N_\delta$. Furthermore, the Poisson measure $\mathcal{P}$
is independent of $\mathcal{E}^{(-\beta,\infty)}$. 
\end{proposition}

This proposition means that all excursions above the minimum (with level smaller than $\beta$) can be viewed as the atoms of a certain Poisson point process.
In contrast with the classical It\^o theorem of excursion theory for Brownian motion,
we have enlarged the underlying probability space in order to construct the Poisson measure $\mathcal{P}$. 

\begin{proof}
We first explain how we can choose the auxiliary random variables $\ov{W}^{\delta,j}$ of
Lemma \ref{complete-seq} in a consistent way when $\delta$ varies. 
We set $\delta_k=2^{-k}$ for every $k\geq 1$ and we restrict our attention
to values of $\delta$ in the sequence $(\delta_k)_{k\geq 1}$.
On 
an auxiliary probability space $(\Omega,\mathcal{F},\P)$, let $\ov{\mathcal{P}}$ be a Poisson measure on 
$\R_+\times \S$ with intensity $\mathrm{d}t\otimes \N^*_0(\mathrm{d}\omega)$. 
For every $k\geq 1$, let $(\ov{t}^{k,j},\ov{W}^{k,j})_{j\geq 1}$ be the sequence of atoms of
$\ov{\mathcal{P}}$ that fall in the set $\R_+\times \{M>\delta_k\}$ (ordered so that $\ov{t}^{k,1}<\ov{t}^{k,2}<\cdots$).
Then, for every $k\geq 1$, $(\ov{W}^{k,1},\ov{W}^{k,2},\ldots)$ forms an i.i.d. sequence of variables
distributed according to $\N^*_0(\cdot\mid M>\delta_k)$. By Lemma \ref{complete-seq},
under the product probability 
measure $\P\otimes \N^{(\beta)}_0$, the sequence $(W^{k,1}, W^{k,2},\ldots)$
defined by
$$W^{k,j}=\left\{\begin{array}{ll}
 \tilde W^{(u^{\delta_k}_j)}\qquad&\hbox{if }1\leq j\leq N_\delta\\
 \overline{W}^{k,j-M_\delta}\qquad&\hbox{if }j> N_\delta
 \end{array}
 \right.
 $$
is also a sequence of i.i.d. random variables distributed according to
 $\N^*_0(\cdot\mid M>\delta_k)$, and is independent 
 of the $\sigma$-field $\mathcal{E}^{(-\beta,\infty)}$. 

Obviously, 
 if $k<k'$, the excursions  $\tilde W^{(u^{\delta_k}_j)}$ $1\leq j\leq N_{\delta_k}$ are
 obtained by considering the elements of the finite sequence $\tilde W^{(u^{\delta_{k'}}_j)}$, $1\leq j\leq N_{\delta_{k'}}$
 that belong to the set $\{M>\delta_k\}$, and similarly the sequence $(\ov{W}^{k,j})_{j\geq 1}$ 
 consists of those terms of the
sequence $(\ov{W}^{k',j})_{j\geq 1}$ that belong to the set $\{M>\delta_k\}$. 
 It follows that, for every $k<k'$,
the sequence $(W^{k,j})_{j\geq 1}$ is obtained by keeping only those terms of the
sequence $(W^{k',j})_{j\geq 1}$ that belong to the set $\{M>\delta_k\}$. Note that the law of the collection 
$$( W^{k,j})_{j\geq 1,k\geq 1}$$
is then completely determined by this consistency property and the fact that, for 
every fixed $k\geq 1$,  $(W^{k,j})_{j\geq 1}$ is a sequence of i.i.d. random variables distributed according to
 $\N^*_0(\cdot\mid M>\delta_k)$. In particular,
 \begin{equation}
\label{identi-Poisson}
( W^{k,j})_{j\geq 1,k\geq 1}\build{=}_{}^{\rm(d)} ( \ov{W}^{k,j})_{j\geq 1,k\geq 1}
\end{equation}
 Also note that the collection $( W^{k,j})_{j\geq 1,k\geq 1}$ is independent 
 of the $\sigma$-field $\mathcal{E}^{(-\beta,\infty)}$.

It is a simple exercise on Poisson measures to verify that $\ov{\mathcal{P}}$ is equal a.s. to
a measurable function of the collection $( \ov{W}^{k,j})_{j\geq 1,k\geq 1}$. Indeed, it suffices to verify that the
times $(\ov{t}^{k,j})_{j\geq 1,k\geq 1}$ are (a.s.) measurable functions of this collection. Let us 
outline the argument in the case $k=j=1$. If, for every $k\geq 1$, we write
$$m_k:=\#\{j\geq 1: \ov{t}^{k,j}<\ov{t}^{1,1}\}$$
then $m_k$ is just the number of terms in the sequence $( \ov{W}^{k,j})_{j\geq 1}$ before the first term
that belongs to $\{M>\delta_1\}$, and is thus a function of $( \ov{W}^{\ell,j})_{j\geq 1,\ell\geq 1}$. Elementary arguments 
using Lemma \ref{loi-max} show that
we have the almost sure convergence
$$\N^*_0(M>\delta_{k})^{-1}\; m_k \build{\la}_{k\to\infty}^{} \ov{t}^{1,1},$$
thus giving the desired measurability property. 

So there 
exists a measurable function $\Phi$ such that we have a.s.
$$\ov{\mathcal{P}}= \Phi\Big( (\ov{W}^{k,j})_{j\geq 1,k\geq 1}\Big).$$
Then we can just set
$$\mathcal{P}=\Phi\Big(( W^{k,j})_{j\geq 1,k\geq 1}\Big).$$
By \eqref{identi-Poisson}, $\mathcal{P}$ has the same distribution as $\ov{\mathcal{P}}$. By construction, the 
properties stated in the proposition hold when $\delta=\delta_k$, for every $k\geq 1$. This implies that they
hold for every $\delta>0$.
\end{proof}

In what follows, we will use not only the statement of Proposition \ref{PoissonPP} but also the explicit construction of $\mathcal{P}$ that is given
in the preceding proof (we did not include this explicit construction in the statement of Proposition \ref{PoissonPP}
for the sake of conciseness). 

We now state an important lemma, which shows that the process $(Z_{\beta+r})_{r\geq 0}$
can be recovered from ($Z_\beta$ and) the Poisson measure $\mathcal{P}$. To this end,
we introduce the point measure $\mathcal{P}^\circ$ defined as
the image of $\mathcal{P}$ under the mapping $(t,\omega)\la(t,Z^*_0(\omega))$. 
From the form of the ``law'' of $Z^*_0$ under $\N^*_0$ given in
Proposition \ref{loisZ0sigma}, $\mathcal{P}^\circ$ is (under $\P\otimes \N^{(\beta)}_0$) a Poisson measure
on $\R_+\times (0,\infty)$ with intensity
$$\mathrm{d}t\otimes  \sqrt{\frac{3}{2 \pi}} z^{-5/2}\,\mathrm{d}z.$$
We can associate with this point measure a centered L\'evy process $U=(U_t)_{t\geq 0}$ 
(with no negative jumps) started from $0$, such that
$$\sum_{t\in \mathcal{D}_U} \delta_{(t,\Delta U_t)} = \mathcal{P}^\circ,$$
where $\mathcal{D}_U$ is the set of discontinuity times of $U$. Note that the 
Laplace transform of $U_t$ is
$$E[\exp(-\lambda U_t)] = \exp(t\psi(\lambda)),$$
where 
$$\psi(\lambda)= \sqrt{\frac{3}{2 \pi}} \int_0^\infty (e^{-\lambda z}-1+\lambda z)\,z^{-5/2}\,\mathrm{d}z=\sqrt{\frac{8}{3}} \,\lambda^{3/2}.$$
Notice that we get the same function $\psi(\lambda)$ as in Section \ref{sec:exitprocess}.

\begin{lemma}
\label{exit-recover}
Set $X_t=Z_\beta + U_t$ for every $t\geq 0$. Then, we have, $\P\otimes \N_0^{(\beta)}$ a.s.,
$$Z_{\beta+r}=X_{\inf\{t\geq 0:\int_0^t (X_s)^{-1}\mathrm{d}s >r\}}\;,\quad \hbox{for every }0\leq r<-W_*-\beta.$$
\end{lemma}

\noindent{\it Remark}. We have $Z_r=0$ for every $r\geq -W_*$, so that the formula of the lemma indeed 
expresses $(Z_{\beta+r})_{r\geq 0}$ as a function of $X$, which is itself  defined in terms of 
$Z_\beta$ and the point measure $\mathcal{P}^\circ$. 

\begin{proof}
First notice that $(U_t)_{t\geq 0}$ is independent of $Z_\beta$ because $\mathcal{P}$ is
independent of $\mathcal{E}^{(-\beta,\infty)}$. Therefore, $(X_t)_{t\geq 0}$ is a L\'evy process 
started from $Z_\beta$. On the other hand, we know that $(Z_{\beta+r})_{r\geq 0}$ is under 
$\N_0^{(\beta)}$ a continuous-state branching process with branching mechanism $\psi$. By the
classical Lamperti transformation (see e.g. \cite{Bin}), if we set $T'_0:=\int_0^\infty Z_{\beta+t}\,\mathrm{d}t$ and, for every $0\leq r < T'_0$,
\begin{equation}
\label{Lamperti1}
X'_r:=Z_{\beta + \inf\{s\geq 0:\int_0^s Z_{\beta+t}\,\mathrm{d}t>r\}},
\end{equation}
the process $(X'_r)_{0\leq r< T'_0}$ has the same distribution as $(X_r)_{0\leq r<T_0}$, where
$T_0:=\inf\{t\geq 0: X_t=0\}$. Furthermore, by inverting \eqref{Lamperti1}, we have also 
\begin{equation}
\label{Lamperti2}
Z_{\beta+r}=X'_{\inf\{t\geq 0:\int_0^t (X'_s)^{-1}\mathrm{d}s >r\}}\;,\quad \hbox{for every }0\leq r<T^Z_0,
\end{equation}
where $T^Z_0=-W_*-\beta$ is the hitting time of $0$ by $Z$. 

Comparing \eqref{Lamperti2} with the statement of the lemma, we see that we only need to verify that we have the
a.s. equality $(X_r)_{0\leq r<T_0}=(X'_r)_{0\leq r< T'_0}$. 
To this end, we first extend the definition of $X'_t$ to values $t\geq T_0$. Recalling the Poisson measure $\ov{\mathcal{P}}$ in the
proof of Proposition \ref{PoissonPP},
we define $\overline{\mathcal{P}}^\circ$ as
the image of $\overline{\mathcal{P}}$ under the mapping $(t,\omega)\mapsto(t,Z^*_0(\omega))$, and associate with
$\overline{\mathcal{P}}^\circ$ a L\'evy process $(\overline U_t)_{t\geq 0}$ having the same distribution as
$(U_t)_{t\geq 0}$. We complete the definition of $X'$ by setting for every $t\geq 0$,
$$X'_{T'_0+t}= \overline U_t.$$

We then observe that $X$ and $X'$ are two L\'evy processes with the same distribution
and the same (random) initial value $Z_\beta$. Furthermore, a.s. for every $\alpha>0$, the ordered sequence of jumps of
size greater than $\alpha$ is the same for $X'$ and for $X$. First note that the jumps of $X'$  that occur before the hitting time of
$0$ are the same as the jumps of $Z$ after time $\beta$, and, by Proposition \ref{identi-jump2}, these are exactly the quantities
$Z^*_0(\tilde W^{(u)})$ when $u$ varies over the excursion debuts with level smaller than $-\beta$. Recalling our construction
of $X$ from the point measure $\mathcal{P}^\circ$, we obtain that, for every $\alpha>0$, the ordered sequence 
of jumps of $X'$ of
size greater than $\alpha$ that occur before the hitting time of
$0$  will also appear as the first $n_\alpha$ jumps of $X$ of size greater than $\alpha$,  for some random integer 
$n_\alpha$
depending on $\alpha$. Then, the ordered sequence of jumps of $X'$ of size greater than $\alpha$  that occur after the hitting time of
$0$ consists of the quantities $Z^*_0(\omega)$ where $(t,\omega)$ varies over the atoms of 
$\overline{\mathcal{P}}$ such that $Z^*_0(\omega)>\alpha$ and these quantities are ranked according to the
values of $t$. Recalling the way $\mathcal{P}$ was defined, we see that the same sequence will 
appear as the sequence of jumps of $X$ of size greater than $\alpha$ occurring after the $n_\alpha$-th one. 

Finally, once we know that, for every $\alpha>0$, the ordered sequence of jumps of
size greater than $\alpha$ is the same for $X'$ and for $X$, the fact that $X$ and $X'$ are two L\'evy processes with the same distribution
and the same initial value implies that they are a.s. equal, which completes the proof.
\end{proof}

\subsection{The main theorem}
\label{subsec:main}

Our main result identifies the conditional distribution of excursions above the
minimum given the exit measure process $Z$. We let $\mathcal{D}_Z$ stand for the set of all jump times of $Z$.
Recall from Proposition \ref{discont}
that there is a one-to-one correspondence between $\mathcal{D}_Z$
and excursions above the minimum. If $u$ is an excursion debut, and $r=-V_u$
is the associated element of $\mathcal{D}_Z$, we write $\tilde W^{(r)}=\tilde W^{(u)}$
in the following statement. We let $\D(0,\infty)$ stand for the usual Skorokhod space
of c\`adl\`ag functions from $(0,\infty)$ into $\R$. 

\begin{theorem}
\label{mainT}
Let $F$ be a nonnegative measurable function on $\D(0,\infty)$, and let 
$G$ be a nonnegative measurable function on $\R_+\times \S$.
Then,
$$\N_0\Bigg(F(Z)\,\exp\Big(-\sum_{r\in \mathcal{D}_Z} G(r,\tilde W^{(r)})\Big)\Bigg)
=\N_0\Bigg(F(Z)\,\prod_{r\in  \mathcal{D}_Z} \N^*_0\Big(\exp(-G(r,\cdot))\,\Big| \,Z^*_0=\Delta Z_r\Big)\Bigg).$$
In other words, under $\N_0$ and conditionally on the exit measure process $Z$, the 
excursions above the minimum are independent, and, for every $r\in  \mathcal{D}_Z$, the
conditional law of the associated excursion is $\N^*_0(\cdot\mid Z^*_0=\Delta Z_r)$.
\end{theorem}

\begin{proof} Let us a start with simple reductions of the proof. First we may assume that $\N_0(F(Z))<\infty$
since the general case will follow by monotone convergence. Then, we may assume that
$G(r,\omega)=0$ if  $r\leq \gamma$, for some $\gamma>0$, and it is also sufficient
to prove that the statement holds when $\N_0$ is replaced by $\N^{(\beta)}_0$ for some fixed $\beta >0$.
Finally, we may restrict the sum or the product over $r$ to jump times such that $\Delta Z_r>\alpha$, for some
fixed $\alpha>0$.

In view of the preceding observations, we only need to verify that, for every $\alpha>0$ and $\beta>0$,
$$\N_0^{(\beta)}\Bigg(F(Z)\,\exp\Bigg(-\build{\sum_{r\in \mathcal{D}_Z^{(\beta)}}}_{\Delta Z_r>\alpha}^{} G(r,\tilde W^{(r)})\Bigg)\Bigg)
=\N_0^{(\beta)}
\Bigg(F(Z)\,\build{\prod_{r\in \mathcal{D}_Z^{(\beta)}}}_{\Delta Z_r>\alpha}^{}\N^*_0\Big(\exp(-G(r,\cdot))\,\Big| \,Z^*_0=\Delta Z_r\Big)\Bigg),$$
where $\mathcal{D}_Z^{(\beta)}=\mathcal{D}_Z\cap(\beta,\infty)$.

From now on, we fix $\alpha>0$ and $\beta>0$. We will use the notation and definitions of the previous subsections, 
where $\beta>0$ was fixed and we argued under $\N_0^{(\beta)}$. In particular it will be convenient to
consider the product probability measure $\P\otimes \N_0^{(\beta)}$ as in Section \ref{sec:Poisson}. Recall 
the definition of the Poisson measure $\mathcal{P}$ and of the process $X$ in Lemma \ref{exit-recover}
(these objects depend on the choice of $\beta$, which is fixed here), and the notation $T_0=\inf\{t\geq 0: X_t=0\}$.
Also recall that $\mathcal{P}^\circ $ is the image of $\mathcal{P}$ under the mapping $(t,\omega)\mapsto (t,Z_0^*(\omega))$.

The first step is to rewrite the quantity
$$\build{\sum_{r\in \mathcal{D}_Z^{(\beta)}}}_{\Delta Z_r>\alpha}^{} G(r,\tilde W^{(r)})
$$
in a different form. Recall from Lemma \ref{exit-recover} that every jump time $r$ of
$Z$ after time $\beta$, hence every excursion debut $u$ with level smaller than $-\beta$, corresponds to
a jump time of $X$ before time $T_0$, and is 
therefore associated with an atom $(t,\omega)$ of $\mathcal{P}$, with $t<T_0$, such that $\omega= \tilde W^{(u)}$
and $Z^*_0(\omega)=Z^*_0(\tilde W^{(u)})=\Delta Z_r$ , where the last equality is Proposition \ref{identi-jump2}. Then,
let $(t^\alpha_1,\omega^\alpha_1),(t^\alpha_2,\omega^\alpha_2),\ldots$ be the time-ordered sequence of all
atoms $(t,\omega)$ of $\mathcal{P}$ such that $Z^*_0(\omega)>\alpha$. Also set $n_\alpha=\max\{i\geq 1: t^\alpha_i<T_0\}$.
For every $1\leq i\leq n_\alpha$, write $z^\alpha_i=Z^*_0(\omega^\alpha_i)$ and 
$r^{\alpha}_i$ for the jump time of $Z$ corresponding to the jump $z^\alpha_i$.
We can rewrite
$$\build{\sum_{r\in \mathcal{D}_Z^{(\beta)}}}_{\Delta Z_r>\alpha}^{} G(r,\tilde W^{(r)})= \sum_{i=1}^{n_\alpha} G(r^\alpha_i,\omega^\alpha_i).$$

Writing $E[\cdot]$ for the expectation under $\P\otimes \N_0^{(\beta)}$, we then have
$$\N_0^{(\beta)}\Bigg(F(Z)\,\exp\Bigg(-\build{\sum_{r\in \mathcal{D}_Z^{(\beta)}}}_{\Delta Z_r>\alpha}^{} G(r,\tilde W^{(r)})\Bigg)\Bigg)=E\Big[F(Z)\,\exp\Big(-\sum_{i=1}^{n_\alpha} G(r^\alpha_i,\omega^\alpha_i)\Big)\Big],$$
We evaluate the right-hand side by conditioning first with respect to the $\sigma$-field $\mathcal{H}$
generated by $\mathcal{E}^{(-\beta,\infty)}$ and the point measure $\mathcal{P}^\circ$. Notice that the process $Z$
is measurable with respect to $\mathcal{H}$ (because $U$ is obviously a measurable function of $\mathcal{P}^\circ$, and we can use Lemma \ref{exit-recover}). The finite sequence $r^\alpha_1,\ldots,r^\alpha_{n_\alpha}$
is also measurable with respect to $\mathcal{H}$ as it is the sequence of jump times of $Z$ (after time $\beta$)
corresponding to jumps of size greater than $\alpha$. In particular, $n_\alpha$ is measurable with respect to $\mathcal{H}$. 
Finally the quantities $z^\alpha_1,\ldots,z^\alpha_{n_\alpha}$ are the corresponding jumps and therefore are also
measurable with respect to $\mathcal{H}$. 

On the other hand, by standard properties of Poisson measures, we know that the sequence 
$\omega^\alpha_1,\omega^\alpha_2,\ldots$ is a sequence of i.i.d. variables distributed according to
$\N^*_0(\cdot\mid Z^*_0>\alpha)$. Recalling that $\mathcal{P}$ is independent of $\mathcal{E}^{(-\beta,\infty)}$, we
see that conditioning this sequence on the $\sigma$-field $\mathcal{H}$ has the effect of conditioning on the values of
$Z^*_0(\omega^\alpha_1),Z^*_0(\omega^\alpha_2),\ldots$. In a more precise way,  the conditional
distribution of $\omega^\alpha_1,\omega^\alpha_2,\ldots$ knowing $\mathcal{H}$ is the distribution of
a sequence of independent variables distributed respectively according to $\N^*_0(\cdot\mid Z^*_0=z^\alpha_1),
\,\N^*_0(\cdot\mid Z^*_0=z^\alpha_2),\ldots$, where these conditional measures are defined thanks to Proposition \ref{conditioned-exit}.  

By combining the preceding considerations, we get
$$E\Big[F(Z)\,\exp\Big(-\sum_{i=1}^{n_\alpha} G(r^\alpha_i,\omega^\alpha_i)\Big)\Big]
= E\Big[F(Z)\,\prod_{i=1}^{n_\alpha} \N^*_0\big(\exp(-G(r^\alpha_i,\cdot)\mid Z^*_0=z^\alpha_i\big)\Big].$$
Now note that, with our definitions,
$$  \prod_{i=1}^{n_\alpha} \N^*_0\big(\exp(-G(r^\alpha_i,\cdot)\mid Z^*_0=z^\alpha_i\big)=
 \build{\prod_{r\in \mathcal{D}_Z^{(\beta)}}}_{\Delta Z_r>\alpha}^{}\N^*_0\Big(\exp(-G(r,\cdot))\,\Big| \,Z^*_0=\Delta Z_r\Big),$$
 and so we have obtained
 \begin{align*}
 \N_0^{(\beta)}\Bigg(F(Z)\,\exp\Bigg(-\build{\sum_{r\in \mathcal{D}_Z^{(\beta)}}}_{\Delta Z_r>\alpha}^{} G(r,\tilde W^{(r)})\Bigg)\Bigg)
 &=E\Big[F(Z)\,\build{\prod_{r\in \mathcal{D}_Z^{(\beta)}}}_{\Delta Z_r>\alpha}^{}\N^*_0\Big(\exp(-G(r,\cdot))\,\Big| \,Z^*_0=\Delta Z_r\Big)\Big]\\
 &=\N^{(\beta)}_0\Bigg(F(Z)\,\build{\prod_{r\in \mathcal{D}_Z^{(\beta)}}}_{\Delta Z_r>\alpha}^{}\N^*_0\Big(\exp(-G(r,\cdot))\,\Big| \,Z^*_0=\Delta Z_r\Big)\Bigg),
 \end{align*}
 which completes the proof of the theorem.
\end{proof}

\section{Excursions away from a point}
\label{excupoint}

In this section, we briefly explain how we can derive the results stated in the 
introduction from our statements concerning excursions above the minimum.
This derivation relies on the famous theorem of L\'evy stating that, if $(B_t)_{t\geq 0}$
is a linear Brownian motion starting from $0$, and if $(L^0_t(B))_{t\geq 0}$ is its local time process
at $0$, then the pair of processes
$$(B_t-\min\{B_r:0\leq r\leq t\},-\min\{B_r:0\leq r\leq t\})_{t\geq 0}$$
has the same distribution as $(|B_t|,L^0_t(B))_{t\geq 0}$. Notice that 
$L^0_t(B)$ can also be interpreted as the local time of $|B|$ at $0$, provided we
consider here the ``symmetric local time'', namely
$$L^0_t(|B|)=\lim_{\ve\to 0} \frac{1}{2\ve}\int_0^t \mathbf{1}_{[-\ve,\ve]}(|B_r|)\,\mathrm{d}r.$$

L\'evy's  identity will show that
(absolute values of) excursions away from $0$
 for our tree-indexed process have the same distribution as
excursions above the minimum, which is essentially what we need to 
derive the results stated in the 
introduction.

Let us explain this in greater detail. For any finite path $\w\in \W_0$,
define two other finite paths $\w^\bullet$
and $\ell^\bullet_\w$ with the same lifetime as $\w$ by the formulas
\begin{align*}
\w^\bullet(t)&:=\w(t)-\min\{\w(r):0\leq r\leq t\}\\
\ell^\bullet_{\w}(t)&:=-\min\{\w(r):0\leq r\leq t\}.
\end{align*}
On our canonical space $\mathcal{S}_0$ of snake trajectories, we 
can then make sense of $W^\bullet_s$ and $\ell^\bullet_{W_s}$
for every $s\geq 0$, and we write $L^\bullet_s=\ell^\bullet_{W_s}$
to simplify notation. Then, under $\N_0$, the pair $(W^\bullet_s,L^\bullet_s)_{s\geq 0}$ 
defines a random element of the space of two-dimensional snake trajectories
with initial point $(0,0)$ (the latter space 
is defined by an obvious extension of Definition \ref{def:snakepaths}). 
Thanks to L\'evy's theorem recalled above, it is then a simple matter to
verify that the ``law'' of the pair $(W^\bullet_s,L^\bullet_s)_{s\geq 0}$
under $\N_0$ is
the excursion measure from the point $(0,0)$ of the Brownian snake whose 
spatial motion is the Markov process $(|B_t|,L^0_t(B))$. We refer to
\cite[Chapter IV]{livrevert} for the definition of the Brownian snake 
associated with a general spatial motion and of its excursion measures.
In a way similar to the beginning of Section \ref{sec:construct}, we then set
$$V^\bullet_u=\hat W^\bullet_s= \hat W_s-\min\{W_s(t):0\leq t\leq \zeta_s\}
=V_u - \min\{V_v:v\in \llbracket \rho,u \rrbracket\},$$
for every $u\in\t_\zeta$ and $s\geq 0$ such that $p_\zeta(s)=u$.

Say that $u\in \t_\zeta$ is an excursion debut away from $0$ for $V^\bullet$
if 
\begin{enumerate}
\item[(i)] $V^\bullet_u=0$;
\item[(ii)] $u$ has a strict descendant $w$ such that $V^\bullet_v\not =0$
for all $v\in \rrbracket \rho,w \rrbracket$.
\end{enumerate}
It follows from our definitions that $u$ is an excursion debut away from $0$ for $V^\bullet$
if and only if $u$ is an excursion debut above the minimum in the sense of Section 
 \ref{sec:construct}, that is, if and only if $u\in D$. Then, Proposition
 \ref{connec-compo} shows that the connected
 components of the open set $\{u\in\t_\zeta: V^\bullet_u>0\}$
 are exactly the sets $\mathrm{Int}(C_u)$, $u\in D$. Furthermore, for every
 $u\in D$, the values of $V^\bullet$ over $C_u$ are described by the snake trajectory
 $\tilde W^{(u)}$ (which can thus be viewed as the excursion of $V^\bullet$
 away from $0$ corresponding to $u$). 
 
 In order to recover the setting of the introduction, we still need to assign signs 
 to the excursions of $V^\bullet$ away from $0$. To this end, we let
 $(v_1,v_2,\ldots)$ be a measurable enumeration of $D$ -- formally we should
 rather enumerate times $s_1,s_2,\ldots$ such that $p_\zeta(s_1)=v_1,
 p_\zeta(s_2)=v_2,\ldots$. On an auxiliary probability space
 $(\Omega,\f,\P)$, we then consider a sequence $(\xi_1,\xi_2,\ldots)$
 of i.i.d. random variables such that
 $$\P(\xi_i=1)=\P(\xi_i=-1)=\frac{1}{2}$$
 for every $i\geq 1$. Under the product measure $\P\otimes \N_0$, we then set, for every $u\in \t_\zeta$,
 $$V^*_u:=\left\{
 \begin{array}{ll}
 \xi_i\,V^\bullet_u\qquad&\hbox{if }u\in \mathrm{Int}(C_{v_i})\hbox{ for some }i\geq 1,\\
 0&\hbox{if }V^\bullet_u=0.
 \end{array}
 \right.
 $$
 The fact that $u\mapsto V^\bullet_u$ is continuous implies that 
 $u\mapsto V^*_u$ is also continuous on $\t_\zeta$. Furthermore the
 pair $(V^*_{p_\zeta(s)},\zeta_s)_{s\geq 0}$ is a tree-like path, and we denote the
 associated snake trajectory by $(W^*_s)_{s\geq 0}$. Then, the ``law''
 of $(W^*_s)_{s\geq 0}$ under $\P\otimes \N_0$ is just the excursion
 measure $\N_0$. This is a consequence of the fact that,
 starting from a process distributed as $(|B_t|)_{t\geq 0}$, one can reconstruct
 a linear Brownian motion started from $0$  by assigning independently
 signs $+1$ or $-1$ with probability $1/2$ to the excursions away from
 $0$. We omit the details.
 
Since the law
 of $(W^*_s)_{s\geq 0}$ under $\P\otimes \N_0$ is  $\N_0$,
we may replace the process $(W_s)_{s\geq 0}$ under $\N_0$ by the process
 $(W^*_s)_{s\geq 0}$ under $\P\otimes \N_0$ in order to prove the various statements of the introduction.
 To this end, we first notice that
the excursion debuts away from $0$ for $V^*$ (obviously defined by 
properties (i) and (ii) with $V^\bullet$ replaced by $V^*$) are the same
 as excursion debuts away from $0$ for $V^\bullet$, and thus the same as 
 excursion debuts above the minimum in the sense of Section 
 \ref{sec:construct}. Moreover, for 
 every $i=1,2,\ldots$, the excursion of $V^*$
 corresponding to $v_i$ is described by
 $$\tilde W^{*(v_i)}=\left\{
 \begin{array}{ll}
 \tilde W^{(v_i)}\qquad&\hbox{if } \xi_i=1,\\
 -\tilde W^{(v_i)}\qquad&\hbox{if }\xi_i=-1.
 \end{array}
  \right.
  $$
 In addition, if $a_i$ is such that $p_\zeta(a_i)=v_i$, the local time at $0$ of the path $W^*_{a_i}$
 is equal to the (symmetric) local time at $0$ of $|W^*_{a_i}|=W^\bullet_{a_i}$,
 $$\ell^*_i= {\hat L}^\bullet_{a_i}= -\underline{W}_{a_i}= - V_{v_i}.$$

From the preceding remarks, it is now easy to derive Theorem \ref{construct} from 
Theorem \ref{construct-N}. Indeed, the left hand side of the formula of 
Theorem \ref{construct} can be rewritten as
$$\P\otimes \N_0\Big(\sum_{i=1}^\infty \Phi(\ell^*_i,W^{*(v_i)})\Big)$$
and, by
the previous observations, the last display is equal to
\begin{align*}
\P\otimes \N_0\Big(\sum_{i=1}^\infty \Phi(-V_{v_i},\xi_i \tilde W^{(v_i)})\Big)
 &= \frac{1}{2}\,\N_0\Big(\sum_{i=1}^\infty (\Phi(-V_{v_i},\tilde W^{(v_i)})
 + \Phi(-V_{v_i},-\tilde W^{(v_i)})\Big) \\
 &= \frac{1}{2}\int \N^*_0(\mathrm{d}\omega)\Big(\int_0^\infty \mathrm{d}x\,(\Phi(x,\omega)+\Phi(x,-\omega))\Big)
 \end{align*}
 where the last equality follows from Theorem \ref{construct-N}. This shows that Theorem \ref{construct} holds
 with $\M_0=\frac{1}{2}(\N^*_0 + \check \N^*_0)$, where $\check\N^*_0$ is the image of 
 $\N^*_0$ under $\omega \mapsto -\omega$. Then Proposition \ref{exit-definition} follows
 from Proposition \ref{construct-exit}. 
 
 In order to derive Proposition \ref{identi-jump}, we note that, for every $r>0$, the (total mass of the) exit measure of
 the snake $(W^\bullet,L^\bullet)$ outside the open set $\Delta_r:=\R_+\times [0,r)$, which is denoted by $\mathcal{X}_r$, satisfies the following approximation
 $\N_0$ a.e.,
 $$\mathcal{X}_r=\lim_{\ve\to 0} \frac{1}{\ve} \int_0^\sigma \mathrm{d}s\,\mathbf{1}_{\{\zeta_s-\ve< \tau_{\Delta_r}(W_s^\bullet,L^\bullet_s)
 < \zeta_s\}},$$
 where $\tau_{\Delta_r}(W_s^\bullet,L^\bullet_s)$ stands for the first exit time from $\Delta_r$
 of the path $(W_s^\bullet(t),L^\bullet_s(t))_{0\leq t\leq \zeta_s}$. This is indeed the analog of the approximation
 result \eqref{limit-exit}, which holds in a very general setting:
 see \cite[Proposition V.1]{livrevert}. Coming back to the definition of $W^\bullet_s$ and $L^\bullet_s$ in terms
 of $W_s$, we see that we have
 $$\mathcal{X}_r=\lim_{\ve\to 0} \frac{1}{\ve} \int_0^\sigma \mathrm{d}s\,\mathbf{1}_{
 \{\zeta_s-\ve< \tau_{-r}(W_s)<\zeta_s\}} = \z_{-r},$$
 where the last equality follows from \eqref{limit-exit}. This simple remark
 allows us to identify the process $(\mathcal{X}_r)_{r>0}$ with the exit measure process $(Z_r)_{r>0}$, and justifies
 the observations preceding Proposition \ref{identi-jump} in the introduction. Proposition \ref{identi-jump}
 itself then follows from Propositions \ref{discont} and \ref{identi-jump2}. Finally, Theorem \ref{main-intro}
 is a consequence of Theorem \ref{mainT} and the fact that the excursions $\tilde W^{*(v_i)}$
 can be written in the form $\xi_i\,\tilde W^{(v_i)}$, for $i=1,2,\ldots$.

\end{document}